\documentclass[11pt]{article}

\usepackage[margin=1in]{geometry}
\usepackage{amsmath,amssymb,amsthm,mathtools}
\usepackage{array,booktabs}
\usepackage{microtype}
\usepackage[colorlinks=true,linkcolor=blue,citecolor=blue,urlcolor=blue]{hyperref}

\newtheorem{theorem}{Theorem}[section]
\newtheorem{proposition}[theorem]{Proposition}
\newtheorem{lemma}[theorem]{Lemma}
\newtheorem{corollary}[theorem]{Corollary}
\theoremstyle{definition}
\newtheorem{definition}[theorem]{Definition}
\theoremstyle{remark}
\newtheorem{remark}[theorem]{Remark}

\newcommand{\E}{\mathbb E}
\newcommand{\Pp}{\mathbb P}
\newcommand{\Qq}{\mathbb Q}
\newcommand{\Var}{\operatorname{Var}}
\newcommand{\Cov}{\operatorname{Cov}}
\newcommand{\Gap}{\operatorname{Gap}}
\newcommand{\Om}{\Omega}
\newcommand{\one}{\mathbf 1}
\newcommand{\cD}{\mathcal D}
\newcommand{\cE}{\mathcal E}
\newcommand{\cA}{\mathcal A}
\newcommand{\cB}{\mathcal B}
\newcommand{\tP}{\widetilde P}

\newcommand{\ip}[2]{\langle #1,#2\rangle}

\numberwithin{equation}{section}

\title{Poincar\'e's Inequality on Sliding-Window Occupation Count Kernels and the Optimal Normalized Spectral-Gap Constant }
\author{%
  Yanjin Xiang and Zhihua Zhang\\
  School of Mathematical Sciences, Peking University\\
  {\small\texttt{yjxiang@stu.pku.edu.cn; zhzhang@math.pku.edu.cn}}
}
\date{\today}
\hypersetup{
  pdftitle={Poincar\'e's Inequality on Sliding-Window Occupation Count Kernels and the Optimal Spectral-Gap Comparison Constant},
  pdfsubject={Optimal state-space order for projected occupation-count kernels},
  pdfkeywords={reversible Markov chain, spectral gap, sliding window,
    occupation count, projected kernel, Poincare inequality}
}

\begin{document}
\maketitle

\begin{abstract}
Let $P$ be an irreducible reversible Markov kernel on an $m$-state space $\Omega$, and denote its right spectral gap by $\Gap(P) =1-\lambda_2(P)$.  For a stationary trajectory,
let $K_t$ be the occupation-count vector of the length-$n$ window beginning
at time $t$.  The stationary pair $(K_0,K_1)$ defines a reversible projected
count kernel $\widetilde P_n$.  For each fixed $m\ge2$, define
\[
 c_m^\star=
 \inf_{\substack{|\Omega|=m,\ P\ \mathrm{irreducible\ and\ reversible}\\
                  n\ge2}}
 \frac{n\Gap(\widetilde P_n)}{\Gap(P)}.
\]
We prove
\[
 \frac1{1080m}\le c_m^\star\le q_{m-2},
 \qquad
 q_0=\frac14,\quad q_{r+1}=q_r(1-q_r).
\]
We also show $q_{m-2}=(m+\log m+O(1))^{-1}$, which implies that
$c_m^\star=\Theta(m^{-1})$.  Thus, the reciprocal optimal comparison coefficient
$C_m^{\mathrm{opt}}=(c_m^\star)^{-1}$ is of the order $m$.  In the two-state case,
we determine the exact values $c_2^\star=1/4$ and $C_2^{\mathrm{opt}}=4$.
The lower bound also holds for
$n=1$ and is uniform in $P$, including sparse and periodic kernels. The
proof combines short-window decorrelation with an averaged anchor-excursion
decomposition, a Green-kernel hitting estimate, and a stopped
Carleson--Hardy inequality.  A nested rare-state construction produces
finite $m$-state witnesses whose normalized Rayleigh quotients approach
$q_{m-2}$ through an ordered sequence of limits.  For every fixed finite
irreducible reversible aperiodic kernel on at least two states,
$\Gap(\widetilde P_n)=\Theta_P(n^{-1})$.
\end{abstract}



\tableofcontents

\section{Introduction}
\label{sec:introduction}

Poincar\'e inequalities associate the variance of a function to its Dirichlet
energy and provide a common framework for spectral-gap and concentration
estimates~\cite{BakryGentilLedoux2014,Ledoux2001}.  In the continuous setting,
the Poincar\'e constant $C_P(\mu)$
of a probability measure $\mu$ on $\mathbb R^d$ is the least constant such
that
\[
 \Var_\mu(f)\le C_P(\mu)\int_{\mathbb R^d}|\nabla f|^2\,d\mu
\]
for all smooth functions with finite variance and energy.  For smooth
potentials $V$ with positive definite Hessian, the Brascamp--Lieb
inequality~\cite{BrascampLieb1976} bounds variance under the law proportional
to $e^{-V}$ by a gradient energy weighted by $(\nabla^2V)^{-1}$.
For log-concave
measures, the Kannan--Lov\'asz--Simonovits (KLS) conjecture~\cite{KLS1995}
asserts that $C_P(\mu)\le C\|\Cov(\mu)\|_{\mathrm{op}}$ for a universal
constant $C$; equivalently, linear functions determine the Poincar\'e
constant up to a universal factor.  The broader connection between
log-concavity, isoperimetry, and functional inequalities is developed in
Bobkov~\cite{Bobkov1999}; Milman~\cite{Milman2009} establishes quantitative
relations between isoperimetry, spectral gaps, and concentration under
convexity assumptions.

Eldan's stochastic localization scheme~\cite{Eldan2013} connects thin-shell
estimates with the KLS problem.  Further developments include the
polynomial bounds of Lee and Vempala~\cite{LeeVempala2024}, the
subpolynomial bound of Chen~\cite{Chen2021}, the polylogarithmic bounds of
Klartag and Lehec~\cite{KlartagLehec2022}, and the bound
$C_P(\mu)\le C\log d$ for isotropic log-concave measures in dimension
$d\ge2$ due to Klartag~\cite{Klartag2023}.  See also
Klartag and Lehec~\cite{KlartagLehec2025Survey} for a survey of this
development and its geometric context.

A related question concerns logarithmic Sobolev inequalities, which
strengthen Poincar\'e inequalities and yield Gaussian concentration.
The Gaussian case was established by Gross~\cite{Gross1975}, and Bakry
and \'Emery~\cite{BakryEmery1985} developed curvature criteria for diffusion
semigroups.
Classical estimates under uniform convexity of the potential are treated,
for example, by Bobkov and Ledoux~\cite{BobkovLedoux2000}.
Bobkov and G\"otze~\cite{BobkovGotze1999} relate sub-Gaussian
exponential-moment bounds for Lipschitz observables to transportation
inequalities.  Under a different
assumption, Bizeul~\cite{Bizeul2026} obtains improved dimension-dependent
logarithmic Sobolev estimates for log-concave measures whose
one-dimensional marginals satisfy a uniform sub-Gaussian bound.

These developments motivate the study of analogous Poincar\'e comparison
problems in discrete settings with dependence.  We consider one such
problem for occupation counts of finite reversible Markov chains: the
Dirichlet form is determined by shifting a stationary observation window
by one step, and the reference scale is set by the window length and the
spectral gap of the underlying chain.  The resulting count distribution
is not assumed to be log-concave.  Our question is how the optimal
comparison constant depends on the number of states once this reference
scale has been factored out.

In particular, let $P$ be an irreducible reversible Markov kernel on a finite $m$-state space
$\Om=\{1, 2, \ldots, m\}$, and $\pi$ be its stationary law.  We work with a
stationary two-sided realization $(X_t)_{t\in\mathbb Z}$ in $\Omega$. We  encode the $i$-th state into the $i$-th unit coordinate vector $e_i$ in $\mathbb{R}^m$. The corresponding length-$n$
occupation-count vector is defined as
\[
 K_t :=\sum_{j=0}^{n-1}e_{X_{t+j}}.
\]
The ordered
window $(X_t,\ldots,X_{t+n-1})$ contains its transition history, while
$K_t$ remembers only how often each state occurs. Sliding the window one
step removes $e_{X_t}$ and  adds $e_{X_{t+n}}$.  This produces a stationary
process on the finite count simplex
\[
 \mathsf S_{m,n} :=\left\{k\in\mathbb Z_+^m:
                  \sum_{x\in\Om}k_x=n\right\}.
\]

The count process is generally not Markov, because the present histogram does not
determine the order of the hidden path and hence does not determine the law
of the departing state.  We study its stationary one-step projection
instead. Assuming $\mu_n$ is the law of $K_0$, we define $\widetilde P_n$ by
\[
\widetilde P_n(k,\ell) := \Pp(K_1=\ell| K_0=k) 
 = \frac{\Pp(K_0=k,K_1=\ell)}{\mu_n(k)}.
\]
It is a Markov kernel on the support of $\mu_n$.  Reversal of a stationary
path interchanges $K_0$ and $K_1$, so reversibility of the original chain
implies reversibility of $\widetilde P_n$.  The distinction
between a one-step projected kernel and a strongly lumpable projection is
classical; see Kemeny and Snell~\cite{KemenySnell}.

More precisely, we ask: how much relaxation is lost
when an ordered Markov path is compressed to its occupation counts?  The
natural comparison scale is $\Gap(P)/n$, but the compression from ordered
paths to counts is both severe and nonlocal.  When $m=2$ the
count space is one-dimensional and $\widetilde P_n$ is a birth--death
kernel; its stationary law is a Markov-binomial distribution, whose shape
can already be nontrivial~\cite{DekkingKong2011}, and discrete Hardy criteria
are available; compare Miclo~\cite{Miclo1999}.  For $m\ge3$ the support lies in a higher-dimensional
simplex, and neither the birth--death order nor its one-dimensional
conductance formulas survive.  Estimates based on strict positivity or a
one-step contraction coefficient also deteriorate as transition
probabilities approach zero, even when the spectral gap of $P$ remains
bounded away from zero.

We seek estimates that are uniform over the entire class of irreducible
reversible kernels on a fixed number of states.  This includes zero
transition probabilities, eigenvalues close to $-1$, and stationary masses
that approach the boundary of the probability simplex.  We show that,
after normalization by $n/\Gap(P)$ in the Poincar\'e inequality, the
optimal comparison constant grows linearly with $m$.

\subsection{Notation and conventions}

Throughout this paper, let $\mathbb Z_+=\{0,1,2,\ldots\}$ and $\one_E$
denote the indicator of an event $E$.  We write
$\Pp$ and $\E$ for probability and expectation under the stationary
two-sided chain, and $\Pp_x$ and $\E_x$ for the conditional law and
expectation given $X_0=x$; on nonnegative times this is the law of the chain
started from $x$.  Subscripts on variances have an analogous meaning.  For
functions $f,g\colon \Om\to\mathbb R$, define
\[
\pi(f)=\sum_{x\in \Om}\pi_x f(x), \quad \ip{f}{g}_\pi=\sum_{x\in\Om}\pi_xf(x)g(x),
 \quad
 \Var_\pi(f)=\ip{f-\pi(f)}{f-\pi(f)}_\pi.
\]
Denote the eigenvalues of $P$ in $L^2(\pi)$ as
\[
 1=\lambda_1>\lambda_2\ge\cdots\ge\lambda_m\ge-1,
 \qquad \gamma=1-\lambda_2.
\]
Here and later, ``$\Gap(P):=\gamma$'' means this right spectral gap; no absolute-gap or
laziness assumption is imposed.

For a fixed window length $n$, let $\mu_n=\operatorname{Law}(K_0)$ and
$\mathsf S_n=\operatorname{supp}(\mu_n)$.  Count functions are defined on $\mathsf S_n$ by default.  For such a function $F$, we write
\[
 H_t=F(K_t),\qquad
 \Delta_t=H_{t+1}-H_t,\qquad
 \cD_t(F)=\frac12\E(H_t-H_0)^2\quad(t\ge1).
\]
Thus, $\cD_1(F)$ is the Dirichlet form of $\widetilde P_n$.  All the state spaces,
variances, Dirichlet forms, and spectral gaps associated with
$\widetilde P_n$ are restricted to $\mathsf S_n$; a singleton support is
set to the gap $1$.  In addition, 
\[
 T_v=\inf\{t\ge 0 \colon X_t=v\},
 \qquad
 \tau_v^+=\inf\{t\ge 1 \colon X_t=v\},
\]
denote respectively the entrance and strict return times. Our spectral and hitting-time conventions follow Levin, Peres, and Wilmer~\cite{LPW}.  

\subsection{Main results}

To compare the sliding-window count kernel with the original kernel, we
normalize its spectral gap by the natural scale $\Gap(P)/n$ and keep the
number of states fixed.

\begin{definition}[Optimal normalized spectral-gap constant]
\label{def:gap-ratio}
For each integer $m\ge2$, define
\begin{equation}
 c_m^\star:=
 \inf_{\substack{|\Om|=m,\ P\ \mathrm{irreducible\ and\ reversible}\\
                  n\ge2}}
 \frac{n\Gap(\widetilde P_n)}{\Gap(P)}.
 \label{eq:best-cm}
\end{equation}
The infimum ranges over all irreducible reversible kernels on exactly
$m$ states and all integer window lengths $n\ge2$.  Its reciprocal is the
optimal comparison constant
\begin{equation}
 C_m^{\mathrm{opt}}:=(c_m^\star)^{-1}.
 \label{eq:best-Cm}
\end{equation}
\end{definition}

Theorem~\ref{thm:intro-main} below shows that $c_m^\star>0$.  By the
variational characterization of the spectral gap, every such $P$ and
$n\ge2$ satisfy Poincar\'e's inequality
\begin{equation}
 \Var_{\mu_n}(F)
 \le C_m^{\mathrm{opt}}\frac{n}{\Gap(P)}\cD_1(F)
 =\frac{n}{c_m^\star\Gap(P)}\cD_1(F)
 \label{eq:poincare-inequality}
\end{equation}
for every real function $F$ on $\mathsf S_n$.  Equivalently,
$\Gap(\widetilde P_n)\ge c_m^\star\Gap(P)/n$.
The following theorem determines the optimal state-space order of these
comparison quantities.

\begin{theorem}[Optimal state-space order]
\label{thm:intro-main}
For every integer $m\ge2$, the optimal normalized spectral-gap constant $c_m^\star$ satisfies
\begin{equation*}
 \frac1{1080m}\le c_m^\star\le q_{m-2},
\end{equation*}
where
\begin{equation}
 q_0=\frac14,
 \qquad
 q_{r+1}=q_r(1-q_r)\quad(r\ge0).
 \label{eq:intro-q-recurrence}
\end{equation}
Moreover,
\begin{equation}
 q_{m-2}\le\frac1{m+2},
 \qquad
 q_{m-2}=\frac1{m+\log m+O(1)}\quad(m\to\infty).
 \label{eq:intro-q-asymptotic}
\end{equation}
Consequently,
\begin{equation}
 c_m^\star=\Theta(m^{-1}),
 \qquad
 C_m^{\mathrm{opt}}=\Theta(m).
 \label{eq:intro-sharp-order}
\end{equation}
\end{theorem}

The elementary bound in \eqref{eq:intro-q-asymptotic} follows from
\[
 \frac1{q_{r+1}}
 =\frac1{q_r}+\frac1{1-q_r}
 \ge\frac1{q_r}+1,
\]
which gives $q_r^{-1}\ge r+4$ by induction.

The lower part of Theorem~\ref{thm:intro-main} shows that for every irreducible
reversible kernel $P$ on $m$ states and for every $n\ge2$,
\begin{equation}
 \Gap(\widetilde P_n)
 \ge\frac1{1080m}\frac{\Gap(P)}n.
 \label{eq:intro-main}
\end{equation}
The case $n=1$ is not
included in the definition of $c_m^\star$, but
the same estimate holds for $n=1$. In particular, because the map $x\mapsto e_x$ is a bijective
encoding,  $\widetilde P_1$ is isomorphic to $P$.  

\begin{remark}[Degenerate and periodic windows]
\label{rem:degenerate-periodic}
If the support of $\mu_n$ is a singleton, we use the convention
$\Gap(\widetilde P_n)=1$, as in Section~\ref{sec:preliminaries}.  For the
deterministic two-state alternating chain, an even window has singleton
count support, while an odd window alternates between two count vectors
and its projected kernel has right gap $2$.  Thus, periodic parity
degeneracies are compatible with the uniform lower bound but need not have
the generic $n^{-1}$ fixed-kernel relaxation scale.
\end{remark}

The bound is uniform in both $P$ and $n$, including sparse and periodic
kernels.  The nested rare-state construction produces finite $m$-state
kernels and windows whose normalized Rayleigh quotients approach
$q_{m-2}$, showing that the optimal comparison constant
$C_m^{\mathrm{opt}}$ cannot be $o(m)$.
For a fixed aperiodic kernel, the $n^{-1}$ order in
\eqref{eq:intro-main} is sharp.  Indeed, applying the count kernel to
a linear statistic associated with a nonconstant eigenfunction gives an
$O(n^{-1})$ upper bound with a constant depending on $P$ (see
Subsection~\ref{sec:upper}).  Periodic boundary
cases can relax faster, so the matching upper statement is kept separate
from the uniform lower theorem.

As we see, the above lower bound  has a comparison cost linear in the number of
states.  We now show that this order cannot be improved.  The witness is a
reversible chain with one dominant root and a nested sequence of rare states.
Successive rare levels produce the scalar recursion
\(q\mapsto q(1-q)\).

\begin{theorem}[Nested rare-state obstruction]
\label{thm:exact-cardinality-upper}
Let
\begin{equation}
 q_0=\frac14,
 \qquad
 q_{r+1}=q_r(1-q_r),\quad r\ge0.
 \label{eq:nested-q-recursion}
\end{equation}
Then, for every integer \(m\ge2\),
\begin{equation}
 {c_m^\star\le q_{m-2}.}
 \label{eq:exact-cardinality-upper}
\end{equation}
More precisely, for every \(\delta>0\) there is a finite irreducible
reversible kernel \(P\) on exactly \(m\) states and a finite window length
\(n\ge2\) such that
\begin{equation}
 \frac{n\Gap(\widetilde P_n)}{\Gap(P)}
 \le q_{m-2}+\delta.
 \label{eq:nested-finite-witness}
\end{equation}
The kernels may also be chosen entrywise positive.  Furthermore, we obtain
\begin{equation}
 q_r=\frac1{r+\log r+O(1)}\quad(r\to\infty),
 \qquad
 q_{m-2}\le\frac1{m+2}.
 \label{eq:nested-q-asymptotic}
\end{equation}
Consequently, \(C_m^{\mathrm{opt}}\ge q_{m-2}^{-1}\). Thus, a linear
comparison cost is necessary.
\end{theorem}

For the two-state case the comparison constant can be determined exactly.

\begin{theorem}[Exact two-state constant]
\label{thm:intro-binary}
For every irreducible two-state Markov kernel $P$ and every $n\ge2$,
\begin{equation}
 \Gap(\widetilde P_n)\ge\frac{\Gap(P)}{4n}.
 \label{eq:intro-binary-bound}
\end{equation}
The coefficient $1/4$ is optimal uniformly over $P$ and $n$.  Thus,
\begin{equation}
 c_2^\star=\frac14,
 \qquad C_2^{\mathrm{opt}}=4.
 \label{eq:intro-binary-exact}
\end{equation}
\end{theorem}

Every irreducible two-state kernel is reversible.  The complete proof of
Theorem~\ref{thm:intro-binary}, including zero diagonal entries and the
periodic boundary, is given in Appendix~\ref{app:binary:section}.
It combines a refresh-block comparison when the nontrivial base
eigenvalue is nonnegative with an edgewise Stein comparison when it is
negative.  Sharpness follows from the two-state instance of Theorem~\ref{thm:exact-cardinality-upper}.  
The corresponding exact-value
question for $m\ge3$ remains open.

\subsection{Proof sketch of Theorem~\ref{thm:intro-main}}

We first prove the uniform lower bound.  Recall that
$\cD_r(F)=\frac12\E(H_r-H_0)^2$, where $H_t=F(K_t)$.  If
$n\gamma\le m$, the adjacent-block estimate (see
Lemma~\ref{lem:block-separation}), followed by the telescoping inequality (see
Lemma~\ref{lem:window-telescoping}), yields
\[
 \cD_1(F)
 \ge \frac{c_0}{m}\frac{\gamma}{n}\Var_{\mu_n}(F),
 \qquad c_0=\frac{1-e^{-1}}2.
\]
This is stronger than the claimed lower bound.

Suppose now that $n\gamma>m$. Let
\[
 \mathsf A=\left\{v \in \Omega \colon  \pi_v\ge\frac1{2m}\right\},
 \qquad \alpha=\pi(\mathsf A)>\frac12,
 \qquad w_v=\frac{\pi_v}{\alpha}.
\]
For every $v\in\mathsf A$, the Green-kernel and interlacing argument in
Lemma~\ref{lem:hitting} shows that
\[
 \mathsf H_v:=\max_x\E_xT_v
 <\mathsf H_*:=\frac{3m}{\gamma}.
\]
We split the window law according to whether the window avoids $v$ or contains
$v$.  Corollary~\ref{cor:no-anchor} controls the first sector, while
Lemma~\ref{lem:first-anchor} transports the second to a length-$n$ window
started at $v$.  Both estimates use the same center $c_v$,  yielding
\[
 \Var_{\mu_n}(F)
 \le\{64\mathsf H_v^2+(8n-4)\mathsf H_v\}\cD_1(F)
     +2p_vV_v^{\mathrm{end}}.
\]
Before estimating the endpoint term, we make averaging the above inequality with the weights $w_v$. This  prevents a loss of order $\pi_v^{-1}$.

To control the endpoint variances, decompose the rooted window into its
i.i.d.\ return excursions from $v$.  Common-deletion Efron--Stein splits
the deletion cost into complete excursions and the unique terminal partial
excursion.  For a complete excursion, cyclically moving the deleted block
to the front produces an exact renewal multiplicity.  After reversal, this
becomes an initially known decreasing weight, and a stopped weighted Hardy
inequality converts the deletion difference into physical window
increments.  Palm inversion and averaging over $v$ allocate these
increments to future visits.  For the terminal partial excursion, an
independent reversed Palm past supplies the deleted tail, while a last-visit
decomposition removes the remaining multiplicity.
Proposition~\ref{prop:endpoint-closure} consequently gives
\[
 \sum_{v\in\mathsf A}w_vV_v^{\mathrm{end}}
 \le(64n\mathsf H_*+48n)\cD_1(F).
\]
Combining the estimates and with $n\gamma>m$ leads to 
\[
 \Var_{\mu_n}(F)
 \le\{64\mathsf H_*^2+136n\mathsf H_*+96n\}\cD_1(F)
 \le1080\,\frac{mn}{\gamma}\cD_1(F).
\]
The three contributions to the last constant are 
$576$, $408$, and $96$, respectively. Consequently, the variational formula proves
$c_m^\star\ge(1080m)^{-1}$.

For the complementary upper bound,
Theorem~\ref{thm:exact-cardinality-upper} attaches $m{-}1$ nested rare levels
to a
dominant root. In the successive singular limits, product tests reduce the
normalized Rayleigh quotient to
\[
 \Phi_r(y_1,\ldots,y_r)
 =y_1^2+(1-2y_1)\Phi_{r-1}(y_2,\ldots,y_r).
\]
The optimization  derives $q_r=q_{r-1}(1-q_{r-1})$.  A reverse-order diagonal
choice then produces finite kernels and finite windows whose quotients
approach $q_{m-2}$, proving $c_m^\star\le q_{m-2}$.  Finally, because
$q_r^{-1}=r+\log r+O(1)$, we complete the claimed state-space order.

\subsection{Related work and explanation}

In an earlier work~\cite{XiangXinZhang2026}, the same stationary one-step
count kernel was studied under a strict positivity assumption, that is, 
each element $p(x, y)$ of $P$ should be strictly positive.  That work proved a
$P$-dependent lower bound of order $n^{-1}$ and established the
linear-statistic upper bound recalled in Subsection~\ref{sec:upper}.
Our current theorem does not require this condition.
Moreover, it determines the sharp state-space order of the comparison:
$c_m^\star=\Theta(m^{-1})$, or equivalently
$C_m^{\mathrm{opt}}=\Theta(m)$.

The projected count kernel is related to Markov-chain aggregation.  Strong
lumpability makes a projected process Markov for every initial law, while
ordinary and exact lumpability preserve weaker stationary or transient
information; see Kemeny and Snell~\cite{KemenySnell} and
Buchholz~\cite{Buchholz1994}.  Occupation counts do not generally satisfy
these conditions.  The stationary-edge definition used here instead retains
the one-step edge measure and produces a reversible kernel even when the full
projected process has memory.  Thus, classical lumpability criteria 
clarify the object but do not provide its Poincar\'e constant.
For decomposable Markov chains, Jerrum, Son, Tetali, and
Vigoda~\cite{JerrumSonTetaliVigoda2004} bound Poincar\'e and log-Sobolev
constants using projection and restriction chains together with a
parameter controlling transitions between blocks.  This gives another
framework for relating the spectral behavior of a chain to that of a
coarse description.

Occupation counts are also finite-time empirical measures.  Long-time
deviations of empirical measures go back to the
Donsker--Varadhan theory~\cite{DonskerVaradhan1975}; finite-state
multivariate empirical measures are treated, for example, by
Ellis~\cite{Ellis1988}.  Spectral and coupling methods also yield
concentration and variance bounds for empirical averages along the original
trajectory~\cite{Paulin2015}.  Our question is different in two respects:
$n$ is finite, and the quantity estimated is the spectral gap of the
projected count kernel rather than a tail probability under the original
path law.

Functional inequalities for local dynamics provide another comparison.
Cheeger-type inequalities for Markov chains~\cite{LawlerSokal1988} and
path-based estimates~\cite{DiaconisStroock1991} relate spectral gaps to
conductance and the geometry of transitions.
Diaconis and Saloff-Coste~\cite{DiaconisSaloffCoste1996} develop logarithmic
Sobolev methods for convergence to equilibrium in finite Markov chains,
complementing spectral-gap estimates.
Reversible-chain comparison theorems~\cite{DiaconisSaloffCoste1993}
and martingale decompositions for Glauber or Kawasaki dynamics, such as
those of Lu and Yau~\cite{LuYau1993}, control chains whose local update is
specified directly.  Under the Dobrushin uniqueness condition,
Wu~\cite{Wu2006} obtains explicit Poincar\'e bounds for Gibbs measures in
terms of the Dirichlet form of single-site conditional resampling.
Here an update removes one hidden endpoint, adds the
other, and then averages over all path orders compatible with the observed
count.  The resulting kernel is neither a path-space heat bath nor a
standard conservative spin dynamics.  In particular, comparisons based on
uniformly positive local conditional probabilities degenerate at sparse
kernels, while Theorem~\ref{thm:intro-main} remains uniform for fixed
$m$.

Recent work also relates discrete functional inequalities to structural
properties of the invariant law.  Anari, Liu, and Oveis
Gharan~\cite{AnariLiuOveisGharan2024} derive Glauber spectral-gap estimates
from spectral independence.  Cryan, Guo, and
Mousa~\cite{CryanGuoMousa2021} prove modified log-Sobolev inequalities for
bases-exchange walks under homogeneous strong log-concavity, while Hermon
and Salez~\cite{HermonSalez2023} use the stochastic covering property for
local walks.  For ordered configurations with fixed color counts,
Salez~\cite{Salez2021} determines the multislice log-Sobolev constant up to
universal factors.  We impose none of the above log-concavity or dependence
conditions on $\mu_n$, and our Dirichlet form is defined by the stationary
window shift.

When $m=2$, the projected count kernel is birth--death and discrete Hardy
criteria are available~\cite{Miclo1999}.  The multidimensional count
simplex has no analogous total order.  The contribution here is to replace
that order by an averaged heavy-anchor excursion decomposition and to obtain
the optimal linear order of the comparison constant for arbitrary state-space
size.  The complementary nested rare-state construction shows that this
linear loss is intrinsic rather than an artifact of the lower-bound method.


\subsection{Proof Roadmap and Paper Structure}

The proof is organized as follows. Section~\ref{sec:preliminaries} collects the necessary preliminaries that will be used.  We first establish the uniform lower
bound in Theorem~\ref{thm:intro-main} in
Sections~\ref{sec:basic-estimates}--\ref{sec:gap-bounds}, combining a short-window
estimate with an averaged anchor-excursion argument for longer windows.
Subsection~\ref{sec:upper} gives the matching window-length order for fixed
aperiodic kernels. We then construct nested rare-state examples in
Section~\ref{sec:obstruction} to obtain an upper bound of the same order
in $m$ (see Theorem~\ref{thm:exact-cardinality-upper}), completing the proof of Theorem~\ref{thm:intro-main}.
Finally, the exact two-state result, Theorem~\ref{thm:intro-binary},
is proved in Appendix~\ref{app:binary:section} by sharpening the lower
bound; its optimality follows from the two-state case of the same
rare-state construction.

\section{Projected Kernels, Energies, and Palm Inversion}
\label{sec:preliminaries}

Let $P$ be an irreducible reversible kernel on a $m$-state space $\Om$,
with stationary law $\pi$.  We use a stationary two-sided realization
$(X_t)_{t\in\mathbb Z}$.  For a fixed window length $n\ge1$, set
\begin{equation*}
 K_t=\sum_{j=0}^{n-1}e_{X_{t+j}},
 \qquad H_t=F(K_t),
 \label{eq:window-def}
\end{equation*}
where $e_i$ is the $i$th unit coordinate vector in the $m$-dimensional space and $F$ is a real function on the support of $K_0$.  Let $\mu_n$ denote
the law of $K_0$.

\subsection{The projected count kernel}

Define the projected count kernel by its stationary edge measure:
\begin{equation}
 \mu_n(k)\widetilde P_n(k,\ell)
 =\Pp(K_0=k,K_1=\ell).
 \label{eq:count-kernel}
\end{equation}
The state space of $\widetilde P_n$ is always
$\mathsf S_n=\operatorname{supp}(\mu_n)$; no rows outside this support are
adjoined.  Reversal of $(X_0,\ldots,X_n)$ preserves its probability and
interchanges the count pair $(K_0,K_1)$.  Hence the right side of
\eqref{eq:count-kernel} is symmetric in $(k,\ell)$, and $\widetilde P_n$ is
reversible with stationary law $\mu_n$.

The kernel $\widetilde P_n$ is irreducible on $\mathsf S_n$.  Indeed, let
$k,\ell\in\mathsf S_n$ and choose admissible length-$n$ words with these
counts.  Irreducibility of $P$ supplies an admissible positive-probability
path from the last state of the first word to the first state of the
second.  Concatenate the first word, this connecting path, and the second
word; and slide a length-$n$ window along the concatenation.  The resulting
sequence of count vectors is a path of positive $\widetilde P_n$-edges
from $k$ to $\ell$.

For $r\ge1$, define the stationary window energy
\begin{equation}
 \cD_r(F)=\frac12\E(H_r-H_0)^2.
 \label{eq:Dr}
\end{equation}
In particular, we have
\begin{equation}
 \cD_1(F)=\cE_{\widetilde P_n}(F,F)
 =\frac12\sum_{k,\ell}\mu_n(k)\widetilde P_n(k,\ell)
                    (F(\ell)-F(k))^2.
 \label{eq:D1-dirichlet}
\end{equation}
If the count support contains at least two points, the variational formula is
\begin{equation}
 \Gap(\widetilde P_n)
 =\inf_{\Var_{\mu_n}(F)>0}
   \frac{\cD_1(F)}{\Var_{\mu_n}(F)}.
 \label{eq:gap-variational}
\end{equation}
We assign the spectral gap $1$ to a one-point kernel.  This convention only removes
vacuous cases. Finally,
we denote 
\[
 \Delta_t=H_{t+1}-H_t.
\]
Throughout this paper, $\Delta_t$ denotes an actual one-step shift of the length-$n$
window.  Although later arguments rearrange or delete return excursions,
every energy estimate is expressed in these increments.

\subsection{Return cycles and the unnormalized Kac identity}

Let $A\subseteq\Om$ be nonempty and define
\[
 \tau_A^+=\inf\{t\ge1\colon X_t\in A\}.
\]
For an observable nonnegative measurable path  $g$, the finite-state Kac
identity in the normalization used here is
\begin{equation}
 \sum_{a\in A}\pi_a\E_a
   \sum_{t=0}^{\tau_A^+-1}g\circ\theta_t
 =\E_\pi g,
 \label{eq:kac-general}
\end{equation}
where $\theta_t$ is the path shift.  No factor $1/\pi(A)$ appears because
the left side is not normalized by $\pi(A)$.

To verify \eqref{eq:kac-general}, we consider a two-sided stationary path and shift each term
on the left-hand side of \eqref{eq:kac-general} to time zero.  Its indicator says that
time $-t$ is the last visit to $A$ at or before time zero.  Exactly one
$t\ge0$ has this property almost surely, by positive recurrence of the
finite irreducible chain.  Summing the shifted terms gives
$g$ at time zero and proves \eqref{eq:kac-general}.  This verification also
shows that $g$ may depend on a finite future word. Later we will take
$g=\Delta_0^2$.

For a singleton $A=\{v\}$, denote
\[
 L=\tau_v^+,
 \qquad \Pp_v=\Pp(\,\cdot\mid X_0=v).
\]
Kac's return formula gives $\E_vL=1/\pi_v$.  Thus,
\begin{equation}
 \frac1{\E_vL}\E_v\sum_{t=0}^{L-1}g\circ\theta_t
 =\E_\pi g.
 \label{eq:palm-singleton}
\end{equation}
We refer to \eqref{eq:palm-singleton} as Palm inversion.  We will also use
it after retaining only a subset of the phases; nonnegativity then gives an
inequality.

\subsection{Reversal and conditional independence at an anchor}

Detailed balance implies the word identity
\begin{equation*}
 \pi_{x_0}\prod_{i=0}^{r-1}P(x_i,x_{i+1})
 =\pi_{x_r}\prod_{i=0}^{r-1}P(x_{i+1},x_i).
 \label{eq:word-reversal-prelim}
\end{equation*}
Consequently, the past and future of a stationary two-sided chain are
conditionally independent given $X_0$, and the reversed transition kernel
is again $P$.  Let
$\mathcal F_t^X=\sigma(X_s\colon s\le t)$ be the natural filtration of the
two-sided path.  More precisely, if $T$ is a stopping time for this
filtration, then for every bounded measurable future-path functional $G$, we have
\[
 \E\left[G(X_T,X_{T+1},\ldots)\mid\mathcal F_T^X\right]
 =\E_{X_T}G(X_0,X_1,\ldots)
 \quad\text{on }\{T<\infty\}.
\]
Thus, on $\{X_T=v\}$, the post-$T$ path has law $\Pp_v$ and is
conditionally independent of $\mathcal F_T^X$.  In later applications,
independence is always conditional on the displayed stopping-time
information.

\subsection{The endpoint-rooted count law}

For $v\in\Om$, define
\begin{equation}
 \zeta_v^{\mathrm{end}}
 =\operatorname{Law}_v\left(\sum_{j=0}^{n-1}e_{X_j}\right),
 \qquad
 c_v=\E_{\zeta_v^{\mathrm{end}}}F,
 \qquad
 V_v^{\mathrm{end}}=\Var_{\zeta_v^{\mathrm{end}}}(F).
 \label{eq:endpoint-law}
\end{equation}
The superscript emphasizes that the anchor is at an endpoint, not a
uniformly chosen occurrence of $v$ inside the window.  These two Palm laws
are generally different.  The strong Markov property yields
$\zeta_v^{\mathrm{end}}$, so no inverse occurrence multiplicity is needed.

\section{Short-window and hitting-time estimates}
\label{sec:basic-estimates}

We collect two estimates used in the proof of the uniform lower bound:
a direct bound for short and intermediate windows, and a hitting-time
bound for heavy anchors in the long-window regime.

\subsection{A uniform estimate in the short and intermediate regimes}
\label{sec:short}

This subsection proves the estimate used throughout the range
$n\gamma\le m$.  The underlying two-block estimate is dimension free and
treats the negative spectrum directly; the factor $m$ appears only when the
one-step and $n$-step window energies are compared.

\begin{lemma}[Separation of two adjacent blocks]
\label{lem:block-separation}
For every $n\ge2$ and every count function $F$, we have that
\begin{equation}
 \cD_n(F)\ge c_0\min\{1,n\gamma\}\Var(H_0),
 \qquad c_0=\frac{1-e^{-1}}2.
 \label{eq:block-separation}
\end{equation}
\end{lemma}

\begin{proof}
Subtract a constant so that $\E H_0=0$, and set
\[
 h(x)=\E[H_0\mid X_{n-1}=x].
\]
Because $H_0$ depends on the ordered block only through its occupation
count, reversing $(X_0,\ldots,X_{n-1})$ leaves $H_0$ unchanged.  Reversibility
therefore gives the second endpoint identity
\begin{equation}
 h(x)=\E[H_0\mid X_0=x].
 \label{eq:two-endpoint-h}
\end{equation}

Conditionally on the boundary edge $(X_{n-1},X_n)$, the past portion ending
at time $n-1$ and the future portion starting at time $n$ are independent.
The corresponding block conditional means are $h(X_{n-1})$ and $h(X_n)$,
respectively.  It follows that
\begin{equation}
 \E(H_0H_n)=\ip{h}{Ph}_\pi,
 \qquad
 \cD_n(F)=\Var(H_0)-\ip{h}{Ph}_\pi.
 \label{eq:endpoint-identity}
\end{equation}

Suppose first that $\lambda_2\ge0$.  Let $h_+$ be the orthogonal projection
of $h$ onto the direct sum of the nonconstant eigenspaces with eigenvalues
in $[0,\lambda_2]$, and write
\[
 a=\|h_+\|_{L^2(\pi)}^2,
 \qquad U=h_+(X_0),\qquad V=h_+(X_{n-1}).
\]
The two endpoint identities imply
\[
 \E(H_0U)=\E(H_0V)=a.
\]
Moreover,
\[
 \E(U+V)^2
 =2a+2\ip{h_+}{P^{n-1}h_+}_\pi
 \le2(1+\lambda_2^{n-1})a.
\]
Since $\E[H_0(U+V)]=2a$, Cauchy--Schwarz yields
\begin{equation}
 a\le\frac{1+\lambda_2^{n-1}}2\Var(H_0).
 \label{eq:positive-projection}
\end{equation}
Every negative spectral component of $h$ contributes a nonpositive amount
to $\ip{h}{Ph}_\pi$.  Hence
\[
 \ip{h}{Ph}_\pi
 \le\lambda_2a
 \le\frac{\lambda_2+\lambda_2^n}{2}\Var(H_0).
\]
Inserting this in \eqref{eq:endpoint-identity} gives the more precise bound
\begin{equation}
 \cD_n(F)
 \ge\frac{(1-\lambda_2)+(1-\lambda_2^n)}2\Var(H_0).
 \label{eq:block-positive-exact}
\end{equation}

If $\lambda_2<0$, all nonconstant spectral components of $h$ make a
nonpositive contribution to $\ip{h}{Ph}_\pi$, and
\eqref{eq:endpoint-identity} directly gives
$\cD_n(F)\ge\Var(H_0)$.  Finally, when $\lambda_2\ge0$, one has
$0<\gamma\le1$; writing $\lambda_2=1-\gamma$, we use
\[
 1-(1-\gamma)^n\ge1-e^{-n\gamma}
 \ge(1-e^{-1})\min\{1,n\gamma\}.
\]
Together with \eqref{eq:block-positive-exact}, this proves
\eqref{eq:block-separation}.  In the case $\lambda_2<0$, one has
$n\gamma>1$, and the already obtained bound is stronger.
\end{proof}

\begin{lemma}[Window telescoping]
\label{lem:window-telescoping}
For every $n\ge1$,
\begin{equation}
 \cD_n(F)\le n^2\cD_1(F).
 \label{eq:window-telescoping}
\end{equation}
\end{lemma}

\begin{proof}
The identity
\[
 H_n-H_0=\sum_{t=0}^{n-1}\Delta_t
\]
and Cauchy--Schwarz give
$(H_n-H_0)^2\le n\sum_{t=0}^{n-1}\Delta_t^2$.
After expectation, stationarity makes all $n$ increment energies equal.
Dividing by $2$ gives \eqref{eq:window-telescoping}.
\end{proof}

Combining Lemmas~\ref{lem:block-separation}
and~\ref{lem:window-telescoping},
\begin{equation}
 \cD_1(F)\ge
 c_0\frac{\min\{1,n\gamma\}}{n^2}\Var(H_0).
 \label{eq:short-two-regime}
\end{equation}
If $n\ge2$ and $n\gamma\le m$, then
\begin{equation}
 \cD_1(F)\ge \frac{c_0}{m}\frac\gamma n\Var(H_0).
 \label{eq:short-final}
\end{equation}
Indeed, when $n\gamma\le1$, \eqref{eq:short-two-regime} gives the
stronger coefficient $c_0\gamma/n$.  When $1\le n\gamma\le m$, it gives
$c_0/n^2=c_0(n\gamma)^{-1}(\gamma/n)$, which is at least the right side
of \eqref{eq:short-final}.

\subsection{Heavy anchors and hitting times}
\label{sec:hitting}

In the long regime, we first obtain an estimate for each anchor and then
average over a set of heavy anchors.  Define
\begin{equation}
 \mathsf A=\left\{v\in\Om:\pi_v\ge\frac1{2m}\right\},
 \qquad
 \alpha=\pi(\mathsf A),
 \qquad
 w_v=\frac{\pi_v}{\alpha}\quad(v\in\mathsf A).
 \label{eq:heavy-anchor}
\end{equation}
Every state outside $\mathsf A$ has mass strictly smaller than $1/(2m)$,
and hence
\begin{equation}
 \alpha>\frac12,
 \qquad
 \sum_{v\in\mathsf A}w_v=1.
 \label{eq:heavy-anchor-mass}
\end{equation}

For any $v\in\Om$, put $B=\Om\setminus\{v\}$ and define
\begin{equation}
 T_v=\inf\{t\ge0:X_t=v\},
 \qquad
 \mathsf H_v=\max_{x\ne v}\E_xT_v,
 \qquad
 \mathsf H_*=\frac{3m}{\gamma}.
 \label{eq:hitting-def}
\end{equation}
Here $T_v$ is an entrance time and equals zero when the chain starts at
$v$; strict return times are denoted by $\tau_v^+$ or $L$.  The displayed maximum equals the
maximum over all $x\in\Om$, since $\E_vT_v=0$.  Because $m\ge2$ and
$T_v\ge1$ when started from $x\ne v$, one also has $\mathsf H_v\ge1$.

\begin{lemma}[Green--interlacing hitting bound]
\label{lem:hitting}
For every $v\in\Om$,
\begin{equation}
 \mathsf H_v\le\frac{\pi_v^{-1}+m-2}{\gamma}.
 \label{eq:hitting-green}
\end{equation}
In particular, every $v\in\mathsf A$ satisfies
\begin{equation}
 \mathsf H_v<\mathsf H_*=\frac{3m}{\gamma}.
 \label{eq:hitting-heavy}
\end{equation}
Moreover, for every $x\in\Om$,
\begin{equation}
 \E_xT_v^2\le2\mathsf H_v^2.
 \label{eq:hitting-bound}
\end{equation}
\end{lemma}

\begin{proof}
Let $Q=P|_{B\times B}$ and $G=(I-Q)^{-1}$ be the killed Green kernel.
For $x,y\in B$, the strong Markov property at the first visit to $y$
gives
\begin{equation}
 G(x,y)=\Pp_x(T_y<T_v)G(y,y).
 \label{eq:green-factorization}
\end{equation}
Consequently,
\begin{equation}
 \E_xT_v=\sum_{y\in B}G(x,y)
 \le\sum_{y\in B}G(y,y)=\operatorname{tr}G.
 \label{eq:green-trace}
\end{equation}

Let $D_\pi=\operatorname{diag}(\pi)$ and
$S=D_\pi^{1/2}PD_\pi^{-1/2}$.  Reversibility makes $S$ symmetric, and the
principal minor obtained by deleting $v$ is
\[
 S_B=D_{\pi,B}^{1/2}QD_{\pi,B}^{-1/2}.
\]
Thus $S_B$ is similar to $Q$, and $(I-S_B)^{-1}$ is similar to
$G=(I-Q)^{-1}$; in particular their traces agree.  Write the eigenvalues
of $S_B$ as $\theta_1\ge\cdots\ge\theta_{m-1}$.

If $f(v)=0$, then Cauchy--Schwarz gives
\[
 \left(\sum_{x\ne v}\pi_xf(x)\right)^2
 \le(1-\pi_v)\sum_{x\ne v}\pi_xf(x)^2,
\]
and hence
\[
 \Var_\pi(f)\ge\pi_v\|f\|_{L^2(\pi)}^2.
\]
The Poincar\'e inequality and the Rayleigh--Ritz principle therefore give
$\theta_1\le1-\gamma\pi_v$.  Cauchy interlacing, in the convention of
decreasing eigenvalues, reads
\[
 \lambda_i(P)\ge\theta_i\ge\lambda_{i+1}(P),
 \qquad 1\le i\le m-1.
\]
Consequently
$\theta_i\le\lambda_i(P)\le\lambda_2(P)=1-\gamma$ for
$2\le i\le m-1$.
It follows that
\begin{equation}
 \operatorname{tr}G
 =\sum_{i=1}^{m-1}\frac1{1-\theta_i}
 \le\frac1{\gamma\pi_v}+\frac{m-2}{\gamma}.
 \label{eq:green-interlacing}
\end{equation}
Together with \eqref{eq:green-trace}, this proves
\eqref{eq:hitting-green}.  If $v\in\mathsf A$, then
$\pi_v^{-1}\le2m$, so \eqref{eq:hitting-heavy} follows.

For the second moment, the pathwise identity and the strong Markov
property give, on $\{s<T_v\}$,
\[
 \E_x[T_v-s\mid\mathcal F_s^X]
 =\E_{X_s}T_v\le\mathsf H_v.
\]
Consequently,
\begin{align*}
 \E_xT_v^2
 &=\E_x\sum_{s\ge0}\one_{\{s<T_v\}}\{2(T_v-s)-1\}\\
 &\le(2\mathsf H_v-1)\sum_{s\ge0}\Pp_x(s<T_v)
 =(2\mathsf H_v-1)\E_xT_v
 \le2\mathsf H_v^2.
\end{align*}
\end{proof}

\section{Windows avoiding an anchor}
\label{sec:no-anchor}

Fix any anchor $v\in\mathsf A$ from Subsection~\ref{sec:hitting}, and write
$B=\Om\setminus\{v\}$.  The anchor-free face of the count simplex is
\[
 \cA_0=\{k:k_v=0\}.
\]
In this section we suppress the dependence of $\cA_0,\xi,\kappa$, and
$\eta$ on the anchor $v$ and the window length $n$.
For $k\in\cA_0$ and $\ell\notin\cA_0$, define the oriented stationary
exit flow
\begin{equation}
 \xi(k,\ell)=\Pp(K_0=k,K_1=\ell),
 \qquad
 \kappa=\sum_{k\in\cA_0,\,\ell\notin\cA_0}\xi(k,\ell).
 \label{eq:exit-flow}
\end{equation}
When $\kappa>0$, let $\eta$ be the normalized second marginal,
\begin{equation}
 \eta(\ell)=\frac1\kappa\sum_{k\in\cA_0}\xi(k,\ell).
 \label{eq:exit-law}
\end{equation}
Thus $\eta$ is the count law immediately after an oriented exit from the
face.  If $\kappa=0$, every expression containing $\kappa\eta$ is
interpreted as zero.

The goal of this section is the centered estimate
\begin{equation}
 \E\left[\one_{\{K_{0,v}=0\}}(F(K_0)-c_v)^2\right]
 \le\{60\mathsf H_v^2+4(n-1)\mathsf H_v\}\cD_1(F),
 \label{eq:no-anchor-goal}
\end{equation}
where $c_v$ is the endpoint-rooted mean in \eqref{eq:endpoint-law}.  We
prove it in two stages.  The first is a killed-excursion estimate with an
affine boundary term; the second transports that boundary term to the same
endpoint-rooted center used in the anchor sector.

\subsection{Excursion coordinates and affine interpolation}

Work under $\Pp_v$.  Let
\begin{equation}
 L=\tau_v^+=\inf\{t\ge1:X_t=v\},
 \qquad
 X_1,\ldots,X_{L-1}\in B.
 \label{eq:return-excursion}
\end{equation}
If $L>n$, put $r=L-n$.  The $r$ consecutive windows lying entirely in $B$
start at times $1,\ldots,r$.  We label their function values and their two
adjacent boundary values by
\begin{equation}
 a_t=H_{t+1}\quad(0\le t\le r-1),
 \qquad b_-=H_0,\qquad b_+=H_{r+1}.
 \label{eq:face-phase-dictionary}
\end{equation}
Set $a_{-1}=b_-$, $a_r=b_+$, and
\begin{equation}
 d_j=a_{j+1}-a_j\quad(-1\le j\le r-1).
 \label{eq:face-increments}
\end{equation}
In terms of physical window shifts,
\begin{equation}
 d_{-1}=\Delta_0,\qquad
 d_j=\Delta_{j+1}\ (0\le j\le r-2),
 \qquad d_{r-1}=\Delta_r.
 \label{eq:face-increment-dictionary}
\end{equation}

Palm inversion gives the face identity
\begin{equation}
 \E\left[\one_{\{K_{0,v}=0\}}(F(K_0)-c)^2\right]
 =\frac1{\E_vL}\E_v\left[
   \one_{\{L>n\}}\sum_{t=0}^{r-1}(a_t-c)^2\right]
 \label{eq:face-palm}
\end{equation}
for every scalar $c$.  Each stationary no-$v$ window occurs at exactly one
of the phases in \eqref{eq:face-phase-dictionary}.

\begin{lemma}[Deterministic affine bridge]
\label{lem:deterministic-affine}
For $r\ge1$, every sequence in \eqref{eq:face-phase-dictionary} satisfies
\begin{equation}
 \sum_{t=0}^{r-1}(a_t-c)^2
 \le2(r+1)^2\sum_{j=-1}^{r-1}d_j^2
 +r\{(b_--c)^2+(b_+-c)^2\}.
 \label{eq:deterministic-affine}
\end{equation}
\end{lemma}

\begin{proof}
Let $\ell_t$ be the affine interpolation between the two boundary values,
\begin{equation}
 \ell_t=\frac{r-t}{r+1}b_-+
        \frac{t+1}{r+1}b_+,
 \qquad -1\le t\le r,
 \label{eq:affine-interpolation}
\end{equation}
and put $u_t=a_t-\ell_t$.  Then $u_{-1}=u_r=0$.  If
$s=(b_+-b_-)/(r+1)$, the increments of $u$ are $d_j-s$.  Since
\[
 u_t=\sum_{j=-1}^{t-1}(d_j-s),
\]
Cauchy--Schwarz gives
$u_t^2\le(t+1)\sum_{j=-1}^{t-1}(d_j-s)^2$.  Since
$\sum_{t=0}^{r-1}(t+1)=r(r+1)/2\le(r+1)^2$, summing over $t$ gives
\begin{align}
 \sum_{t=0}^{r-1}u_t^2
 &\le(r+1)^2\sum_{j=-1}^{r-1}(d_j-s)^2 \notag\\
 &=(r+1)^2\left\{\sum_{j=-1}^{r-1}d_j^2
          -\frac{(b_+-b_-)^2}{r+1}\right\}
 \le(r+1)^2\sum_{j=-1}^{r-1}d_j^2.
 \label{eq:zero-affine}
\end{align}
For the affine part, convexity of $x\mapsto(x-c)^2$ and summation of the
interpolation weights give
\begin{equation}
 \sum_{t=0}^{r-1}(\ell_t-c)^2
 \le\frac r2\{(b_--c)^2+(b_+-c)^2\}.
 \label{eq:affine-convexity}
\end{equation}
Apply $(x+y)^2\le2x^2+2y^2$ to
$a_t-c=u_t+(\ell_t-c)$ and combine
\eqref{eq:zero-affine}--\eqref{eq:affine-convexity}.
\end{proof}

\subsection{Controlling the random bridge length}

The factor $(r+1)^2$ in \eqref{eq:deterministic-affine} is correlated with
the window energy.  We first mark a local increment and then integrate over
the two remaining excursion segments.

\begin{lemma}[Marked-word second-moment estimate]
\label{lem:marked-word}
With the excursion notation above,
\begin{equation}
 \E_v\left[\one_{\{L>n\}}(r+1)^2
          \sum_{j=-1}^{r-1}d_j^2\right]
 \le15\mathsf H_v^2
 \E_v\left[\one_{\{L>n\}}
          \sum_{j=-1}^{r-1}d_j^2\right].
 \label{eq:marked-word}
\end{equation}
\end{lemma}

\begin{proof}
To avoid conditioning on a fixed phase, introduce the finite marked
excursion measure
\begin{equation}
 \mathfrak M_{\mathrm{int}}(G)
 =\E_v\left[\one_{\{L>n\}}
   \sum_{u=1}^{r-1}G(X_0,\ldots,X_L;u)\right].
 \label{eq:internal-marked-measure}
\end{equation}
When $r\le1$, this sum is empty.  Under this measure the marked phase is
part of the random object.  Its local word
\begin{equation}
 W_u=(X_u,X_{u+1},\ldots,X_{u+n})\in B^{n+1}.
 \label{eq:marked-local-word}
\end{equation}
determines $\Delta_u^2$.  Let
\[
 A=u,
 \qquad R=L-(u+n).
\]
These are the lengths of the two pieces outside the local word, and
$r=A+R$.  After normalizing the finite measure
$\mathfrak M_{\mathrm{int}}$ (equivalently, disintegrating it with respect
to $W_u$), conditional expectations given $W_u$ are well defined.  The two
pieces are conditionally independent.  The right piece
is a first-hitting path from $X_{u+n}$ to $v$.  Reversing the left piece by
detailed balance makes it a first-hitting path from $X_u$ to $v$.

Fix $W_u=(x_0,\ldots,x_n)$.  A compatible left segment has the form
\[
 v=z_0,z_1,\ldots,z_{A-1},z_A=x_0,
 \qquad z_1,\ldots,z_{A-1}\in B,
\]
and a compatible right segment has the form
\[
 x_n=y_0,y_1,\ldots,y_{R-1},y_R=v,
 \qquad y_1,\ldots,y_{R-1}\in B.
\]
The weight of the complete marked word is the product of its left, local,
and right transition weights.  Detailed balance gives
\begin{equation}
 \prod_{i=0}^{A-1}P(z_i,z_{i+1})
 =\frac{\pi_{x_0}}{\pi_v}
   \prod_{i=0}^{A-1}P(z_{i+1},z_i).
 \label{eq:left-word-reversal}
\end{equation}
The factor $\pi_{x_0}/\pi_v$, as well as the transition weight of the local
word, depends only on $W_u$.  Summing over all admissible $A$ and $R$ in
\eqref{eq:internal-marked-measure} therefore leaves the product of a
reversed left first-hit law and an ordinary right first-hit law.  This
proves both the asserted conditional independence and
\begin{equation}
 \E[A^2\mid W_u]\le2\mathsf H_v^2,
 \qquad
 \E[R^2\mid W_u]\le2\mathsf H_v^2.
 \label{eq:marked-excess-moments}
\end{equation}

Since $\Delta_u^2$ is measurable with respect to $W_u$ and
\[
 (r+1)^2=(A+R+1)^2\le3(A^2+R^2+1),
\]
Lemma~\ref{lem:hitting} and $\mathsf H_v\ge1$ give, under the marked
measure,
\begin{equation}
 \mathfrak M_{\mathrm{int}}\bigl((r+1)^2\Delta_u^2\bigr)
 \le15\mathsf H_v^2
       \mathfrak M_{\mathrm{int}}\bigl(\Delta_u^2\bigr).
 \label{eq:internal-mark-bound}
\end{equation}

For the entrance edge $\Delta_0$, condition on
$(X_0,\ldots,X_n)$.  On $L>n$, the excess $r=L-n$ is the hitting time of
$v$ from $X_n$.  Thus
\[
 \E[(r+1)^2\mid X_0,\ldots,X_n]
 \le2\mathsf H_v^2+2\mathsf H_v+1
 \le6\mathsf H_v^2
 \quad\text{on }\{L>n\}.
\]
The exit edge $\Delta_r$ is its
reverse-time copy.  Absorb both constants into $15$.  The integral in
\eqref{eq:internal-mark-bound} is exactly the sum over all internal marked
phases, and the entrance and exit estimates supply the two boundary
phases.  Their union is the sum in \eqref{eq:marked-word}.
\end{proof}

The selected marked edges are a subset of the phases of the return cycle.
Consequently, Palm inversion and nonnegativity give
\begin{equation}
 \frac1{\E_vL}\E_v\left[
  \one_{\{L>n\}}\sum_{j=-1}^{r-1}d_j^2\right]
 \le\E\Delta_0^2=2\cD_1(F).
 \label{eq:selected-edge-palm}
\end{equation}

\subsection{The affine boundary term}

The exit flow has the following Palm representation.  The event that a
stationary window exits the no-$v$ face is
\[
 \{X_0,\ldots,X_{n-1}\in B,\ X_n=v\}.
\]
By word reversal,
\begin{equation}
 \kappa
 =\pi_v\Pp_v(X_1,\ldots,X_n\in B)
 =\pi_v\Pp_v(L>n).
 \label{eq:kappa-palm}
\end{equation}
On $\{L>n\}$, each boundary value $b_-$ and $b_+$ has count law $\eta$;
for $b_-$ this follows by reversing the boundary word, and for $b_+$ it is
the direct exit word.

\begin{lemma}[Residual-life boundary estimate]
\label{lem:boundary-palm}
For every scalar $c$ and each sign $\sigma\in\{-,+\}$,
\begin{equation}
 \frac1{\E_vL}\E_v\left[
  \one_{\{L>n\}}r(b_\sigma-c)^2\right]
 \le\mathsf H_v\kappa\E_\eta(F-c)^2.
 \label{eq:boundary-palm}
\end{equation}
\end{lemma}

\begin{proof}
For $b_-$, condition on the boundary word $(X_0,\ldots,X_n)$.  On
$L>n$, the residual $r=L-n$ is an ordinary hitting time of $v$ from
$X_n\in B$, and hence its conditional mean is at most $\mathsf H_v$.
Therefore
\[
 \E_v[\one_{\{L>n\}}r(b_--c)^2]
 \le\mathsf H_v\E_v[\one_{\{L>n\}}(b_--c)^2].
\]
Use $(\E_vL)^{-1}=\pi_v$, \eqref{eq:kappa-palm}, and the boundary count law
$\eta$ to obtain \eqref{eq:boundary-palm}.  The $b_+$ estimate follows by
reversing the complete excursion.
\end{proof}

\begin{lemma}[Affine killed-excursion inequality]
\label{lem:affine-killed}
For every scalar $c$,
\begin{equation}
 \E\left[\one_{\{K_{0,v}=0\}}(F(K_0)-c)^2\right]
 \le60\mathsf H_v^2\cD_1(F)
 +2\mathsf H_v\kappa\E_\eta(F-c)^2.
 \label{eq:affine-killed}
\end{equation}
\end{lemma}

\begin{proof}
Insert Lemma~\ref{lem:deterministic-affine} into the Palm identity
\eqref{eq:face-palm}.  For the zero-boundary part, use
Lemma~\ref{lem:marked-word} and \eqref{eq:selected-edge-palm}; its constant
is $2\cdot15\cdot2=60$.  Apply Lemma~\ref{lem:boundary-palm} to the two
affine boundary values.  Their combined contribution is
$2\mathsf H_v\kappa\E_\eta(F-c)^2$.
\end{proof}

\subsection{From the exit law to the endpoint-rooted law}

It remains to control the affine boundary term when $c=c_v$.  Define
\[
 p_n=\Pp_v(X_1,\ldots,X_n\in B).
\]
By \eqref{eq:kappa-palm}, $\kappa=\pi_vp_n$.  On a two-sided Palm path with
$X_0=v$, if $p_n=0$, then $\kappa=0$ and the claim below is trivial.
Assume $p_n>0$ and work under the conditioned law
\begin{equation}
 \Qq=\Pp_v(\,\cdot\mid X_{-n},\ldots,X_{-1}\in B).
 \label{eq:trace-Q}
\end{equation}
The condition at time $-n$ records the letter removed from the exiting
window.  Indeed, if
\[
 E^{\mathrm{exit}}
 =\{X_0,\ldots,X_{n-1}\in B,\ X_n=v\},
\]
then $\kappa=\Pp(E^{\mathrm{exit}})$ and
$\eta=\operatorname{Law}(K_1\mid E^{\mathrm{exit}})$.  Reversing the word
on $[0,n]$ and re-rooting its endpoint at time zero gives, for every count
vector $\ell$,
\[
 \eta(\ell)
 =\Pp_v\left(
 e_v+\sum_{i=-n+1}^{-1}e_{X_i}=\ell
 \,\middle|\,
 X_{-n},\ldots,X_{-1}\in B\right).
\]
The conditioning has probability $p_n$ by reversal.  For
$0\le j\le n-1$, put
\begin{equation}
 s_j=-n+1+j,
 \qquad
 Z_j=K_{s_j}
 =e_v+\sum_{i=-n+1+j}^{-1}e_{X_i}
       +\sum_{i=1}^{j}e_{X_i}.
 \label{eq:trace-windows}
\end{equation}
Empty sums are omitted.  The preceding identity shows that $Z_0$ has law
$\eta$ under $\Qq$.  Since the conditioning concerns only the past, the
future from $X_0=v$ remains a fresh chain with law $\Pp_v$, so
\[
 \operatorname{Law}_{\Qq}(Z_{n-1})=\zeta_v^{\mathrm{end}},
 \qquad \E_{\Qq}F(Z_{n-1})=c_v.
\]
The conditioned past and fresh future are independent given the fixed
value $X_0=v$.  Hence
\begin{equation}
 \E_\eta(F-c_v)^2
 \le\E_\Qq\{F(Z_0)-F(Z_{n-1})\}^2.
 \label{eq:trace-independent}
\end{equation}

\begin{lemma}[Disjoint first-anchor trace]
\label{lem:trace-propagation}
With $c_v$ as in \eqref{eq:endpoint-law},
\begin{equation}
 \kappa\E_\eta(F-c_v)^2\le2(n-1)\cD_1(F).
 \label{eq:trace-propagation}
\end{equation}
\end{lemma}

\begin{proof}
The $Z_j$ are consecutive physical count windows, so
\begin{equation}
 \{F(Z_0)-F(Z_{n-1})\}^2
 \le(n-1)\sum_{j=0}^{n-2}
       \{F(Z_{j+1})-F(Z_j)\}^2.
 \label{eq:trace-telescope}
\end{equation}
Fix $j$ and translate the edge starting at $s_j$ to time zero.  Before
translation,
\begin{align}
 &\kappa\E_\Qq\{F(Z_{j+1})-F(Z_j)\}^2 \notag\\
 &\quad=
 \E\left[\one_{\{X_0=v\}}
          \one_{\{X_{-n},\ldots,X_{-1}\in B\}}
          \Delta_{s_j}^2\right].
 \label{eq:trace-before-shift}
\end{align}
After translation, the central anchor formerly at time zero is at
\begin{equation}
 q=-s_j=n-1-j.
 \label{eq:trace-q-index}
\end{equation}
The old avoidance interval $[-n,-1]$ becomes
$[-1-j,q-1]$.  After translation one has the explicit inclusion
\[
 \{X_q=v,\ X_{-1-j},\ldots,X_{q-1}\in B\}
 \subseteq
 \{X_1,\ldots,X_{q-1}\in B,\ X_q=v\}=E_q.
\]
Thus dropping its negative-time sites and the additional condition at time
zero gives
\begin{equation}
 \kappa\E_\Qq\{F(Z_{j+1})-F(Z_j)\}^2
 \le\E\left[\Delta_0^2\one_{E_q}\right],
 \quad
 E_q=\{X_1,\ldots,X_{q-1}\in B,\ X_q=v\}.
 \label{eq:trace-marked-edge}
\end{equation}
Because $Z_j$ begins at $-n+1+j$, one has $q=n-1-j$.  Thus, as $j$ runs
from $0$ to $n-2$, $q$ runs from $n-1$ down to $1$.

The events $E_q$, $1\le q\le n-1$, are disjoint: each specifies the first
visit to $v$ after time zero.  Thus summing \eqref{eq:trace-marked-edge}
over $j$ gives at most $\E\Delta_0^2=2\cD_1(F)$.  Combine this with
\eqref{eq:trace-independent} and \eqref{eq:trace-telescope}.
\end{proof}

\begin{corollary}[Centered no-anchor estimate]
\label{cor:no-anchor}
For every count function $F$,
\begin{equation}
 \E\left[\one_{\{K_{0,v}=0\}}(F(K_0)-c_v)^2\right]
 \le\{60\mathsf H_v^2+4(n-1)\mathsf H_v\}\cD_1(F).
 \label{eq:no-anchor}
\end{equation}
\end{corollary}

\begin{proof}
Use $c=c_v$ in Lemma~\ref{lem:affine-killed} and insert
Lemma~\ref{lem:trace-propagation}.
\end{proof}

\section{Transport to the first anchor}
\label{sec:first-anchor}

Define the entrance time
\begin{equation}
 T=\inf\{t\ge0:X_t=v\},
 \label{eq:first-anchor-T}
\end{equation}
and set
\[
 \cA_1=\{T\le n-1\}=\{K_{0,v}\ge1\},
 \qquad p_v=\Pp(\cA_1).
\]
On $\cA_1$, $K_T$ is a length-$n$ window rooted at $v$.  The next lemma
controls the energy needed to move from $K_0$ to this rooted window.

\begin{lemma}[First-anchor transport]
\label{lem:first-anchor}
One has
\begin{equation}
 \E\left[\one_{\cA_1}(H_0-H_T)^2\right]
 \le2(n\mathsf H_v+\mathsf H_v^2)\cD_1(F).
 \label{eq:first-anchor-energy}
\end{equation}
Conditionally on $\cA_1$, $K_T$ has law $\zeta_v^{\mathrm{end}}$ and its
post-hit word is independent of the pre-hit event.  Consequently,
\begin{equation}
 \E\left[\one_{\cA_1}(F(K_0)-c_v)^2\right]
 \le4(n\mathsf H_v+\mathsf H_v^2)\cD_1(F)
      +2p_vV_v^{\mathrm{end}}.
 \label{eq:first-anchor-centered}
\end{equation}
\end{lemma}

\begin{proof}
The contribution of $T=0$ is zero.  On $\{T=t\}$, $1\le t\le n-1$,
Cauchy--Schwarz gives
\begin{equation}
 (H_0-H_t)^2
 \le t\sum_{s=0}^{t-1}\Delta_s^2.
 \label{eq:first-anchor-telescope}
\end{equation}
Expanding by $t$ and $s$, translating the edge indexed by $s$ to time
zero, and setting $q=t-s$, we obtain the exact upper sum
\begin{align}
 &\E\left[\one_{\cA_1}(H_0-H_T)^2\right] \notag\\
 &\quad\le
 \sum_{q=1}^{n-1}\sum_{s=0}^{n-1-q}(s+q)
 \E\left[
  \Delta_0^2\one_{E_q^+}
  \one_{\{X_{-s},\ldots,X_{-1}\in B\}}
 \right],
 \label{eq:first-anchor-double-sum}
\end{align}
where the second avoidance indicator is one when $s=0$, and
\begin{equation}
 E_q^+=\{X_0,X_1,\ldots,X_{q-1}\in B,\ X_q=v\}.
 \label{eq:first-anchor-Eq}
\end{equation}
The events $E_q^+$ are mutually disjoint.

Let
\begin{equation}
 R^-=\inf\{r\ge1:X_{-r}=v\}.
 \label{eq:backward-hit}
\end{equation}
For fixed translated coordinates and fixed $q$, the exact admissible range
of the past multiplicity is
\begin{equation}
 0\le s\le\min\{n-1-q,R^--1\}.
 \label{eq:first-anchor-exact-range}
\end{equation}
Extending this range to $0\le s\le R^--1$ gives the deterministic estimate
\begin{align}
 \sum_{s=0}^{R^--1}(s+q)
 &=qR^-+\frac{R^-(R^--1)}2
 \le qR^-+\frac12(R^-)^2.
 \label{eq:first-anchor-multiplicity}
\end{align}

Condition on
$\mathcal G=\sigma(X_0,X_1,\ldots,X_n)$.  This sigma-field determines
$\Delta_0^2$ and every event $E_q^+$.  Given $\mathcal G$, the past depends
on the future only through $X_0$, and by reversibility its transition kernel
is $P$.  On $E_q^+$, $X_0\in B$, so $R^-$ is an ordinary entrance time to
$v$, not a return time.  Lemma~\ref{lem:hitting} therefore yields
\begin{equation}
 \E[R^-\mid\mathcal G]\le\mathsf H_v,
 \qquad
 \E[(R^-)^2\mid\mathcal G]\le2\mathsf H_v^2.
 \label{eq:backward-hit-moments}
\end{equation}
Use \eqref{eq:first-anchor-multiplicity}--\eqref{eq:backward-hit-moments}
in \eqref{eq:first-anchor-double-sum}.  The conditional multiplicity is at
most $q\mathsf H_v+\mathsf H_v^2$.  Since $q\le n-1$ and the $E_q^+$ are
disjoint,
\[
 \E\left[\one_{\cA_1}(H_0-H_T)^2\right]
 \le(n\mathsf H_v+\mathsf H_v^2)\E\Delta_0^2,
\]
which is \eqref{eq:first-anchor-energy}.

The time $T$ is a stopping time, and $X_T=v$.  The strong Markov property
says that
\[
 (X_T,X_{T+1},\ldots,X_{T+n-1})
\]
is an unconditioned chain started at $v$, independently of the event
$\{T\le n-1\}$.  Therefore $K_T$ has law
$\zeta_v^{\mathrm{end}}$ conditionally on $\cA_1$, with no renewal-length
bias.  Finally,
\[
 (H_0-c_v)^2\le2(H_0-H_T)^2+2(H_T-c_v)^2.
\]
Average, use \eqref{eq:first-anchor-energy}, and note that the second term
is $2p_vV_v^{\mathrm{end}}$.
\end{proof}

\section{An unweighted reverse-excursion inequality}
\label{sec:reverse-bridge}

Before treating the weighted endpoint term, we isolate the reverse-time
argument for complete excursions in an unweighted form.  The result is
stated for a general partition
$\Om=\cA\sqcup\cB$.  Set
\begin{equation}
 \tau_\cA^+=\inf\{t\ge1:X_t\in\cA\},
 \qquad
 T_\cA=\inf\{t\ge0:X_t\in\cA\},
 \qquad
 \mathsf H_\cA=\max_{x\in\cB}\E_xT_\cA.
 \label{eq:general-hitting}
\end{equation}
The complete excursion trace, truncated at the window length, is
\begin{equation}
 \mathfrak T_{\cA,\cB,n}(F)
 =\sum_{a\in\cA}\pi_a\E_a\left[
  \one_{\{X_1\in\cB,\ 2\le\tau_\cA^+\le n\}}
  (H_{\tau_\cA^+}-H_0)^2\right].
 \label{eq:complete-trace}
\end{equation}

\subsection{A stopped Carleson--Hardy inequality}

\begin{lemma}[Stopped Carleson--Hardy inequality]
\label{lem:stochastic-hardy}
Let $T$ be an integrable stopping time for a filtration
$(\mathcal F_k)_{k\ge0}$, and let $(V_k)_{k\ge0}$ be adapted.  Suppose that
for every $k\ge1$,
\begin{equation}
 \E[T-k\mid\mathcal F_k]\le H
 \quad\text{on }\{T\ge k\}.
 \label{eq:residual-stopping}
\end{equation}
Then
\begin{equation}
 \E(V_T-V_0)^2
 \le4(H+1)\E\sum_{k=1}^{T}(V_k-V_{k-1})^2.
 \label{eq:stochastic-hardy}
\end{equation}
\end{lemma}

\begin{proof}
Put $\nu_k=\one_{\{T\ge k\}}$.  If $S\ge1$ is any stopping time, then
on $\{S\le T\}$,
\[
 \sum_{k\ge S}\nu_k=T-S+1.
\]
The optional version of \eqref{eq:residual-stopping}, obtained by
partitioning over the values of $S$, therefore gives
\begin{equation}
 \E\left[\sum_{k\ge S}\nu_k\,\middle|\,\mathcal F_S\right]
 \le(H+1)\one_{\{S\le T\}}.
 \label{eq:carleson-tail}
\end{equation}

Fix $N\ge1$ and put
\[
 T_N=T\wedge N,
 \qquad
 \nu_k^{(N)}=\one_{\{k\le T_N\}}\quad(1\le k\le N).
\]
If $S$ is a stopping time taking values in
$\{1,\ldots,N\}\cup\{\infty\}$, then
we define both sides below to be zero on $\{S=\infty\}$, and
\begin{equation}
 \sum_{k=S}^{N}\nu_k^{(N)}
 \le\one_{\{S\le T\}}(T-S+1).
 \label{eq:finite-carleson-tail}
\end{equation}
The conditional expectation of the right side given $\mathcal F_S$ is at
most $(H+1)\one_{\{S\le T\}}$ by
\eqref{eq:carleson-tail}.

Let $G\in L^2$ and $M_k=\E[G\mid\mathcal F_k]$.  For $\lambda>0$, define
\[
 S_{\lambda,N}=\inf\{1\le k\le N:|M_k|>\lambda\},
\]
with the infimum set to infinity when the set is empty.  Whenever
$|M_k|>\lambda$ and $k\le N$, one has $k\ge S_{\lambda,N}$, so
\begin{align}
 \E\sum_{k=1}^{N}\nu_k^{(N)}\one_{\{|M_k|>\lambda\}}
 &\le\E\left[
   \one_{\{S_{\lambda,N}<\infty\}}
   \sum_{k=S_{\lambda,N}}^{N}\nu_k^{(N)}\right] \notag\\
 &\le(H+1)\Pp\left(\max_{1\le k\le N}|M_k|>\lambda\right).
 \label{eq:carleson-level}
\end{align}
Integrating $2\lambda$ over $\lambda>0$ and using Doob's $L^2$ maximal
inequality yields
\begin{equation}
 \E\sum_{k=1}^{N}\nu_k^{(N)}M_k^2
 \le(H+1)\E\max_{1\le k\le N}|M_k|^2
 \le4(H+1)\E G^2.
 \label{eq:carleson-embedding}
\end{equation}

Let $d_k=V_k-V_{k-1}$.  If the energy on the right side of
\eqref{eq:stochastic-hardy} is infinite, the assertion is immediate;
otherwise the following finite sums are square integrable.  Since
$\nu_k^{(N)}d_k$ is $\mathcal F_k$-measurable,
\begin{align}
 |\E[G(V_{T_N}-V_0)]|
 &=\left|\E\sum_{k=1}^{N}\nu_k^{(N)}d_kM_k\right| \notag\\
 &\le
 \left(\E\sum_{k=1}^{N}\nu_k^{(N)}d_k^2\right)^{1/2}
 \left(\E\sum_{k=1}^{N}\nu_k^{(N)}M_k^2\right)^{1/2} \notag\\
 &\le2\sqrt{H+1}\,\|G\|_2
 \left(\E\sum_{k=1}^{T_N}d_k^2\right)^{1/2}.
 \label{eq:hardy-duality}
\end{align}
Duality in $L^2$ proves
\[
 \E(V_{T_N}-V_0)^2
 \le4(H+1)\E\sum_{k=1}^{T_N}d_k^2.
\]
Since $T<\infty$ almost surely, $V_{T_N}\to V_T$ almost surely.  Fatou's
lemma on the left and monotone convergence on the right now give
\eqref{eq:stochastic-hardy}.
\end{proof}

\subsection{Reversing a complete excursion}

\begin{theorem}[Reverse-excursion bridge]
\label{thm:reverse-bridge}
For every nontrivial partition $\Om=\cA\sqcup\cB$,
\begin{equation}
 \mathfrak T_{\cA,\cB,n}(F)
 \le8(\mathsf H_\cA+1)\cD_1(F).
 \label{eq:reverse-bridge}
\end{equation}
\end{theorem}

\begin{proof}
Let
\[
 \kappa_\cA=\sum_{a\in\cA}\pi_aP(a,\cB).
\]
If $\kappa_\cA=0$, the trace vanishes.  Otherwise, work first under the
finite measure
\begin{equation}
 \sum_{a\in\cA}\pi_a\Pp_a(\,\cdot\,;X_1\in\cB).
 \label{eq:excursion-measure}
\end{equation}
For a complete excursion of length $\tau=\tau_\cA^+$, reverse it at its
exit and retain the actual post-exit future:
\begin{equation}
 Y_k=X_{\tau-k}\quad(0\le k\le\tau),
 \qquad
 Z_j=X_{\tau+j}\quad(j\ge0).
 \label{eq:reverse-paths}
\end{equation}
After division by $\kappa_\cA$, this defines a probability law $\Qq$.
Cylinder probabilities under $\Qq$ give the following conditional
structure.  If
\[
 y_0=b\in\cA,\quad y_1,\ldots,y_{r-1}\in\cB,\quad y_r\in\cA,
 \qquad z_0=b,
\]
then detailed balance gives
\[
 \Qq\left(
   Y_0=y_0,\ldots,Y_r=y_r,\,
   Z_1=z_1,\ldots,Z_\ell=z_\ell\right)
 =\frac{\pi_b}{\kappa_\cA}
 \prod_{h=0}^{r-1}P(y_h,y_{h+1})
 \prod_{h=0}^{\ell-1}P(z_h,z_{h+1}).
\]
Consequently, conditionally on $Y_0=b$, the $Z$-future is a
$\Pp_b$-chain and is independent of the stopped $Y$-path.  The latter has
the law of a $\Pp_b$-chain conditioned on $\{Y_1\in\cB\}$ and stopped at
its next hit of $\cA$.  Thus the conditioning affects the transition at
time zero, but not the transition kernel after the first step.

Let $T=\inf\{k\ge1:Y_k\in\cA\}$.  Reveal the entire $Z$-future at time
zero and use the stopped filtration
\begin{equation}
 \mathcal F_k=\sigma(Z_0,Z_1,\ldots;
                     Y_0,\ldots,Y_{k\wedge T}),
 \qquad k\ge0.
 \label{eq:reverse-filtration}
\end{equation}
The full return time $T=\tau$ is a stopping time.  For $0\le k\le T$, set
\begin{equation}
 C_k=\sum_{u=0}^{n-1}
 \begin{cases}
  e_{Y_{k-u}},&u\le k,\\
  e_{Z_{u-k}},&u>k,
 \end{cases}
 \qquad V_k=F(C_k).
 \label{eq:adapted-reverse-window}
\end{equation}
The variable $V_k$ is $\mathcal F_k$-measurable and its definition never
uses the unknown terminal time.  Because $F$ depends only on occupation
counts, direct inspection gives
\begin{equation}
 V_k=H_{\tau-k}\quad(0\le k\le T),
 \qquad V_0=H_\tau,\qquad V_T=H_0.
 \label{eq:reverse-window-identity}
\end{equation}
Set $V_k=V_T$ for $k>T$, as in the stopped-process convention of
Lemma~\ref{lem:random-cutoff-hardy} below.
When $k\ge n-1$, the whole count window lies in the already exposed
reversed prefix, so \eqref{eq:reverse-window-identity} remains valid even
for excursions longer than $n$.

For every $k\ge1$,
\[
 \{Y_1\in\cB\}\in\sigma(Y_0,\ldots,Y_k)\subseteq\mathcal F_k.
\]
Hence, on $\{k<T\}$, the conditioning on the first step has already been
revealed and causes no further change of kernel.  Moreover,
$Z\perp Y\mid Y_0$, so enlarging the natural $Y$-filtration by the entire
$Z$-future also leaves the conditional law of the future of $Y$
unchanged.  Therefore, on $\{T\ge k\}$,
\[
 \E_\Qq[T-k\mid\mathcal F_k]
 =\one_{\{k<T\}}\E_{Y_k}T_\cA
 \le\mathsf H_\cA.
\]
Here the right side before the inequality is zero when $k=T$.
Lemma~\ref{lem:stochastic-hardy} and
\eqref{eq:reverse-window-identity} therefore imply, after multiplying back
by $\kappa_\cA$,
\begin{align}
 &\sum_{a\in\cA}\pi_a\E_a\left[
  \one_{\{X_1\in\cB\}}(H_{\tau_\cA^+}-H_0)^2\right] \notag\\
 &\quad\le4(\mathsf H_\cA+1)
 \sum_{a\in\cA}\pi_a\E_a\left[
  \one_{\{X_1\in\cB\}}
  \sum_{t=0}^{\tau_\cA^+-1}\Delta_t^2\right].
 \label{eq:reverse-before-kac}
\end{align}

Drop $\one_{\{X_1\in\cB\}}$ from the nonnegative energy and apply the
unnormalized Kac identity \eqref{eq:kac-general} with $g=\Delta_0^2$:
\begin{align}
 \sum_{a\in\cA}\pi_a\E_a\left[
  \one_{\{X_1\in\cB\}}
  \sum_{t=0}^{\tau_\cA^+-1}\Delta_t^2\right]
 &\le\E\Delta_0^2=2\cD_1(F).
 \label{eq:reverse-kac-energy}
\end{align}
Finally, the restriction $2\le\tau_\cA^+\le n$ in
\eqref{eq:complete-trace} only removes nonnegative terms.  Combining
\eqref{eq:reverse-before-kac} and \eqref{eq:reverse-kac-energy} proves
\eqref{eq:reverse-bridge}.
\end{proof}

\begin{remark}[Relation to the weighted estimate]
\label{rem:prototype-role}
Section~\ref{sec:endpoint} uses the same reverse-time filtration with an
initially known random weight; see Lemmas~\ref{lem:random-cutoff-hardy}
and~\ref{lem:weighted-hardy}.  The present unweighted theorem is a
pedagogical prototype and is not a separate logical input to the proof of
the main theorem.
\end{remark}

\begin{remark}[Why the full return time is used]
The last interior time $T-1$ need not be a stopping time.  The adapted
construction \eqref{eq:adapted-reverse-window} instead stops at the full
return time $T$, where $V_T=H_0$.
\end{remark}

\section{Averaged endpoint variance}
\label{sec:endpoint}

The first-anchor reduction leaves the endpoint-rooted variances
$V_v^{\mathrm{end}}$.  A bound for one fixed anchor loses a factor
$1/\pi_v$.  We instead prove the deletion identity for each
$v\in\mathsf A$ and
estimate it only after averaging with the weights
$w_v=\pi_v/\alpha$ from \eqref{eq:heavy-anchor}.  This order is essential:
the future renewal multiplicity remains correlated with the nonlinear
deletion difference until it has been allocated to physical window edges.

Fix $v\in\mathsf A$.  Under $\Pp_v$, let
\begin{equation}
 0=S_0<S_1<S_2<\cdots
 \label{eq:return-times}
\end{equation}
be the successive visits to $v$.  The excursion word $E_i$ consists of the
letters on $[S_i,S_{i+1})$, and
\begin{equation}
 L_i=S_{i+1}-S_i\ge1,
 \qquad
 S_i=\sum_{r=0}^{i-1}L_r.
 \label{eq:excursion-lengths}
\end{equation}
The words $(E_i)_{i\ge0}$ are i.i.d.  Put
\begin{equation}
 Y=F\left(\sum_{t=0}^{n-1}e_{X_t}\right).
 \label{eq:endpoint-Y}
\end{equation}

For $0\le i<n$, let $Y^{(i)}$ be obtained by replacing $E_i$ with an
independent copy and reading the first $n$ letters of the new infinite
concatenation.  Let $Y^{(-i)}$ be obtained by deleting $E_i$, shifting all
later excursions left, and again reading the first $n$ letters.  Since
$L_i\ge1$, only the first $n$ excursion coordinates can affect $Y$.

\subsection{Deletion decomposition}

\begin{lemma}[Common-deletion Efron--Stein]
\label{lem:common-deletion}
One has
\begin{equation}
 V_v^{\mathrm{end}}=\Var_v(Y)
 \le2\sum_{i=0}^{n-1}\E_v(Y-Y^{(-i)})^2.
 \label{eq:common-deletion}
\end{equation}
\end{lemma}

\begin{proof}
The product-space Efron--Stein inequality~\cite{EfronStein1981} gives
\[
 \Var_v(Y)\le\frac12\sum_{i=0}^{n-1}\E_v(Y-Y^{(i)})^2.
\]
Condition on every excursion word except the word inserted at coordinate
$i$.  The original and replacement words are i.i.d., while the deletion value
is a common baseline.  Hence
\begin{equation}
 \E_v(Y-Y^{(-i)})^2
 =\E_v(Y^{(i)}-Y^{(-i)})^2.
 \label{eq:deletion-symmetry}
\end{equation}
The inequality
\[
 |Y-Y^{(i)}|^2
 \le2|Y-Y^{(-i)}|^2+2|Y^{(i)}-Y^{(-i)}|^2
\]
and \eqref{eq:deletion-symmetry} prove
\eqref{eq:common-deletion}.
\end{proof}

Before the outer factor two in \eqref{eq:common-deletion}, define
\begin{align}
 \mathcal S_v
 &=\sum_{i=0}^{n-1}\E_v\left[
   \one_{\{S_{i+1}\le n,\ L_i=1\}}(Y-Y^{(-i)})^2\right],
 \label{eq:short-sector}\\
 \mathcal L_v
 &=\sum_{i=0}^{n-1}\E_v\left[
   \one_{\{S_{i+1}\le n,\ L_i\ge2\}}(Y-Y^{(-i)})^2\right],
 \label{eq:long-sector}\\
 \mathcal T_{n,v}
 &=\sum_{i=0}^{n-1}\E_v\left[
   \one_{\{S_i<n<S_{i+1}\}}(Y-Y^{(-i)})^2\right].
 \label{eq:terminal-sector}
\end{align}
The equality case $S_{i+1}=n$ is complete.  There is at most one terminal
partial excursion, and a coordinate outside the three displayed sectors
does not change the first $n$ letters.  Thus
\begin{equation}
 V_v^{\mathrm{end}}
 \le2\{\mathcal S_v+\mathcal L_v+\mathcal T_{n,v}\}.
 \label{eq:three-sectors}
\end{equation}

Let $L=L_0$ and define the future renewal multiplicity
\begin{equation}
 M_t^{(1)}=
 \sum_{i\ge0}\one_{\{L_1+\cdots+L_i\le t\}},
 \label{eq:renewal-multiplicity}
\end{equation}
where the empty sum for $i=0$ is zero and contributes one renewal epoch.

\begin{lemma}[Complete-excursion multiplicity]
\label{lem:exact-multiplicity}
The two complete sectors satisfy
\begin{align}
 \mathcal S_v
 &=\E_v\left[
   \one_{\{L=1\}}M_{n-1}^{(1)}(H_L-H_0)^2\right],
 \label{eq:short-multiplicity}\\
 \mathcal L_v
 &=\E_v\left[
   \one_{\{2\le L\le n\}}M_{n-L}^{(1)}(H_L-H_0)^2\right].
 \label{eq:exact-multiplicity}
\end{align}
Consequently, with $\mathcal C_v=\mathcal S_v+\mathcal L_v$,
\begin{equation}
 \mathcal C_v
 =\E_v\left[
   \one_{\{L\le n\}}M_{n-L}^{(1)}(H_L-H_0)^2\right].
 \label{eq:complete-correlated}
\end{equation}
\end{lemma}

\begin{proof}
For a concatenated word $w$, let $\operatorname{cnt}_n(w)$ denote the
count vector of its first $n$ letters, and let
$\operatorname{sh}_r w$ denote the word obtained by deleting its first
$r$ letters.  Fix $i\ge0$ and let $\Phi_i$
cyclically permute the first $i+1$ excursion coordinates:
\[
 \Phi_i(E_0,E_1,\ldots,E_i,E_{i+1},\ldots)
 =(E_i,E_0,\ldots,E_{i-1},E_{i+1},\ldots).
\]
Because the excursion coordinates are i.i.d., $\Phi_i$ preserves their joint
law.  On
\[
 A_i=\{L_0+\cdots+L_i\le n\},
\]
all of the first $i+1$ blocks occur completely inside the original
length-$n$ window.  Hence, pathwise on $A_i$,
\begin{align*}
 \operatorname{cnt}_n(E_0\cdots E_iE_{i+1}\cdots)
 &=\operatorname{cnt}_n(E_iE_0\cdots E_{i-1}E_{i+1}\cdots),\\
 \operatorname{cnt}_n(E_0\cdots E_{i-1}E_{i+1}\cdots)
 &=\operatorname{cnt}_n\!\left(
   \operatorname{sh}_{L_i}(E_iE_0\cdots E_{i-1}E_{i+1}\cdots)\right).
\end{align*}
Thus, under $\Phi_i$, the pair $(Y,Y^{(-i)})$ has the same joint law as
$(H_0,H_L)$, where $L$ is the length of the marked block after the
permutation.  Consequently,
\begin{equation}
 \E_v\left[\one_{A_i}(Y-Y^{(-i)})^2\right]
 =\E_v\left[
   \one_{\{L+L_1+\cdots+L_i\le n\}}(H_L-H_0)^2\right].
 \label{eq:cyclic-deletion-pair}
\end{equation}
The same identity remains true after inserting either
$\one_{\{L_i=1\}}$ or $\one_{\{L_i\ge2\}}$ on the left, with respectively
$\one_{\{L=1\}}$ or $\one_{\{L\ge2\}}$ on the right.

Since every $L_r\ge1$, the indicator in
\eqref{eq:cyclic-deletion-pair} vanishes automatically for $i\ge n$.
Tonelli's theorem and the pathwise identity
\begin{align}
 \sum_{i\ge0}\one_{\{L+L_1+\cdots+L_i\le n\}}
 &=\one_{\{L\le n\}}
   \sum_{i\ge0}\one_{\{L_1+\cdots+L_i\le n-L\}} \notag\\
 &=\one_{\{L\le n\}}M_{n-L}^{(1)}
 \label{eq:cyclic-exact-multiplicity}
\end{align}
therefore give
\begin{align*}
 \mathcal S_v
 &=\E_v\left[
   \one_{\{L=1\}}M_{n-1}^{(1)}(H_L-H_0)^2\right],\\
 \mathcal L_v
 &=\E_v\left[
   \one_{\{2\le L\le n\}}M_{n-L}^{(1)}(H_L-H_0)^2\right].
\end{align*}
Adding the two identities proves \eqref{eq:complete-correlated}.
\end{proof}

The two factors in \eqref{eq:complete-correlated} are correlated because
the later excursions determine both of them.  The next weighted inequality
controls their product without decoupling.

\subsection{Stopped inequalities with random cutoffs}

\begin{lemma}[Initially known random cutoff]
\label{lem:random-cutoff-hardy}
Let $T\ge1$ be an integrable stopping time for
$(\mathcal F_k)_{k\ge0}$, and let $(V_k)_{k\ge0}$ be adapted.  Extend
$V_k=V_T$ after $T$, and suppose
\begin{equation}
 \E[T-k\mid\mathcal F_k]\le H
 \quad\text{on }\{T\ge k\},\qquad k\ge1.
 \label{eq:weighted-residual}
\end{equation}
If $N$ is a nonnegative integer-valued $\mathcal F_0$-measurable random
variable, then
\begin{equation}
 \E\left[\one_{\{T\le N\}}(V_T-V_0)^2\right]
 \le4(H+1)\E\sum_{k=1}^{T\wedge N}(V_k-V_{k-1})^2.
 \label{eq:random-cutoff-hardy}
\end{equation}
\end{lemma}

\begin{proof}
Put
\[
 d_k=V_k-V_{k-1},\qquad
 \nu_k=\one_{\{T\ge k\}},\qquad
 E_N=\{T\le N\}.
\]
Because $T$ is a stopping time,
$\{T\ge k\}=\{T>k-1\}\in\mathcal F_{k-1}$; because $N$ is
$\mathcal F_0$-measurable, $\{k>N\}\in\mathcal F_0$.

If
\[
 A:=\E\sum_{k=1}^{T\wedge N}d_k^2
\]
is infinite, there is nothing to prove, so assume $A<\infty$.  For a
deterministic $K\ge1$, set
\[
W_K=\sum_{k=1}^{K}\nu_kd_k=V_{T\wedge K}-V_0.
\]
On $E_N$ one has $T\le N$, and therefore
\[
 \E\big[\one_{E_N}(\nu_kd_k)^2\big]
 \le \E\big[\one_{\{k\le T\wedge N\}}d_k^2\big]\le A.
\]
It follows from the finite-sum Cauchy--Schwarz inequality that
$\one_{E_N}W_K\in L^2$, so the following $L^2$ duality is legitimate.
Let $G\in L^2$ and
\[
 M_k=\E[G\one_{E_N}\mid\mathcal F_k].
\]
The sum defining $W_K$ is finite, and $\nu_kd_k$ is
$\mathcal F_k$-measurable.  Therefore
\begin{equation}
 \E[G\one_{E_N}W_K]
 =\E\sum_{k=1}^{K}\nu_kd_kM_k.
 \label{eq:random-cutoff-dual}
\end{equation}
On $\{T\ge k,\ k>N\}$ the event $E_N$ is impossible.  Since this set is
$\mathcal F_k$-measurable,
\begin{equation}
 \nu_k\one_{\{k>N\}}M_k
 =\E[G\one_{E_N}\nu_k\one_{\{k>N\}}\mid\mathcal F_k]
 =0.
 \label{eq:cutoff-support}
\end{equation}

For every stopping time $S\ge1$, partitioning over its values in
\eqref{eq:weighted-residual} gives
\begin{equation}
 \E\left[\sum_{k\ge S}\nu_k\,\middle|\,\mathcal F_S\right]
 \le(H+1)\one_{\{S\le T\}}.
 \label{eq:weighted-carleson-tail}
\end{equation}
For $\lambda>0$, let
\[
 S_{\lambda,K}=\inf\{1\le k\le K:|M_k|>\lambda\},
\]
with value infinity if the set is empty.  Applying
\eqref{eq:weighted-carleson-tail} to $S_{\lambda,K}$ and integrating the
resulting level-set inequality gives
\[
 \E\sum_{k=1}^{K}\nu_kM_k^2
 \le(H+1)\E\max_{1\le k\le K}|M_k|^2.
\]
Doob's $L^2$ maximal inequality now yields
\begin{equation}
 \E\sum_{k=1}^{K}\nu_kM_k^2
 \le4(H+1)\E[G^2\one_{E_N}]
 \le4(H+1)\E G^2.
 \label{eq:weighted-carleson-embedding}
\end{equation}
Using \eqref{eq:cutoff-support} in
\eqref{eq:random-cutoff-dual}, followed by Cauchy--Schwarz and
\eqref{eq:weighted-carleson-embedding}, gives
\[
 \left|\E[G\one_{E_N}W_K]\right|
 \le2\sqrt{H+1}\,\|G\|_2
 \left(\E\sum_{k=1}^{T\wedge N\wedge K}d_k^2\right)^{1/2}.
\]
Duality in $L^2$ proves the finite-horizon estimate
\begin{equation}
 \E[\one_{E_N}W_K^2]
 \le4(H+1)\E\sum_{k=1}^{T\wedge N\wedge K}d_k^2.
 \label{eq:random-cutoff-finite}
\end{equation}

For $J>K$, the identical duality argument applied to $W_J-W_K$ gives
\[
 \|\one_{E_N}(W_J-W_K)\|_2^2
 \le4(H+1)\E
   \sum_{k=K+1}^{J}\one_{\{k\le T\wedge N\}}d_k^2.
\]
The right side tends to zero: the partial energies increase to $A$ by
monotone convergence, equivalently their tails tend to zero by dominated
convergence with dominating variable
$\sum_{k=1}^{T\wedge N}d_k^2\in L^1$.  Hence
$\one_{E_N}W_K$ is Cauchy in $L^2$.  Since integrability of $T$ implies
$T<\infty$ almost surely, $W_K=V_T-V_0$ eventually almost surely.  Thus
\begin{equation}
 \one_{E_N}(V_T-V_0)
 =\lim_{K\to\infty}\one_{E_N}
   \sum_{k=1}^{K}\nu_kd_k
 \quad\text{in }L^2.
 \label{eq:random-cutoff-L2}
\end{equation}
Letting $K\to\infty$ in \eqref{eq:random-cutoff-finite}, using the
$L^2$ convergence on the left and monotone convergence of the
nonnegative energies on the right, proves
\eqref{eq:random-cutoff-hardy}.
\end{proof}

\begin{lemma}[Decreasing initially known weights]
\label{lem:weighted-hardy}
Under the assumptions of Lemma~\ref{lem:random-cutoff-hardy}, let
$(a_k)_{k\ge1}$ be a nonnegative, integer-valued, nonincreasing,
finite-support sequence whose entries are all $\mathcal F_0$-measurable.
Then
\begin{equation}
 \E[a_T(V_T-V_0)^2]
 \le4(H+1)\E\sum_{k=1}^{T}a_k(V_k-V_{k-1})^2.
 \label{eq:weighted-hardy}
\end{equation}
\end{lemma}

\begin{proof}
For $r\ge1$, let
\[
 N_r=\max\{k\ge1:a_k\ge r\},
\]
with $N_r=0$ if the set is empty.  Every $N_r$ is
$\mathcal F_0$-measurable, and monotonicity gives
\begin{equation}
 a_T=\sum_{r\ge1}\one_{\{T\le N_r\}}.
 \label{eq:weight-layers}
\end{equation}
Apply Lemma~\ref{lem:random-cutoff-hardy} to every layer, sum, and use
Tonelli:
\[
 \begin{aligned}
 \E[a_T(V_T-V_0)^2]
 &\le4(H+1)\E
   \sum_{r\ge1}\sum_{k=1}^{T\wedge N_r}(V_k-V_{k-1})^2\\
 &=4(H+1)\E\sum_{k=1}^{T}a_k(V_k-V_{k-1})^2.
 \end{aligned}
\]
\end{proof}

\subsection{The complete-excursion contribution}

Fix $v\in\mathsf A$.  Let $\Qq_v$ be the push-forward of $\Pp_v$ obtained
by reversing the first-return word at its endpoint and retaining the actual
post-return future.  Under $\Qq_v$, write
\begin{equation}
 \widehat X_k=X_{L-k}\quad(0\le k\le L),
 \qquad
 Z_s=X_{L+s}\quad(s\ge0),
 \qquad T=L.
 \label{eq:weighted-reversal}
\end{equation}
Detailed balance shows that the reversed first-return word has exactly
the law of a $\Pp_v$-path stopped at $\tau_v^+$, including the self-loop
case $L=1$.  The strong Markov property at the return time shows, jointly,
that this stopped path is independent of $(Z_s)_{s\ge1}$ under $\Qq_v$;
the common value $\widehat X_0=Z_0=v$ is deterministic.
Reveal all of $Z$ at time zero and use
\begin{equation}
 \mathcal F_k=\sigma(Z_0,Z_1,\ldots;
                     \widehat X_0,\ldots,\widehat X_{k\wedge T}).
 \label{eq:weighted-filtration}
\end{equation}
Thus revealing all of $Z$ at time zero does not alter the forward
transition kernel of $\widehat X$.  For every $k\ge1$, on $\{T\ge k\}$,
\begin{equation}
 \E_{\Qq_v}[T-k\mid\mathcal F_k]
 =\one_{\{k<T\}}\E_{\widehat X_k}T_v
 \le\mathsf H_v.
 \label{eq:weighted-return-residual}
\end{equation}
The equality includes the case $k=T$, when both sides before the
inequality are zero.

For $0\le k\le T$, define the count
\begin{equation}
 C_k=\sum_{u=0}^{n-1}
 \begin{cases}
  e_{\widehat X_{k-u}},&u\le k,\\
  e_{Z_{u-k}},&u>k,
 \end{cases}
 \qquad V_k=F(C_k).
 \label{eq:weighted-reverse-window}
\end{equation}
This expression never uses the unknown value of $T$, so $V_k$ is adapted.
Because $F$ depends only on occupation counts,
\begin{equation}
 V_k=H_{L-k},\qquad
 V_0=H_L,\qquad V_T=H_0,
 \qquad
 (V_k-V_{k-1})^2=\Delta_{L-k}^2.
 \label{eq:weighted-window-identity}
\end{equation}
Set $V_k=V_T$ for $k>T$ before applying
Lemma~\ref{lem:weighted-hardy}.

Define
\begin{equation}
 a_k^{(v)}=
 \begin{cases}
 \displaystyle\sum_{s=0}^{n-k}\one_{\{Z_s=v\}},&1\le k\le n,\\
 0,&k>n.
 \end{cases}
 \label{eq:renewal-weight}
\end{equation}
This sequence is $\mathcal F_0$-measurable, integer-valued,
nonincreasing, and
\begin{equation}
 a_T^{(v)}
 =\one_{\{L\le n\}}M_{n-L}^{(1)}.
 \label{eq:terminal-renewal-weight}
\end{equation}
Lemma~\ref{lem:weighted-hardy} and
\eqref{eq:complete-correlated} therefore give
\begin{equation}
 \pi_v\mathcal C_v
 \le4(\mathsf H_v+1)B_v,
 \label{eq:weighted-complete-prep}
\end{equation}
where
\begin{equation}
 B_v=\pi_v\E_{\Qq_v}
 \sum_{k=1}^{T\wedge n}a_k^{(v)}(V_k-V_{k-1})^2.
 \label{eq:Bv-def}
\end{equation}

Undo the reversal in \eqref{eq:Bv-def}, and put $t=L-k$.  Then
\begin{equation}
 B_v=\pi_v\E_v\left[
  \sum_{t=(L-n)_+}^{L-1}
  \left(\sum_{s=0}^{n-L+t}\one_{\{X_{L+s}=v\}}\right)
  \Delta_t^2\right].
 \label{eq:Bv-cycle}
\end{equation}
Let
\begin{equation}
 R_v^+=\inf\{q\ge1:X_q=v\}
 \label{eq:strict-return}
\end{equation}
in a stationary two-sided chain.  At phase $t$ of the Palm return cycle,
$R_v^+\circ\theta_t=L-t$.  Hence Palm inversion
\eqref{eq:palm-singleton} transforms \eqref{eq:Bv-cycle} into the exact
identity
\begin{equation}
 B_v=\E\left[
  \Delta_0^2\one_{\{R_v^+\le n\}}
  \sum_{u=R_v^+}^{n}\one_{\{X_u=v\}}\right].
 \label{eq:Bv-palm}
\end{equation}
The strict return convention includes a self-loop as $R_v^+=1$ and gives,
pathwise,
\begin{equation}
 \one_{\{R_v^+\le n\}}
 \sum_{u=R_v^+}^{n}\one_{\{X_u=v\}}
 =\sum_{u=1}^{n}\one_{\{X_u=v\}}.
 \label{eq:strict-return-count}
\end{equation}
Consequently,
\begin{equation}
 \sum_{v\in\mathsf A}B_v
 \le\E\left[\Delta_0^2
       \sum_{u=1}^{n}\one_{\{X_u\in\mathsf A\}}\right]
 \le n\E\Delta_0^2
 =2n\cD_1(F).
 \label{eq:Bv-sum}
\end{equation}
Summing \eqref{eq:weighted-complete-prep}, dividing by $\alpha$, and using
\eqref{eq:Bv-sum} proves
\begin{equation}
 \sum_{v\in\mathsf A}w_v(\mathcal S_v+\mathcal L_v)
 \le\frac{8}{\alpha}n(\mathsf H_*+1)\cD_1(F)
 <16n(\mathsf H_*+1)\cD_1(F).
 \label{eq:averaged-complete}
\end{equation}

\subsection{The terminal partial-excursion contribution}

We first derive the physical marked-edge representation of the terminal
deletion.  Let
\[
 I=\min\{i\ge0:S_{i+1}>n\}.
\]
Since
\[
 \{I\le i\}=\{S_{i+1}>n\}\in\mathcal G_i,
 \qquad
 \mathcal G_i=\sigma(E_0,\ldots,E_i),
\]
the index $I$ is a stopping time for the i.i.d.-excursion filtration.  Put
\[
 s=S_I,\qquad R=n-s.
\]
Both $R$ and the retained value $Y$ are
$\mathcal G_I$-measurable, because the first $n$ letters are contained
in $E_0,\ldots,E_I$.  Write
\[
 \mathsf E_{\mathrm{term}}=\{S_I<n<S_{I+1}\}.
\]
There is exactly one terminal deletion coordinate on this event, and
therefore
\[
 \mathcal T_{n,v}
 =\E_v\left[
   \one_{\mathsf E_{\mathrm{term}}}(Y-Y^{(-I)})^2\right].
\]

Conditionally on $\mathcal G_I$, the excursions after $E_I$ form an
independent fresh Palm sequence.  Equivalently, the $R$ letters inserted
after deleting $E_I$ have the law of the first $R$ letters of a
$\Pp_v$-chain, independently of $(Y,R)$ and of the retained prefix.
Thus, conditionally on $\mathcal G_I$, if $X'$ is an independent
$\Pp_v$-chain, then
\[
 Y^{(-I)}\stackrel{d}{=}
 F\left(\sum_{a=0}^{s-1}e_{X_a}
       +\sum_{r=0}^{R-1}e_{X'_r}\right).
\]

To preserve this joint law, extend the original positive path to a
two-sided Palm chain conditional on $X_0=v$, using an independent past.
The Palm past and future are conditionally independent given $X_0=v$,
and reversibility implies that
\[
 (X_0,X_{-1},\ldots,X_{1-R})
\]
has, conditionally on the positive path and on $R$, the same count law as
the fresh length-$R$ word inserted by the deletion.  This couples the joint
pair $(Y,Y^{(-I)})$.

On $\mathsf E_{\mathrm{term}}$, one has $X_0=X_s=v$, so the retained
complete-loop background satisfies
\[
 \sum_{a=0}^{s-1}e_{X_a}=\sum_{a=1}^{s}e_{X_a}.
\]
Under the preceding joint coupling, the deleted count is therefore
\[
 \sum_{a=1-R}^{0}e_{X_a}+\sum_{a=1}^{s}e_{X_a}
 =\sum_{a=1-R}^{s}e_{X_a}=K_{1-R}.
\]
The original value is $H_0$.  Hence the coupling gives the exact identity
\begin{equation}
 \mathcal T_{n,v}
 =\E_v\left[
   \one_{\mathsf E_{\mathrm{term}}}(H_0-H_{1-R})^2\right].
 \label{eq:terminal-coupling}
\end{equation}
If $R=1$, then $H_{1-R}=H_0$, so the deletion difference is zero; if
$S_I=n$, then $R=0$ and there is no terminal sector.  Sliding from
$K_{1-R}$ to $K_0$ and applying Cauchy--Schwarz now gives
\begin{equation}
 \mathcal T_{n,v}
 \le\E_v\left[
  \one_{\{R\ge2\}}(R-1)\sum_{j=1}^{R-1}\Delta_{-j}^2\right].
 \label{eq:terminal-backward-shift}
\end{equation}
On $\mathsf E_{\mathrm{term}}$, the time $s=n-R$ is the last visit to $v$
in $[0,n]$.  Expanding \eqref{eq:terminal-backward-shift} by $R$ and passing
from $\Pp_v$ to stationarity therefore gives
\begin{align}
 \mathcal T_{n,v}
 &\le\frac1{\pi_v}\sum_{R=2}^{n}(R-1)
       \sum_{j=1}^{R-1}
       \E\bigl[\Delta_{-j}^2\one_{\{X_0=v,\,X_{n-R}=v\}}
 \times
       \one_{\{X_r\ne v\ \text{for }n-R+1\le r\le n\}}\bigr].
 \label{eq:terminal-last-visit}
\end{align}
Writing $q=R-j$ and shifting the marked edge by $j$ yields
\begin{align}
 \mathcal T_{n,v}
 \le \frac1{\pi_v}
 \sum_{\substack{j,q\ge1\\j+q\le n}}(j+q-1)
 \E\big[&\Delta_0^2\one_{\{X_j=v\}}
 \times\one_{\{X_{n-q}=v,\,
       X_r\ne v\ \text{for }n-q+1\le r\le n+j\}}\big].
 \label{eq:terminal-expanded}
\end{align}

Multiply \eqref{eq:terminal-expanded} by $w_v$ and sum over
$v\in\mathsf A$.  The Palm factor cancels exactly:
$w_v/\pi_v=1/\alpha$.  Fix $j\ge1$ and condition on
$\mathcal F_n^X=\sigma(X_0,\ldots,X_n)$.  For
$1\le q\le n-j$, set
\[
 I_{j,v,q}
 =\one_{\{X_j=v,\ X_{n-q}=v,\ 
          X_r\ne v\ \text{for }n-q+1\le r\le n\}}.
\]
Because $q\ge1$, the avoidance interval contains time $n$, so
$I_{j,v,q}=1$ forces $X_n\ne v$.  Moreover, $v$ must equal the already
revealed state $X_j$, and $n-q$ must be the last visit to this state in
$[0,n]$ (equivalently, in $[j,n]$).  Hence, pathwise,
\begin{equation}
 \sum_{v\in\mathsf A}\sum_{q=1}^{n-j}I_{j,v,q}\le1.
 \label{eq:terminal-unique-pair}
\end{equation}
For the unique possible pair $(v,q)$, the remaining unrevealed condition
is avoidance of $v$ at times $n+1,\ldots,n+j$.  By the Markov property
and the already established fact $X_n\ne v$,
\[
 \Pp\left(
   X_{n+1},\ldots,X_{n+j}\ne v\,\middle|\,\mathcal F_n^X\right)
 =\Pp_{X_n}(T_v>j)
 \le\sup_{x\in\Om,\,u\in\mathsf A}\Pp_x(T_u>j).
\]
Here $T_v=\inf\{t\ge0:X_t=v\}$ is the entrance time.  It is not the
strict return time $\tau_v^+$, since the conditioned starting state
$X_n$ is different from $v$.  Finally $j+q-1<n$, so
\eqref{eq:terminal-unique-pair} yields
\begin{equation}
 \sum_{v\in\mathsf A}w_v\mathcal T_{n,v}
 \le\frac n\alpha\E\Delta_0^2
 \sum_{j\ge1}\sup_{x\in\Om,\,v\in\mathsf A}\Pp_x(T_v>j).
 \label{eq:terminal-after-last}
\end{equation}

Put $b=\lceil2\mathsf H_*\rceil$.  Markov's inequality and the strong
Markov property in successive blocks give
\begin{equation}
 \sup_{x,\,v\in\mathsf A}\Pp_x(T_v>kb)\le2^{-k},
 \qquad k\ge1.
 \label{eq:hitting-geometric-tail}
\end{equation}
Therefore
\begin{equation}
 \sum_{j\ge1}\sup_{x,\,v\in\mathsf A}\Pp_x(T_v>j)
 \le2b\le4\mathsf H_*+2.
 \label{eq:hitting-tail-blocks}
\end{equation}
Using $\alpha>1/2$ and $\E\Delta_0^2=2\cD_1(F)$ in
\eqref{eq:terminal-after-last}, we obtain
\begin{equation}
 \sum_{v\in\mathsf A}w_v\mathcal T_{n,v}
 \le(16n\mathsf H_*+8n)\cD_1(F).
 \label{eq:averaged-terminal}
\end{equation}

\subsection{The averaged endpoint bound}

\begin{proposition}[Averaged endpoint-variance bound]
\label{prop:endpoint-closure}
For every count function $F$,
\begin{equation}
 \sum_{v\in\mathsf A}w_vV_v^{\mathrm{end}}
 \le(64n\mathsf H_*+48n)\cD_1(F).
 \label{eq:endpoint-closure}
\end{equation}
\end{proposition}

\begin{proof}
Average \eqref{eq:three-sectors} over $v\in\mathsf A$, and then use
\eqref{eq:averaged-complete} and \eqref{eq:averaged-terminal}:
\[
 \begin{aligned}
 \sum_{v\in\mathsf A}w_vV_v^{\mathrm{end}}
 &\le2\left\{
   16n(\mathsf H_*+1)+16n\mathsf H_*+8n\right\}\cD_1(F)\\
 &=(64n\mathsf H_*+48n)\cD_1(F).
 \end{aligned}
\]
\end{proof}

\section{Uniform lower bound and fixed-kernel sharpness}
\label{sec:gap-bounds}

We first assemble the preceding estimates to prove the uniform lower
bound. We then establish the matching window-length order for fixed
aperiodic kernels; the sharp dependence on the number of states is
treated in Section~\ref{sec:obstruction}.

\subsection{Proof of the uniform lower bound}
\label{sec:assembly}

We average the single-anchor inequalities and apply
Proposition~\ref{prop:endpoint-closure}.  The resulting comparison cost is
linear in the number of states.

\begin{theorem}[Uniform spectral-gap lower bound]
\label{thm:main-body}
For every integer $m\ge2$, every irreducible reversible kernel $P$ on $m$
states, and every $n\ge1$,
\begin{equation}
 \Gap(\widetilde P_n)
 \ge\frac1{1080m}\frac{\Gap(P)}n.
 \label{eq:main-gap}
\end{equation}
Thus $1080m$ is an admissible comparison cost.
For the optimal quantities defined in \eqref{eq:best-cm} and
\eqref{eq:best-Cm}, this implies
\begin{equation}
 c_m^\star\ge\frac1{1080m},
 \qquad
 C_m^{\mathrm{opt}}\le1080m.
 \label{eq:optimal-lower-consequence}
\end{equation}
\end{theorem}

\begin{proof}
If the support of $\mu_n$ is a singleton, then
$\Gap(\widetilde P_n)=1$ by convention, and \eqref{eq:main-gap} follows
immediately from $\gamma\le2$.  We may therefore assume that the count
support contains at least two points.

Assume first that $n\ge2$ and $n\gamma>m$.  Fix $v\in\mathsf A$.  Both the
anchor-free and anchor-containing sectors are centered at the same scalar
$c_v$.
Explicitly,
\[
 \E(F(K_0)-c_v)^2
 =\E\!\left[\one_{\{K_{0,v}=0\}}(F(K_0)-c_v)^2\right]
 +\E\!\left[\one_{\{K_{0,v}\ge1\}}(F(K_0)-c_v)^2\right].
\]
Therefore Corollary~\ref{cor:no-anchor}, Lemma~\ref{lem:first-anchor}, and
$\Var_{\mu_n}(F)\le\E(F(K_0)-c_v)^2$ give the single-anchor
inequality
\begin{align}
 \Var_{\mu_n}(F)
 &\le\{60\mathsf H_v^2+4(n-1)\mathsf H_v\}\cD_1(F)
 \notag\\
 &\quad+4(n\mathsf H_v+\mathsf H_v^2)\cD_1(F)
       +2p_vV_v^{\mathrm{end}} \notag\\
 &=\{64\mathsf H_v^2+(8n-4)\mathsf H_v\}\cD_1(F)
       +2p_vV_v^{\mathrm{end}}.
 \label{eq:global-exact}
\end{align}
Because the two sectors use the same center $c_v$, no between-sector
variance term or auxiliary projection-chain Poincar\'e constant appears.

Average \eqref{eq:global-exact} with
$w_v=\pi_v/\alpha$.  Since $p_v\le1$ and
$\mathsf H_v<\mathsf H_*$ for every $v\in\mathsf A$,
\begin{equation}
 \Var_{\mu_n}(F)
 \le\{64\mathsf H_*^2+8n\mathsf H_*\}\cD_1(F)
 +2\sum_{v\in\mathsf A}w_vV_v^{\mathrm{end}}.
 \label{eq:global-averaged}
\end{equation}
Proposition~\ref{prop:endpoint-closure} now gives
\begin{equation}
 \Var_{\mu_n}(F)
 \le\{64\mathsf H_*^2+136n\mathsf H_*+96n\}\cD_1(F).
 \label{eq:global-endpoint-term}
\end{equation}
The constants in \eqref{eq:global-endpoint-term} decompose as
\[
 136=8+2\cdot64,
 \qquad
 96=2\cdot48,
\]
where $8n\mathsf H_*$ comes from the averaged single-anchor estimate and
$64n\mathsf H_*+48n$ from
Proposition~\ref{prop:endpoint-closure}.  In that proposition,
$\alpha>1/2$ is used exactly when $8/\alpha$ is replaced by $16$ in
\eqref{eq:averaged-complete} and when the terminal factor $1/\alpha$ is
bounded by $2$.

Substituting $\mathsf H_*=3m/\gamma$ gives the numerical budget
\begin{center}
\small
\begin{tabular}{@{}lll@{}}
\toprule
Term & Input & Coefficient of $mn/\gamma$\\
\midrule
$64\mathsf H_*^2$
 & $\mathsf H_*=3m/\gamma$, $n\gamma>m$
 & $64\cdot9=576$\\
$136n\mathsf H_*$
 & $\mathsf H_*=3m/\gamma$
 & $136\cdot3=408$\\
$96n$
 & $\gamma\le2\le m$ (using $m\ge2$)
 & $96$\\
\bottomrule
\end{tabular}
\end{center}
Thus $576+408+96=1080$.  The long-window condition is used only to bound
$\mathsf H_*^2<9mn/\gamma$; the final row uses $\gamma\le2\le m$.
Consequently,
\begin{equation}
 \Var_{\mu_n}(F)
 \le1080m\frac n\gamma\cD_1(F)
 \qquad(n\gamma>m).
 \label{eq:long-poincare}
\end{equation}

If $n\ge2$ and $n\gamma\le m$, \eqref{eq:short-final} gives
\[
 \cD_1(F)\ge\frac{c_0}{m}\frac\gamma n\Var_{\mu_n}(F).
\]
Since $c_0=(1-e^{-1})/2>1/1080$, this is stronger than
\eqref{eq:main-gap}.  Finally, if $n=1$, the map $x\mapsto e_x$ is a
bijective relabeling and $\widetilde P_1=P$.  The variational formula
\eqref{eq:gap-variational} now proves \eqref{eq:main-gap} in every
non-singleton case.
Taking the infimum over $P$ and $n\ge2$ gives
\eqref{eq:optimal-lower-consequence}.
\end{proof}

\subsection{A window-length upper bound for fixed kernels}
\label{sec:upper}

The lower bound has the correct $n^{-1}$ order for every fixed kernel with
a nonconstant eigenvalue strictly between $-1$ and $1$.  The proof uses a
linear count statistic and is included to distinguish sharpness in $n$ from
the separate question of sharp dependence on $m$.  This calculation first
appeared in~\cite{XiangXinZhang2026} and is recalled here for completeness.

\begin{lemma}[Linear-statistic Rayleigh quotient]
\label{lem:linear-upper}
Let $f$ be a nonconstant eigenfunction of $P$ with eigenvalue
$\lambda\in(-1,1)$, normalized by
\[
 \E_\pi f=0,
 \qquad \E_\pi f^2=1.
\]
Then
\begin{equation}
 \Gap(\widetilde P_n)
 \le
 \frac{1-\lambda^n}
 {n+2\sum_{r=1}^{n-1}(n-r)\lambda^r}.
 \label{eq:linear-upper}
\end{equation}
\end{lemma}

\begin{proof}
Define the linear count function
\begin{equation*}
 F(k)=\sum_{x\in\Om}k_xf(x).
 \label{eq:linear-count-function}
\end{equation*}
Then $F(K_0)=\sum_{i=0}^{n-1}f(X_i)$.  Since $Pf=\lambda f$,
\[
 \E[f(X_i)f(X_j)]=\lambda^{|i-j|},
\]
and hence
\begin{equation}
 \Var_{\mu_n}(F)
 =n+2\sum_{r=1}^{n-1}(n-r)\lambda^r.
 \label{eq:linear-variance}
\end{equation}
Writing the right side as $D_n(\lambda)$, direct summation gives
\begin{equation}
 D_n(\lambda)
 =n\frac{1+\lambda}{1-\lambda}
 -\frac{2\lambda(1-\lambda^n)}{(1-\lambda)^2}.
 \label{eq:Dn-closed}
\end{equation}
The denominator is strictly positive.  If $0\le\lambda<1$, the covariance
sum in \eqref{eq:linear-variance} is at least its leading term $n>0$.  If
$-1<\lambda<0$, \eqref{eq:Dn-closed} is the sum of the strictly positive
term $n(1+\lambda)/(1-\lambda)$ and the nonnegative term
$-2\lambda(1-\lambda^n)/(1-\lambda)^2$.

One window shift changes the statistic by
\[
 F(K_1)-F(K_0)=f(X_n)-f(X_0).
\]
Therefore, we obtain
\begin{equation*}
 \cD_1(F)
 =\frac12\E(f(X_n)-f(X_0))^2
 =1-\lambda^n.
 \label{eq:linear-energy}
\end{equation*}
The variational formula \eqref{eq:gap-variational} gives
\eqref{eq:linear-upper}.
\end{proof}

\begin{proposition}[Fixed-kernel $O(n^{-1})$ upper bound]
\label{prop:fixed-upper}
Under the hypothesis of Lemma~\ref{lem:linear-upper},
\begin{equation}
 \Gap(\widetilde P_n)\le\frac{C_\lambda}{n}
 \qquad(n\ge2),
 \label{eq:fixed-upper}
\end{equation}
where one may take
\begin{equation}
 C_\lambda=
 \begin{cases}
 1,&0\le\lambda<1,\\[3pt]
 \displaystyle\frac{2(1-\lambda)}{1+\lambda},&-1<\lambda<0.
 \end{cases}
 \label{eq:C-lambda}
\end{equation}
\end{proposition}

\begin{proof}
Let $D_n(\lambda)$ denote the denominator in
\eqref{eq:linear-upper}.  Use the closed form
\eqref{eq:Dn-closed}.
If $0\le\lambda<1$, every term in the defining covariance sum is
nonnegative, so $D_n(\lambda)\ge n$, while $1-\lambda^n\le1$.  If
$-1<\lambda<0$, the second term on the right of
\eqref{eq:Dn-closed} is nonnegative; hence
\[
 D_n(\lambda)\ge n\frac{1+\lambda}{1-\lambda},
 \qquad 1-\lambda^n\le2.
\]
Insert these bounds in Lemma~\ref{lem:linear-upper}.
\end{proof}

\begin{corollary}[Sharp window-length order for fixed aperiodic chains]
\label{cor:fixed-matching}
For every fixed finite irreducible reversible aperiodic kernel $P$ on at
least two states, there exist constants $0<c_-(P)\le c_+(P)<\infty$ such
that
\begin{equation}
 \frac{c_-(P)}n
 \le\Gap(\widetilde P_n)
 \le\frac{c_+(P)}n
 \qquad(n\ge2).
 \label{eq:fixed-matching}
\end{equation}
\end{corollary}

\begin{proof}
The lower bound follows from Theorem~\ref{thm:main-body}, with
$c_-(P)=\Gap(P)/(1080m)$.  Let
\[
 \Lambda(P)=\sigma(P)\setminus\{1\}.
\]
For a finite irreducible reversible aperiodic kernel on at least two states,
$\Lambda(P)$ is a nonempty finite subset of $(-1,1)$.  Applying
Proposition~\ref{prop:fixed-upper} to any $\lambda\in\Lambda(P)$ gives an
admissible upper constant $C_\lambda$; one may therefore take
\[
 c_+(P)=\min_{\lambda\in\Lambda(P)}C_\lambda.
\]
This minimum is finite and lies in $[1,\infty)$.  Thus the upper constant
depends on the chosen nonconstant spectral mode, with the displayed
minimum giving the best constant supplied by this argument.
\end{proof}

\begin{remark}[Periodic endpoint cases]
The aperiodicity assumption is used only as a convenient sufficient
condition for the matching upper statement; it can be weakened to the
existence of a nonconstant eigenvalue in $(-1,1)$.  It is not needed for
the main lower bound.  For example, the deterministic two-state
alternating chain has nonconstant eigenvalue $-1$; depending on the parity
of $n$, its count support may collapse or alternate without the generic
$n^{-1}$ relaxation scale.
\end{remark}

\section{A Nested Rare-State Obstruction}
\label{sec:obstruction}

In this section, we give the proof of Theorem~\ref{thm:exact-cardinality-upper} by constructing a nested rare-state example. The proof includes four parts as follows. 

\subsection{The nested reversible family}

Let \(d=m-1\).  Fix \(0<\rho<1\), and denote
\begin{equation*}
 s=\rho^2,
 \qquad a=1-s,
 \qquad c=as,
 \qquad b=a^2.
 \label{eq:nested-parameters}
\end{equation*}
Thus, \(a=b+c\) and \(s+c=1-b\).  Suppose first that \(d\ge2\).
Choose
\begin{equation}
 0<\varepsilon_1<a,
 \qquad
 0<\varepsilon_i<b\quad(2\le i<d).
 \label{eq:nested-epsilon-range}
\end{equation}
On the rare states \(\{1,\ldots,d\}\), define a killed kernel \(Q\) by
\begin{equation*}
 Q_{ii}=s\quad(1\le i\le d),
 \qquad Q_{i,i+1}=\varepsilon_i\quad(1\le i<d),
 \qquad Q_{i,i-1}=c\quad(2\le i\le d),
 \label{eq:nested-killed-kernel}
\end{equation*}
with unlisted entries zero.  Its killing probabilities are
\begin{equation*}
 \kappa_1=a-\varepsilon_1,
 \qquad
 \kappa_i=b-\varepsilon_i\quad(2\le i<d),
 \qquad
 \kappa_d=b.
 \label{eq:nested-killing}
\end{equation*}
Set
\begin{equation*}
 \mu_1=c,
 \qquad
 \mu_i=c^{2-i}\prod_{h<i}\varepsilon_h\quad(2\le i\le d),
 \qquad
 r_i=\mu_i\kappa_i,
 \qquad
 R=\sum_{i=1}^d r_i.
 \label{eq:nested-rare-masses}
\end{equation*}
For sufficiently small \(\eta>0\), define a kernel on
\(\{0,1,\ldots,d\}\) by
\begin{equation}
 \begin{alignedat}{2}
 P_{\eta,\varepsilon,\rho}(0,0)&=1-\eta R,
 &\qquad P_{\eta,\varepsilon,\rho}(0,i)&=\eta r_i,\\
 P_{\eta,\varepsilon,\rho}(i,0)&=\kappa_i,
 &P_{\eta,\varepsilon,\rho}(i,j)&=Q_{ij}.
 \end{alignedat}
 \label{eq:nested-full-kernel}
\end{equation}
The identities
\begin{equation*}
 \mu_i\varepsilon_i=\mu_{i+1}c,
 \qquad
 1\cdot\eta r_i=(\eta\mu_i)\kappa_i
 \label{eq:nested-detailed-balance}
\end{equation*}
show that the unnormalized reversible measure is
\begin{equation}
 (1,\eta\mu_1,\ldots,\eta\mu_d).
 \label{eq:nested-stationary-law}
\end{equation}
All rare states have a positive edge to the root, and every upward edge is
positive, so the kernel is irreducible.

When \(d=1\), there is no \(\varepsilon\)-parameter. Taking
\(Q_{11}=s\), \(\kappa_1=a\), \(\mu_1=c\), and using
\eqref{eq:nested-full-kernel}, we have
\begin{equation}
 P_{\eta,\rho}
 =\begin{pmatrix}
  1-\eta ca&\eta ca\\
  a&s
 \end{pmatrix},
 \label{eq:nested-two-state-kernel}
\end{equation}
with reversible measure \((1,\eta c)\).

Let $Q_{\varepsilon,\rho}=Q$.  At \(\eta=0\), the matrix is reducible,
so we do not assign it an
irreducible-chain gap. We instead define the limiting gap parameter
\begin{equation*}
 \gamma_0(\varepsilon,\rho)
 =1-\max\sigma(Q_{\varepsilon,\rho}).
 \label{eq:nested-limiting-gap}
\end{equation*}
The killed block is reversible and has real spectrum.  As
\(\max_i\varepsilon_i\to0\), it tends to a lower triangular matrix with
diagonal \(s\), while finite-matrix continuity gives
\(\Gap(P_{\eta,\varepsilon,\rho})\to
\gamma_0(\varepsilon,\rho)\) as \(\eta\downarrow0\).  Hence, for each
fixed \(\rho\),
\begin{equation}
 \lim_{\varepsilon\to0}\lim_{\eta\downarrow0}
 \Gap(P_{\eta,\varepsilon,\rho})=1-s=a,
 \label{eq:nested-base-gap}
\end{equation}
where \(\varepsilon\to0\) means
\(\max_i\varepsilon_i\to0\). Combining the matrix continuity also permits the two
small parameters to be chosen successively as in the finite diagonal below.
For \(d=1\), the nonconstant eigenvalue in
\eqref{eq:nested-two-state-kernel} is \(s-\eta ca\), and the same conclusion
holds directly.

\paragraph{Parameters and order of limits.}
For the family above, write \(\mu_n^{\eta,\varepsilon,\rho}\) and
\(\widetilde P_n^{\eta,\varepsilon,\rho}\) for the count law and induced
count kernel, respectively.  For a nonconstant test \(f\), set
\[
 \mathcal R_n^{\eta,\varepsilon,\rho}(f)
 =n\frac{\cE_{\widetilde P_n^{\eta,\varepsilon,\rho}}(f,f)}
          {\Var_{\mu_n^{\eta,\varepsilon,\rho}}(f)}.
\]
All limits \(\varepsilon\to0\) mean
\(\max_i\varepsilon_i\to0\).  From inner to outer, the limiting argument
is ordered as
\[
 \eta\downarrow0,\qquad n\to\infty,\qquad
 \varepsilon\to0,\qquad x_d\uparrow\rho^{-1},\qquad \rho\uparrow1.
\]
The finite diagonal below chooses the parameters in the reverse order.
When $d=1$, $\varepsilon$ denotes the empty tuple and
$\gamma_0(\varepsilon,\rho)=1-s$.

\subsection{The limiting rare-excursion form}

We identify a count function with a function of the rare coordinates
\(k=(k_1,\ldots,k_d)\); the root count is then
\(n-\sum_i k_i\).  We use finitely supported functions satisfying
\begin{equation}
 f(k_1,\ldots,k_{d-1},0)=0.
 \label{eq:nested-zero-boundary-test}
\end{equation}
For \(d\ge2\), put
\begin{equation*}
 E=\prod_{i=1}^{d-1}\varepsilon_i,
 \label{eq:nested-common-factor}
\end{equation*}
and take \(E=1\) when \(d=1\).  Define
\begin{equation*}
 \tau(z)=\frac{z}{1-sz},
 \qquad
 \kappa_1^0=a,
 \qquad
 \kappa_i^0=b\quad(i\ge2),
 \label{eq:nested-tau-kappa}
\end{equation*}
and
\begin{equation}
 V_d(z_1,\ldots,z_{d-1})
 =\sum_{p=1}^d c^{d-p}\kappa_p^0
   \prod_{i=p}^{d-1}\tau(z_i).
 \label{eq:nested-Vd}
\end{equation}
Empty products are one; in particular, \(V_1=a\).  Equivalently,
\begin{equation*}
 V_1=a,
 \qquad
 V_j=b+c\tau(z_{j-1})V_{j-1}\quad(j\ge2).
 \label{eq:nested-V-recursion}
\end{equation*}

\begin{lemma}[Leading mass and conductances]
\label{lem:nested-leading-form}
There are nonnegative coefficients \(M_k,H_{j,k}\), and \(B_{k'}\), where
\(k'=(k_1,\ldots,k_{d-1})\), with generating functions
\begin{align}
 \sum_{k_d\ge1}M_kz^k
 &=c^{2-d}\tau(z_d)V_d^2,
 \label{eq:nested-mass-gf}\\
 \sum_{k_d\ge1}H_{j,k}z^k
 &=c^{2-j}V_d\tau(z_d)
   \frac{\prod_{i=j+1}^{d-1}\tau(z_i)}{1-sz_j},
 &&1\le j<d,
 \label{eq:nested-horizontal-gf}\\
 \sum_{k_d\ge1}H_{d,k}z^k
 &=s c^{2-d}V_d\tau(z_d),
 \label{eq:nested-deep-gf}\\
 \sum_{k'}B_{k'}z^{k'}
 &=c^{2-d}V_d.
 \label{eq:nested-boundary-gf}
\end{align}
Here \(z^k=\prod_i z_i^{k_i}\), and every \(V_d\) is evaluated at
\((z_1,\ldots,z_{d-1})\).  If
\begin{align}
 \|f\|_M^2&=\sum_{k_d\ge1}M_kf(k)^2,
 \label{eq:nested-mass-norm}\\
 \cE_d(f)&=\sum_{j=1}^d\sum_{k_d\ge1}
 H_{j,k}\bigl(f(k+e_j)-f(k)\bigr)^2
 +\sum_{k'}B_{k'}f(k',1)^2,
 \label{eq:nested-leading-energy}
\end{align}
then every finitely supported \(f\) satisfying
\eqref{eq:nested-zero-boundary-test} and \(\|f\|_M>0\) obeys
\begin{equation}
 \lim_{\varepsilon\to0}\lim_{n\to\infty}\lim_{\eta\downarrow0}
 n\,
 \frac{\cE_{\widetilde P_n}(f,f)}{\Var(f(K_0))}
 =\frac{\cE_d(f)}{\|f\|_M^2}.
 \label{eq:nested-sequential-limit}
\end{equation}
For every sufficiently small positive $\varepsilon$, the two inner limits
exist, with positive denominators for all sufficiently large $n$ and all
sufficiently small $\eta>0$.  For \(d=1\), the first limit is absent.
\end{lemma}

\begin{proof}
Fix \(\rho\), fix positive \(\varepsilon_i\) in
\eqref{eq:nested-epsilon-range}, and write
\begin{equation}
 U_\varepsilon=\sum_{i=1}^d\mu_i,
 \qquad
 \kappa_*=\min_{1\le i\le d}\kappa_i,
 \qquad
 \theta=1-\kappa_*<1.
 \label{eq:nested-tail-parameters}
\end{equation}
The normalized stationary law in \eqref{eq:nested-stationary-law} is
\begin{equation}
 \pi_\eta(0)=\frac1{1+\eta U_\varepsilon},
 \qquad
 \pi_\eta(i)=\frac{\eta\mu_i}{1+\eta U_\varepsilon}.
 \label{eq:nested-normalized-stationary-law}
\end{equation}
All estimates up to the final \(\varepsilon\)-pass are made with these
parameters fixed.  Moreover, when \(d\ge2\) and \(\rho\) is fixed, the tail
estimates are uniform for \(0<\varepsilon_1\le a/2\) and
\(0<\varepsilon_i\le b/2\), \(2\le i<d\).  Only the closures of these
half-open boxes are used in the final coefficientwise
\(\varepsilon\)-limit; for \(d=1\) there is no such parameter.

\smallskip
\noindent\emph{Geometric-tail sublemma.}
Let \(L\) be the lifetime of the killed chain with kernel \(Q\), including
its initial rare state.  Since every row sum of \(Q\) is at most \(\theta\),
\begin{equation}
 \Pp_i(L\ge h)=(Q^{h-1}\one)(i)\le\theta^{h-1},
 \qquad h\ge1.
 \label{eq:nested-excursion-tail}
\end{equation}
For a complete rare word \(w=(x_1,\ldots,x_\ell)\), and for a rare prefix
\(u=(x_1,\ldots,x_h)\), put
\begin{align*}
 A_\varepsilon(w)
 &=r_{x_1}\prod_{t<\ell}Q_{x_t,x_{t+1}}\kappa_{x_\ell},
 \label{eq:nested-word-weight}\\
 C_\varepsilon(u)
 &=r_{x_1}\prod_{t<h}Q_{x_t,x_{t+1}}.
\end{align*}
Thus \(A_\varepsilon(w)\) is the coefficient, after removal of the one
factor \(\eta\), of the complete excursion \(0,w,0\), whereas
\(C_\varepsilon(u)\) already sums over every continuation after the
observed prefix.  To verify the latter statement, write
\(\kappa=(\kappa_1,\ldots,\kappa_d)^\mathsf T=(I-Q)\one\).  The geometric
tail gives \(Q^q\one\to0\), and hence
\begin{equation*}
 \sum_{q\ge0}Q^q\kappa
 =\sum_{q\ge0}Q^q(I-Q)\one=\one.
 \label{eq:nested-continuation-sum}
\end{equation*}
Consequently, for every prefix \(u\),
\begin{equation*}
 \sum_{\substack{w:\ u\text{ is a prefix of }w}}
 A_\varepsilon(w)=C_\varepsilon(u).
 \label{eq:nested-prefix-marginalization}
\end{equation*}
If \(|w|\) and \(|u|\) denote word length, then
\begin{align}
 \sum_{|w|\ge h}A_\varepsilon(w)
 &\le R\theta^{h-1},
 \label{eq:nested-complete-word-tail}\\
 \sum_{|u|\ge h}C_\varepsilon(u)
 &\le \frac{R}{1-\theta}\theta^{h-1}.
 \label{eq:nested-prefix-tail}
\end{align}
Indeed, the first estimate is \eqref{eq:nested-excursion-tail} averaged
with the entrance weights \(r_i\); the second follows by summing
\(\sum_i r_i(Q^{q-1}\one)(i)\) over \(q\ge h\).  This proves the sublemma
and also justifies every rearrangement of the nonnegative word sums below.

Write \(k(w)\) for the rare-state count vector of a word, and let \(u^-\)
denote the count vector obtained from \(k(u)\) by deleting the last letter
of \(u\).  Equivalently, if the last letter is \(j\), then
\(u^-=k(u)-e_j\).  For fixed \(\varepsilon\), define
\begin{align}
 M_k^{(\varepsilon)}
 &=\sum_{k(w)=k}A_\varepsilon(w),
 \label{eq:nested-fixed-epsilon-mass}\\
 H_{j,k}^{(\varepsilon)}
 &=\sum_{\substack{u_h=j,\ u^-=k}}C_\varepsilon(u),
 \qquad k_d\ge1,
 \label{eq:nested-fixed-epsilon-horizontal}\\
 B_{k'}^{(\varepsilon)}
 &=\sum_{\substack{u_h=d,\ u^-=(k',0)}}C_\varepsilon(u).
 \label{eq:nested-fixed-epsilon-boundary}
\end{align}
For a finitely supported \(f\) satisfying
\eqref{eq:nested-zero-boundary-test}, set
\begin{align}
 M_\varepsilon(f)
 &=\sum_{k_d\ge1}M_k^{(\varepsilon)}f(k)^2,
 \label{eq:nested-fixed-epsilon-mass-form}\\
 D_\varepsilon(f)
 &=\sum_{j=1}^d\sum_{k_d\ge1}
 H_{j,k}^{(\varepsilon)}\bigl(f(k+e_j)-f(k)\bigr)^2
   +\sum_{k'}B_{k'}^{(\varepsilon)}f(k',1)^2.
 \label{eq:nested-fixed-epsilon-energy-form}
\end{align}
Only finitely many words enter these two forms: the support of \(f\),
together with its nearest neighbours, has bounded total count.  The tail
bounds above are nevertheless useful for the endpoint words before that
support restriction is imposed.

We next take \(\eta\downarrow0\) with \(n\) fixed.  Expanding stationary
path cylinders according to the number of transitions from the root into
the rare set gives finite coefficients \(V_{n,\varepsilon}(f)\) and
\(D_{n,\varepsilon}(f)\) such that
\begin{align}
 \E f(K_0)^2
 &=\eta V_{n,\varepsilon}(f)+O_{n,\varepsilon,f}(\eta^2),
 \label{eq:nested-second-moment-expansion}\\
 \cE_{\widetilde P_n}(f,f)
 &=\eta D_{n,\varepsilon}(f)+O_{n,\varepsilon,f}(\eta^2).
 \label{eq:nested-energy-eta-expansion}
\end{align}
Here and below an excursion already in progress at the left endpoint is
counted using \eqref{eq:nested-normalized-stationary-law}; reversibility
identifies it with the corresponding reversed suffix.  A nonzero value of
\(f(K_0)\) requires at least one visit to state \(d\).  Consequently,
\begin{equation*}
 |\E f(K_0)|
 \le \|f\|_\infty\Pp(K_{0,d}\ge1)
 \le \|f\|_\infty n\pi_\eta(d)
 =O_{n,\varepsilon,f}(\eta),
 \label{eq:nested-mean-bound}
\end{equation*}
and the square of the mean has no first-order \(\eta\)-term.  Hence
\begin{equation}
 \lim_{\eta\downarrow0}
 \frac{\Var(f(K_0))}{\eta}=V_{n,\varepsilon}(f),
 \qquad
 \lim_{\eta\downarrow0}
 \frac{\cE_{\widetilde P_n}(f,f)}{\eta}
 =D_{n,\varepsilon}(f).
 \label{eq:nested-fixed-window-eta-limits}
\end{equation}

We make the subsequent \(n\)-limit explicit.  Let \(L_f\) exceed the
largest total count at which \(f\) is nonzero.  A complete excursion word
\(w\) of length \(\ell\le L_f\) has exactly
\((n-\ell-1)_+\) placements with both adjacent root letters inside the
window.  Excursions cut by the window form two endpoint families (including
the case in which both endpoints cut the same word), rather than two single
terms.  For a word of length \(\ell\), at most \(2(\ell+1)\) alignments
belong to these families.  Thus the first-order coefficient has the
placement decomposition
\begin{equation*}
 V_{n,\varepsilon}(f)
 =\sum_w(n-|w|-1)_+A_\varepsilon(w)f(k(w))^2
   +E_{n,\varepsilon}(f),
 \quad
 0\le E_{n,\varepsilon}(f)
 \le2\|f\|_\infty^2\sum_w(|w|+1)A_\varepsilon(w).
 \label{eq:nested-variance-placement}
\end{equation*}
The last series is finite by
\eqref{eq:nested-complete-word-tail}.  Since only words of length at most
\(L_f\) survive in the first sum, it follows that
\begin{equation}
 V_{n,\varepsilon}(f)
 =nM_\varepsilon(f)+O_{\varepsilon,f}(1).
 \label{eq:nested-variance-coefficient}
\end{equation}

For the energy, use the pathwise identity
\begin{equation*}
 K_1-K_0=e_{X_n}-e_{X_0}.
 \label{eq:nested-count-increment}
\end{equation*}
To first order in \(\eta\), if exactly one endpoint is the root, the rare
endpoint is the last letter \(j\) of a prefix \(u\), and the count edge is
\(u^-\leftrightarrow u^-+e_j\).  The orientation with \(X_0=0\) has
coefficient \(C_\varepsilon(u)\).  Reversibility gives the opposite
orientation with the same coefficient, so the two orientations cancel the
factor \(1/2\) in the Dirichlet form.  The cases \((u^-)_d\ge1\) and
\((u^-)_d=0\) are exactly the \(H^{(\varepsilon)}\)- and
\(B^{(\varepsilon)}\)-terms in
\eqref{eq:nested-fixed-epsilon-energy-form}; in the latter case a nonzero
edge requires \(j=d\) by \eqref{eq:nested-zero-boundary-test}.

There are two remaining possibilities.  If the two rare endpoints belong
to different excursions, the path contains two root-to-rare transitions
and contributes only \(O_{n,\varepsilon,f}(\eta^2)\) in
\eqref{eq:nested-energy-eta-expansion}.  If they belong to the same
excursion, that excursion survives across the entire window; this is a
first-order event, but it is geometrically small.  More explicitly,
\begin{equation*}
 \Pp_\eta\bigl(X_0\text{ and }X_n\text{ lie in the same rare excursion}\bigr)
 =\frac{\eta}{1+\eta U_\varepsilon}
   \sum_{i=1}^d\mu_i(Q^n\one)(i)
 \le\eta U_\varepsilon\theta^n.
 \label{eq:nested-spanning-excursion}
\end{equation*}
Hence its energy coefficient is at most
\(4\|f\|_\infty^2U_\varepsilon\theta^n\).  Prefixes longer than the window
have the same geometric bound by \eqref{eq:nested-prefix-tail}.  Therefore
\begin{equation}
 |D_{n,\varepsilon}(f)-D_\varepsilon(f)|
 \le C_{\varepsilon,f}\theta^n.
 \label{eq:nested-energy-n-limit}
\end{equation}
Combining \eqref{eq:nested-fixed-window-eta-limits},
\eqref{eq:nested-variance-coefficient}, and
\eqref{eq:nested-energy-n-limit} gives
\begin{equation}
 \lim_{n\to\infty}\lim_{\eta\downarrow0}
 n\frac{\cE_{\widetilde P_n}(f,f)}{\Var(f(K_0))}
 =\frac{D_\varepsilon(f)}{M_\varepsilon(f)},
 \label{eq:nested-two-inner-limits}
\end{equation}
whenever \(M_\varepsilon(f)>0\).  Because the limits are iterated, uniformity
in $n$ is not required: $\eta$ is sent to zero first for each fixed $n$.

It remains to identify the leading \(\varepsilon\)-coefficient.  For a
fixed count vector, the sums in
\eqref{eq:nested-fixed-epsilon-mass}--\eqref{eq:nested-fixed-epsilon-boundary}
are finite, hence polynomial in
\((\varepsilon_1,\ldots,\varepsilon_{d-1})\).  Suppose a word enters the
rare set at level \(p\) and reaches level \(d\).  Its entrance weight
contains \(\prod_{i<p}\varepsilon_i\), and reaching \(d\) forces one
upward crossing of every cut \(i\to i+1\), \(p\le i<d\).  Thus every
coefficient relevant to \(f\) is divisible by
\(E=\prod_{i<d}\varepsilon_i\).  Equality in total degree permits exactly
one upward crossing of each remaining cut.  Before the first visit to
\(d\), the word therefore holds or moves monotonically upward; afterwards
it holds or moves monotonically downward.  The terms
\(-\varepsilon_i\) in the killing probabilities have strictly higher
degree.  The leading complete words are consequently the mountains
\begin{equation*}
 p\nearrow p+1\nearrow\cdots\nearrow d
 \searrow d-1\searrow\cdots\searrow q,
 \label{eq:nested-mountain-word}
\end{equation*}
with arbitrary positive holding blocks.

A positive holding block at level \(i\) has generating function
\(\tau(z_i)\).  For fixed entrance and exit levels \((p,q)\), the leading
transition and killing coefficient after division by \(E\) is
\begin{equation*}
 c^{2-p}\kappa_p^0\,c^{d-q}\kappa_q^0.
 \label{eq:nested-mountain-coefficient}
\end{equation*}
Summing independently over \(p\) and \(q\) gives
\eqref{eq:nested-mass-gf}.  For a prefix edge adding a visit at level
\(j<d\), the new endpoint either creates the terminal visit at \(j\) after
the descent from \(j+1\), or extends a holding block already present at
\(j\).  After the endpoint is deleted, these two cases contribute the
factor
\begin{equation*}
 1+s\tau(z_j)=\frac1{1-sz_j},
 \label{eq:nested-holding-identity}
\end{equation*}
and give \eqref{eq:nested-horizontal-gf}.  If \(j=d\) and the deleted
prefix still contains a visit to \(d\), the endpoint extends the deepest
holding block and gives \eqref{eq:nested-deep-gf}.  If it is the first
visit to \(d\), deletion lands on the zero boundary and gives
\eqref{eq:nested-boundary-gf}.  These are the only leading cases.  In
particular, after a prefix has already visited \(d\), an endpoint reached
from \(j-1\) would use an additional upward transition
\(Q_{j-1,j}=\varepsilon_{j-1}\); after division by \(E\) it has positive
remaining \(\varepsilon\)-degree and vanishes at the full origin.  For
\(j=d\), the conditions ``a previous visit to \(d\) remains after
deletion'' and ``the endpoint is the first visit to \(d\)'' are disjoint,
so the \(H_d\)- and \(B\)-coefficients do not overlap.

We have thus proved, coefficient by coefficient on the finite support of
\(f\),
\begin{equation}
 \left(\frac{M_\varepsilon(f)}E,
       \frac{D_\varepsilon(f)}E\right)
 \longrightarrow
 \bigl(\|f\|_M^2,\cE_d(f)\bigr)
 \qquad(\max_i\varepsilon_i\to0).
 \label{eq:nested-epsilon-layer}
\end{equation}
There are no ratios \(\varepsilon_i/\varepsilon_j\): divisibility by the
full monomial \(E\) makes both quotients polynomial and therefore
continuous at the full origin.  Since \(\|f\|_M>0\), the denominator in
\eqref{eq:nested-two-inner-limits} is positive for all sufficiently small
\(\varepsilon\).  For each such \(\varepsilon\),
\eqref{eq:nested-variance-coefficient} gives
\(V_{n,\varepsilon}(f)>0\) for all sufficiently large \(n\); for each
such fixed \(n\), \eqref{eq:nested-second-moment-expansion} then gives
\(\Var(f(K_0))>0\) for all sufficiently small \(\eta>0\).  Taking the ratio in
\eqref{eq:nested-epsilon-layer} proves
\eqref{eq:nested-sequential-limit}.  When \(d=1\), the same fixed-\(\eta\)
and fixed-\(n\) argument applies with \(E=1\), and the final
\(\varepsilon\)-pass is absent.
\end{proof}

\subsection{Product tests and the scalar recursion}

For \(0<x_i<s^{-1/2}\), take the product test
\begin{equation*}
 f_x(k)=\prod_{i=1}^d x_i^{k_i}\quad(k_d\ge1),
 \qquad f_x(k',0)=0.
 \label{eq:nested-product-test}
\end{equation*}
Put
\begin{equation*}
 D_i=1-sx_i^2,
 \qquad
 t_i=\frac{x_i^2}{D_i},
 \qquad
 V=V_d(x_1^2,\ldots,x_{d-1}^2).
 \label{eq:nested-product-notation}
\end{equation*}
The last equality specifies the evaluation point of \(V_d\).  Evaluating
\eqref{eq:nested-mass-gf}--\eqref{eq:nested-boundary-gf} gives
\begin{equation}
 \|f_x\|_M^2=c^{2-d}t_dV^2.
 \label{eq:nested-product-mass}
\end{equation}
For \(j<d\), the \(j\)-edge energy is
\begin{equation*}
 c^{2-j}Vt_d
 \frac{\prod_{i=j+1}^{d-1}t_i}{D_j}(x_j-1)^2.
 \label{eq:nested-product-horizontal-energy}
\end{equation*}
The deepest interior and boundary energies are, respectively,
\begin{equation*}
 s c^{2-d}Vt_d(x_d-1)^2,
 \qquad
 c^{2-d}Vx_d^2.
 \label{eq:nested-product-deep-boundary}
\end{equation*}
Since \(x_d^2/t_d=D_d\), division by
\eqref{eq:nested-product-mass} yields the exact quotient
\begin{equation}
 \mathcal Q_{d,\rho}(x)
 =\frac1V\left[
 1+s-2sx_d
 +\sum_{j=1}^{d-1}c^{d-j}(x_j-1)^2
   \frac{\prod_{i=j+1}^{d-1}t_i}{D_j}
 \right].
 \label{eq:nested-product-quotient}
\end{equation}

The product at the critical value \(x_d=1/\rho\) is not square summable.
It is used only as a limiting test.  For a strictly subcritical
\(x_d<1/\rho\),
all mass and conductance series above are finite.  Let
$\Lambda_N=\{0,\ldots,N\}^{d-1}\times\{1,\ldots,N\}$ and set
$f_x^{(N)}=f_x\one_{\Lambda_N}$.  The mass, the boundary term, and the
energies of edges contained in $\Lambda_N$ converge by monotone
convergence.  Put
$\partial_j\Lambda_N=\{k\in\Lambda_N:k+e_j\notin\Lambda_N\}$, the
basepoints of the outgoing $j$-edges crossing the upper truncation face, and
write
$\cE_d(f_x^{(N)};\partial\Lambda_N)$ for the terms of
\eqref{eq:nested-leading-energy} supported on these edges.  This additional
face energy is bounded by
\begin{equation*}
 \cE_d\bigl(f_x^{(N)};\partial\Lambda_N\bigr)
 \le2\sum_{j=1}^d(1+x_j^2)
       \sum_{k\in\partial_j\Lambda_N}H_{j,k}f_x(k)^2
 \longrightarrow0.
 \label{eq:nested-truncation-tail}
\end{equation*}
Indeed, each full series on the right is finite by
\eqref{eq:nested-horizontal-gf}--\eqref{eq:nested-deep-gf} evaluated at
$z_i=x_i^2$, while the boundary sets escape every finite rectangle.  Thus
finite-support tests approach \eqref{eq:nested-product-quotient}, after
which one may send \(x_d\uparrow1/\rho\).  At that boundary,
\begin{equation*}
 1+s-2s/\rho=(1-\rho)^2.
 \label{eq:nested-critical-deep-term}
\end{equation*}

We now let \(a=1-s\downarrow0\).  For \(j<d\), set
\begin{equation*}
 x_j=1+ay_j,
 \qquad
 0\le y_j<\frac12,
 \qquad
 w_j=1-2y_j.
 \label{eq:nested-shallow-scaling}
\end{equation*}
For fixed \(d\) and fixed \(y\),
\begin{equation*}
 D_j=aw_j+O(a^2),
 \qquad
 t_j=(aw_j)^{-1}+O(1),
 \qquad
 c=a+O(a^2),
 \qquad
 (1-\rho)^2=\frac{a^2}{4}+O(a^3).
 \label{eq:nested-scaling-expansions}
\end{equation*}
The \(p=1\) term in \eqref{eq:nested-Vd} has order \(a\), while every
\(p\ge2\) term has order \(a^2\).  Hence,
\begin{equation*}
 V=\frac{a}{\prod_{j=1}^{d-1}w_j}+O(a^2).
 \label{eq:nested-V-asymptotic}
\end{equation*}
Substitution into \eqref{eq:nested-product-quotient}, with
\(x_d\uparrow1/\rho\), gives
\begin{equation*}
 \frac{\mathcal Q_{d,\rho}(x)}a
 \longrightarrow
 \Phi_{d-1}(y_1,\ldots,y_{d-1}),
 \label{eq:nested-Phi-limit}
\end{equation*}
where
\begin{equation*}
 \Phi_r(y_1,\ldots,y_r)
 =\sum_{j=1}^r y_j^2\prod_{i<j}(1-2y_i)
 +\frac14\prod_{i=1}^r(1-2y_i),
 \qquad \Phi_0=\frac14.
 \label{eq:nested-Phi}
\end{equation*}
Separating the first coordinate gives
\begin{equation}
 \Phi_r(y_1,\ldots,y_r)
 =y_1^2+(1-2y_1)\Phi_{r-1}(y_2,\ldots,y_r).
 \label{eq:nested-Phi-recursion}
\end{equation}
Let
\[
 q_r=\inf_{(y_1,\ldots,y_r)\in[0,1/2)^r}
      \Phi_r(y_1,\ldots,y_r).
\]
Then $q_0=1/4$.  For fixed $y_1$, the factor $1-2y_1$ is strictly
positive.  Hence optimizing over the remaining coordinates
$(y_2,\ldots,y_r)$ in \eqref{eq:nested-Phi-recursion} leads to
\begin{equation*}
 \inf_{0\le y_1<1/2}\{y_1^2+(1-2y_1)q_{r-1}\}.
 \label{eq:nested-one-variable-minimization}
\end{equation*}
Inductively, \(0<q_{r-1}\le1/4<1/2\), so the unconstrained minimizer lies in
the allowed interval.  Its value at \(y=q_{r-1}\) is
\(q_{r-1}(1-q_{r-1})\).  This proves
\eqref{eq:nested-q-recursion}; the minimizing coordinates are
\begin{equation}
 (y_1,\ldots,y_r)=(q_{r-1},q_{r-2},\ldots,q_0).
 \label{eq:nested-Phi-minimizer}
\end{equation}
Together with \eqref{eq:nested-base-gap}, the limiting normalized quotient
is \(q_{d-1}=q_{m-2}\).  The case \(d=1\) follows from the
same construction rather than from an external two-state formula.  Here
\begin{equation*}
 \sum_{k\ge1}M_kz^k=ca^2\tau(z),\qquad
 \sum_{k\ge1}H_{1,k}z^k=sca\tau(z),\qquad B=ca,
 \label{eq:nested-d1-form}
\end{equation*}
and for \(f(k)=x^k\), \(k\ge1\), with \(f(0)=0\),
\begin{equation*}
 \mathcal Q_{1,\rho}(x)=\frac{1+s-2sx}{a}.
 \label{eq:nested-d1-quotient}
\end{equation*}
Letting \(x\uparrow1/\rho\) and then \(\rho\uparrow1\) gives
\begin{equation*}
 \lim_{x\uparrow1/\rho}\frac{\mathcal Q_{1,\rho}(x)}{a}
 =\frac{(1-\rho)^2}{(1-\rho^2)^2}
 =\frac1{(1+\rho)^2}\longrightarrow\frac14=q_0.
 \label{eq:nested-d1-limit}
\end{equation*}
As above, the critical product is approached by subcritical, finitely
supported tests before the finite-window diagonal is taken.

\subsection{Finite kernels and finite windows}

We now realize the limiting quotient by admissible finite kernels, finite
windows, and finitely supported tests.

The following reverse-order selection converts the iterated limiting form
into a single admissible sequence of finite kernels and finite windows.

\begin{proof}[Proof of Theorem~\ref{thm:exact-cardinality-upper}]
Fix \(m\), put \(d=m-1\), and let \(\delta_N\downarrow0\).  Make the
following choices in order.

\begin{enumerate}
\item[1)] Choose \(\rho_N\uparrow1\) and the shallow coordinates from
\eqref{eq:nested-Phi-minimizer} so that the critical product quotient,
divided by \(a_N=1-\rho_N^2\), is at most
\(q_{m-2}+\delta_N\).

\item[2)] Choose a strictly subcritical \(x_{d,N}<1/\rho_N\), and then a
finite rectangular truncation \(f_N\), so that its leading-form quotient
differs from the critical product quotient by at most
\(\delta_Na_N\), and so that \(\|f_N\|_M>0\).  Let
\begin{equation*}
 \Gamma_N^{\mathrm{lead}}
 =\frac{\cE_d(f_N)}{\|f_N\|_M^2},
 \label{eq:nested-diagonal-leading-quotient}
\end{equation*}
and let \(L_N\) exceed the total rare-count support of \(f_N\).

\item[3)] With \(\rho_N\) and \(f_N\) fixed, choose every
\(\varepsilon_{i,N}>0\) sufficiently small that the outer-rare,
large-window quotient
\begin{equation*}
 \Gamma_N^{\varepsilon}
 =\frac{D_{\varepsilon_N}(f_N)}{M_{\varepsilon_N}(f_N)}
 \label{eq:nested-diagonal-epsilon-quotient}
\end{equation*}
differs from \(\Gamma_N^{\mathrm{lead}}\) by at most
\(\delta_Na_N\), and that
\(\gamma_0(\varepsilon_N,\rho_N)\) differs from \(a_N\) by at most
\(\delta_Na_N\).  The first choice follows from
\eqref{eq:nested-epsilon-layer}; its denominator is positive for all
sufficiently small \(\varepsilon_N\).  The second follows from matrix
continuity in \eqref{eq:nested-base-gap}.  This step is absent when
\(d=1\), with \(\Gamma_N^\varepsilon=\Gamma_N^{\mathrm{lead}}\).

\item[4)] By \eqref{eq:nested-variance-coefficient} and
\eqref{eq:nested-energy-n-limit}, choose a finite \(n_N>L_N\), sufficiently
large that the test embeds, \(V_{n_N,\varepsilon_N}(f_N)>0\), and
\begin{equation*}
 \left|
  n_N\frac{D_{n_N,\varepsilon_N}(f_N)}
             {V_{n_N,\varepsilon_N}(f_N)}
  -\Gamma_N^\varepsilon
 \right|\le\delta_Na_N.
 \label{eq:nested-diagonal-window-choice}
\end{equation*}

\item[5)] With all preceding objects fixed, choose \(\eta_N>0\) sufficiently
small, and put
\[
 P_N=
 \begin{cases}
  P_{\eta_N,\varepsilon_N,\rho_N},&d\ge2,\\
  P_{\eta_N,\rho_N},&d=1.
 \end{cases}
\]
Let \(\mu_{N,n}\) and \(\widetilde P_{N,n}\) be, respectively, the
stationary count law and the induced count kernel for \(P_N\) and window
length \(n\).  The choice can be made so that
\begin{equation*}
 \left|
  n_N\frac{\cE_{\widetilde P_{N,n_N}}(f_N,f_N)}
             {\Var_{\mu_{N,n_N}}(f_N)}
  -n_N\frac{D_{n_N,\varepsilon_N}(f_N)}
              {V_{n_N,\varepsilon_N}(f_N)}
 \right|\le\delta_Na_N,
 \label{eq:nested-diagonal-eta-choice}
\end{equation*}
and $\Gap(P_N)$ differs from
$\gamma_0(\varepsilon_N,\rho_N)$ by at most $\delta_Na_N$.  The quotient
convergence follows from
\eqref{eq:nested-fixed-window-eta-limits}, whose limiting variance
coefficient is positive for this fixed \(n_N\); the gap convergence is
ordinary finite-matrix continuity.
\end{enumerate}

Define the actual scaled Rayleigh quotient
\begin{equation*}
 R_N=
 n_N\frac{\cE_{\widetilde P_{N,n_N}}(f_N,f_N)}
              {\Var_{\mu_{N,n_N}}(f_N)}.
 \label{eq:nested-RN-definition}
\end{equation*}
Thus every triple $(P_N,n_N,f_N)$ is admissible in the definition of
$c_m^\star$.  Since the projected gap is at most its Rayleigh quotient, the
preceding tolerances give, for all large \(N\),
\begin{equation}
 R_N\le a_N\{q_{m-2}+5\delta_N\},
 \qquad
 \Gap(P_N)\ge a_N(1-2\delta_N),
 \label{eq:nested-diagonal-errors}
\end{equation}
We thus have
\begin{equation*}
 \limsup_{N\to\infty}
 \frac{n_N\Gap(\widetilde P_{N,n_N})}{\Gap(P_N)}
 \le q_{m-2}.
 \label{eq:nested-finite-diagonal-limit}
\end{equation*}
This proves \eqref{eq:exact-cardinality-upper} and
\eqref{eq:nested-finite-witness}.

If entrywise positivity is desired, fix \(N\) and let \(\pi_N\) be the
stationary law of \(P_N\).  Replacing \(P_N\) by
\begin{equation*}
 (1-\zeta)P_N+\zeta\Pi_N,
 \qquad
 \Pi_N(x,y)=\pi_N(y),
 \label{eq:nested-positive-perturbation}
\end{equation*}
preserves reversibility and makes every entry positive.  For fixed
\(N,n_N\), the base gap and the $n_N$-scaled count Rayleigh quotient are
continuous in \(\zeta\).  Choose \(\zeta_N>0\) so that each changes by at
most \(\delta_Na_N\); the extra errors vanish in
\eqref{eq:nested-diagonal-errors}, so the same diagonal limit holds.

It remains to analyze the recursion.  Let \(Y_r=q_r^{-1}\).  Then
\begin{equation}
 Y_{r+1}=Y_r+1+\frac1{Y_r-1},
 \qquad
 Y_0=4.
 \label{eq:nested-Y-recursion}
\end{equation}
In particular, \(Y_r\ge r+4\), which gives
\(q_{m-2}\le(m+2)^{-1}\).  Summing
\eqref{eq:nested-Y-recursion} gives
\[
 Y_r=4+r+\sum_{k=0}^{r-1}\frac1{Y_k-1}.
\]
Since $Y_k\ge k+4$,
\[
 0\le\sum_{k=0}^{r-1}\frac1{Y_k-1}
 \le\sum_{k=0}^{r-1}\frac1{k+3}=O(\log r),
\]
and hence \(Y_r=r+O(\log r)\).  Consequently,
\begin{equation*}
 \frac1{Y_r-1}
 =\frac1r+O\!\left(\frac{\log r}{r^2}\right).
 \label{eq:nested-Y-increment-asymptotic}
\end{equation*}
Indeed,
\[
 \sum_{k=1}^{r-1}\frac1k=\log r+O(1),
 \qquad
 \sum_{k=1}^{\infty}\frac{\log k}{k^2}<\infty.
\]
Thus the error is summable, and a second summation yields
\begin{equation*}
 Y_r=r+\log r+O(1),
 \qquad
 q_r=\frac1{r+\log r+O(1)}.
 \label{eq:nested-q-final-asymptotic}
\end{equation*}
This proves \eqref{eq:nested-q-asymptotic} and the theorem.
\end{proof}

\begin{remark}[Scope of the construction]
\label{rem:nested-one-sided}
The construction establishes $c_m^\star\le q_{m-2}$.  For $m=2$, the
reverse inequality is proved independently in
Appendix~\ref{app:binary:section}; for $m\ge3$, it remains open.
The optimal linear order follows after combining this obstruction with
Theorem~\ref{thm:main-body}.
\end{remark}

\section{Conclusion}
\label{sec:discussion}

In Theorem~\ref{thm:intro-main}  we have established
$c_m^\star=\Theta(m^{-1})$, or equivalently
$C_m^{\mathrm{opt}}=\Theta(m)$. Therefore, the lower and upper bounds  have
the same state-space exponent, while the leading constant and lower-order asymptotics remain open.
The kernel $\widetilde P_n$ is the projected count kernel defined by the stationary edge law in \eqref{eq:count-kernel}. The full process $(K_t)$
does not need to be Markov, and the theorem does not assert strong lumpability of the hidden path chain.

All spectral gaps in the theorem are right spectral gaps.  Negative
eigenvalues are favorable in the short-window endpoint identity, so the
lower bound does not require laziness.  This differs from
total-variation mixing, where periodicity must be treated separately~\cite{XiangXinZhang2026}.

The numerical factor $1080$ is not optimized.  The linear dependence on
$m$ comes from averaging the single-anchor inequalities before estimating
the endpoint sectors.  This cancels the Palm normalization and prevents an
additional factor of $m$ when the window energies are summed over anchor labels.

When $m=2$, the count vector is determined by one coordinate and the
projected count kernel is birth--death. Thus, its gap can be studied
through tail masses, edge conductances, and discrete Hardy inequalities.
Appendix~\ref{app:binary:section} uses this one-dimensional structure
to determine the exact comparison constant.  The argument does not extend
directly to a count simplex of dimension at least two.

In particular, Theorem~\ref{thm:intro-binary} shows that
\[
 c_2^\star=\frac14,
 \qquad C_2^{\mathrm{opt}}=4.
\]
The proof combines an annealed refresh-block Poincar\'e inequality and a
Palm conductance comparison in the persistent regime with an edgewise
Stein comparison in the antipersistent regime.  A finite-word continuity
argument extends the bounds to zero diagonal entries on the actual count
support, including the deterministic alternating chain.  The matching
upper bound is the two-state case of the nested construction in
Theorem~\ref{thm:exact-cardinality-upper}; the exact value is an infimum
over kernels and windows, not an assertion of attainment by a fixed
interior kernel.

\appendix
\section{The exact two-state constant}
\label{app:binary:section}

The two-state instance of the optimal normalized comparison constant
can be determined exactly. The proof below includes both signs of the
nontrivial base eigenvalue and then passes to kernels with zero entries.
Throughout this appendix the one-point gap convention is the convention
of Section~\ref{sec:preliminaries}.

\begin{theorem}[Sharp binary comparison]
\label{app:binary:main}
Let
\[
 P=\begin{pmatrix}1-b&b\\ c&1-c\end{pmatrix},
 \qquad 0<b,c\le1,\qquad \gamma=b+c,
\]
and let $\widetilde P_n$ be the actual stationary endpoint-count kernel.
For every $n\ge2$,
\begin{equation}
 \Gap(\widetilde P_n)\ge\frac{\gamma}{4n},
 \qquad c_2^\star=\frac14.
 \label{app:binary:sharp}
\end{equation}
More precisely, the following pointwise lower bounds hold:
\begin{equation}
 \Gap(\widetilde P_n)\ge
 \begin{cases}
 \displaystyle\frac{\gamma}
 {n(1+\sqrt{1-\gamma})^2},&0<\gamma\le1,\\[6pt]
 \displaystyle\frac{\gamma}{n},&1\le\gamma\le2.
 \end{cases}
 \label{app:binary:regimes}
\end{equation}
\end{theorem}

We first prove the lower bounds for $0<b,c<1$. All displayed word
probabilities then have their full finite support. The final subsection
extends the inequalities to $b=1$ or $c=1$ on the actual count support
and combines them with the upper bound already proved in
Theorem~\ref{thm:exact-cardinality-upper}.

\subsection{Setup and the projected count kernel}
\label{app:binary:sec:setup}

Let
\[
P=\begin{pmatrix}1-b&b\\ c&1-c\end{pmatrix},
\qquad 0<b,c<1,
\qquad \gamma=b+c.
\]
Its stationary law and nontrivial eigenvalue are
\[
\pi_0=\frac c\gamma,
\qquad
\pi_1=\frac b\gamma,
\qquad
\rho=1-\gamma.
\]
We use the right, or Poincare, spectral gap.  Thus
$\Gap(P)=1-\rho=\gamma$, including the antipersistent case $\rho<0$.

Let $(X_t)_{t\in\mathbb Z}$ be the stationary $P$-chain and define
\[
N_t^{(n)}=\sum_{r=1}^{n}X_{t+r},
\qquad
\nu_n(k)=\Pp(N_0^{(n)}=k).
\]
The stationary one-step projected kernel is defined by
\begin{equation}
\nu_n(k)\tP_n(k,\ell)
=
\Pp(N_0^{(n)}=k,N_1^{(n)}=\ell).
\label{app:binary:eq:2-1}
\end{equation}
This definition does not assert that $(N_t^{(n)})_t$ is Markov.  It retains
exactly its stationary one-step law.

\begin{lemma}[Birth--death structure and reversibility]
\label{app:binary:lem:reversible}
The kernel $\tP_n$ is an irreducible reversible birth--death kernel on
$\{0,\ldots,n\}$ with stationary law $\nu_n$.
\end{lemma}

\begin{proof}
The identity
\[
N_1^{(n)}-N_0^{(n)}=X_{n+1}-X_1
\]
shows that the projected move belongs to $\{-1,0,1\}$.  Since $b,c>0$, every
count communicates with its neighbors through a path of positive stationary
probability, so the kernel is irreducible.  Finally, the stationary path law
of $(X_1,\ldots,X_{n+1})$ is invariant under reversal.  Reversal exchanges
the two overlapping window counts.  Hence the pair
$(N_0^{(n)},N_1^{(n)})$ is exchangeable, which is detailed balance in
\eqref{app:binary:eq:2-1}.
\end{proof}

Write
\[
C_j=\nu_n(j)\tP_n(j,j+1),
\qquad 0\le j<n,
\]
for the edge conductances.  The projected Dirichlet form is
\begin{equation}
\cE_n(f,f)
=
\sum_{j=0}^{n-1}C_j\{f(j+1)-f(j)\}^2,
\label{app:binary:eq:2-2}
\end{equation}
For the count-vector function $F(k_0,k_1)=f(k_1)$, this is precisely
the physical form $\cD_1(F)$ of the main text. The separate notation
$\mathcal D^{\mathrm{blk}}_n$ introduced below denotes an auxiliary
refresh-block form and is not this physical form. Thus
\begin{equation}
\Gap(\tP_n)
=
\inf_{\Var_{\nu_n}(f)>0}
\frac{\cE_n(f,f)}{\Var_{\nu_n}(f)}.
\label{app:binary:eq:2-3}
\end{equation}
The conductance also has the endpoint-event representation
\begin{equation}
C_j
=
\Pp(N_0^{(n)}=j,X_1=0,X_{n+1}=1).
\label{app:binary:eq:2-4}
\end{equation}
Indeed, the count increases exactly on the endpoint event in \eqref{app:binary:eq:2-4}.

\subsection{The Palm edge profile}
\label{app:binary:sec:palm}

Throughout Subsections~\ref{app:binary:sec:palm}--\ref{app:binary:sec:block-kernel}, assume
$0<\gamma\le1$, so $0\le\rho<1$.  In this regime
\begin{equation}
P=\rho I+\gamma\Pi_\pi,
\qquad
\Pi_\pi(x,y)=\pi_y.
\label{app:binary:eq:3-1}
\end{equation}
Thus a transition copies the current state with probability $\rho$ and
refreshes from $\pi$ with probability $\gamma$.

Use the same refresh indicators and refresh colors to couple two length-$m$
paths started from $0$ and $1$.  The paths are ordered and coalesce at their
first refresh.  Let $L_m\le U_m$ be their occupation counts, and define
\begin{equation}
q_{m,j}=\Pp(L_m\le j<U_m),
\qquad 0\le j<m.
\label{app:binary:eq:3-2}
\end{equation}
We call $q_m=(q_{m,0},\ldots,q_{m,m-1})$ the Palm edge profile.  It is the
edge-incidence profile of the first refresh block.

Set
\[
\alpha=1-b=\pi_0+\rho\pi_1,
\qquad
\delta=1-c=\pi_1+\rho\pi_0,
\]
\begin{equation}
\beta=b=\gamma\pi_1,
\qquad
\chi=c=\gamma\pi_0.
\label{app:binary:eq:3-3}
\end{equation}
Then
\begin{equation}
\alpha\delta-\beta\chi=\rho.
\label{app:binary:eq:3-4}
\end{equation}

\begin{lemma}[Palm and stationary generating functions]
\label{app:binary:lem:generating}
Let
\[
H_m(z)=\sum_{j=0}^{m-1}q_{m,j}z^j,
\qquad
V_m(z)=\E z^{N_0^{(m)}},
\qquad V_0(z)=1.
\]
Then
\begin{equation}
\mathsf H(t,z)
:=
\sum_{m\ge1}H_m(z)t^{m-1}
=
\frac{1}{1-(\alpha+\delta z)t+\rho zt^2},
\label{app:binary:eq:3-5}
\end{equation}
and
\begin{equation}
\mathsf V(t,z)
:=
\sum_{m\ge0}V_m(z)t^m
=
\{1-\rho(\pi_1+\pi_0z)t\}\mathsf H(t,z).
\label{app:binary:eq:3-6}
\end{equation}
Moreover, if
\begin{equation}
y=\frac{\beta\chi}{\alpha\delta},
\qquad
R_{r,s}(y)=\sum_{h\ge0}\binom rh\binom sh y^h,
\label{app:binary:eq:3-7}
\end{equation}
then, with $N=m-1$,
\begin{equation}
q_{m,j}
=
\alpha^{N-j}\delta^jR_{N-j,j}(y).
\label{app:binary:eq:3-8}
\end{equation}
\end{lemma}

\begin{proof}
Put $D(z)=\operatorname{diag}(1,z)$ and let
\[
\Phi_m^{(a)}(z)=z^a e_a^{\mathsf T}\{PD(z)\}^{m-1}\mathbf 1,
\qquad a\in\{0,1\}.
\]
The common-refresh coupling gives
\[
H_m(z)=\frac{\Phi_m^{(0)}(z)-\Phi_m^{(1)}(z)}{1-z}.
\]
Since $PD(z)$ has trace $\alpha+\delta z$ and determinant $\rho z$,
Cayley--Hamilton gives the Palm recursion
\[
H_1(z)=1,
\qquad
H_2(z)=\alpha+\delta z,
\]
\begin{equation}
H_{m+2}(z)
=(\alpha+\delta z)H_{m+1}(z)-\rho zH_m(z).
\label{app:binary:eq:3-9}
\end{equation}
This yields \eqref{app:binary:eq:3-5}.  For $m\ge1$,
\[
V_m(z)=(\pi_0,\pi_1z)\{PD(z)\}^{m-1}\mathbf 1,
\]
so the stationary count polynomials satisfy the same second-order
recursion.  Since
\[
V_0(z)=1,
\qquad
V_1(z)=\pi_0+\pi_1z
=H_2(z)-\rho(\pi_1+\pi_0z),
\]
their generating function is \eqref{app:binary:eq:3-6}.

For completeness, factor the denominator in \eqref{app:binary:eq:3-5} as
\[
(1-\alpha t)(1-\delta zt)-\beta\chi zt^2.
\]
Expanding geometrically in the last term and then expanding the two
negative powers gives
\[
\mathsf H(t,z)
=
\sum_{h\ge0}
\frac{(\beta\chi zt^2)^h}
{(1-\alpha t)^{h+1}(1-\delta zt)^{h+1}}.
\]
Extracting the coefficient of $t^{m-1}z^j$ gives \eqref{app:binary:eq:3-8}.
\end{proof}

\begin{lemma}[Palm conductance identity]
\label{app:binary:lem:palm-conductance}
For every $0\le j<n$,
\begin{equation}
C_j=\gamma\pi_0\pi_1q_{n,j}.
\label{app:binary:eq:3-10}
\end{equation}
\end{lemma}

\begin{proof}
Let $D(z)=\operatorname{diag}(1,z)$ and define
\[
\mathsf C_n(z)=\sum_{j=0}^{n-1}C_jz^j.
\]
The endpoint formula \eqref{app:binary:eq:2-4} gives
\begin{equation}
\mathsf C_n(z)
=
\pi_0e_0^{\mathsf T}\{PD(z)\}^{n-1}Pe_1.
\label{app:binary:eq:3-11}
\end{equation}
The matrix $PD(z)$ has trace $\alpha+\delta z$ and determinant $\rho z$.
By Cayley--Hamilton, the polynomials $\mathsf C_n(z)$ satisfy the recursion
\eqref{app:binary:eq:3-9}.  The initial values are
\[
\mathsf C_1(z)=\pi_0b=\gamma\pi_0\pi_1,
\]
\[
\mathsf C_2(z)
=
\pi_0b\{(1-b)+(1-c)z\}
=
\gamma\pi_0\pi_1(\alpha+\delta z).
\]
Comparison with $H_1,H_2$ proves
$\mathsf C_n=\gamma\pi_0\pi_1H_n$, coefficient by coefficient.
\end{proof}

\subsection{An annealed refresh-block Poincare inequality}
\label{app:binary:sec:annealed}

Put an independent cut between consecutive sites with probability $\gamma$.
The cuts partition $\{1,\ldots,n\}$ into interval blocks
$\cB_n$.  Give the blocks independent Bernoulli$(\pi_1)$ colors.  The
resulting colored path has exactly the stationary $P$-path law in \eqref{app:binary:eq:3-1}.
For a block $B$, write $X^{B\leftarrow a}$ for the path obtained by setting
its color to $a\in\{0,1\}$.

\begin{theorem}[Annealed block Poincare inequality]
\label{app:binary:thm:annealed}
Let $u=\pi_0\pi_1$.  For every function
$F:\{0,1\}^n\to\mathbb R$,
\begin{equation}
\Var\{F(X_1,\ldots,X_n)\}
\le
\mathcal D^{\mathrm{blk}}_n(F),
\label{app:binary:eq:4-1}
\end{equation}
where
\begin{equation}
\mathcal D^{\mathrm{blk}}_n(F)
:=
u\,\E_{\cB_n}\sum_{B\in\cB_n}
\E_{-B}
\bigl[F(X^{B\leftarrow1})-F(X^{B\leftarrow0})\bigr]^2.
\label{app:binary:eq:4-2}
\end{equation}
Here $\E_{-B}$ averages all block colors other than that of $B$.
\end{theorem}

\begin{proof}
We induct on $n$.  For $n=1$, \eqref{app:binary:eq:4-1} is the two-point variance identity.
Suppose the result holds at length $n$, and let $F$ be a function on paths
of length $n+1$.

Append the last site using an independent cut
$C\sim\operatorname{Bernoulli}(\gamma)$ and an independent refresh color
$Z\sim\operatorname{Bernoulli}(\pi_1)$.  If $C=0$, site $n+1$ joins the
last old block and has color $X_n$; if $C=1$, it is a new singleton of color
$Z$.  Define
\begin{equation}
H(x_1,\ldots,x_n)
=
\E_{C,Z}F(x_1,\ldots,x_n,X_{n+1}).
\label{app:binary:eq:4-3}
\end{equation}
Equivalently,
\begin{equation}
H(x)
=
\rho F(x,x_n)
+\gamma\{\pi_0F(x,0)+\pi_1F(x,1)\}.
\label{app:binary:eq:4-4}
\end{equation}
Total variance and the induction hypothesis give
\begin{equation}
\Var_{n+1}(F)
\le
\E\Var_{C,Z}(F\mid X_1,\ldots,X_n)
+\mathcal D^{\mathrm{blk}}_n(H).
\label{app:binary:eq:4-5}
\end{equation}

Fix the old skeleton and let $B_*$ be its last block.  For a block
$A\ne B_*$, use the same pair $(C,Z)$ in the two recolorings of $A$.
Changing $A$ leaves the old terminal color unchanged, so Jensen's inequality
and \eqref{app:binary:eq:4-3} give
\begin{equation}
\bigl[H(X^{A\leftarrow1})-H(X^{A\leftarrow0})\bigr]^2
\le
\E_{C,Z}
\bigl[F(\widehat X^{A\leftarrow1})
-F(\widehat X^{A\leftarrow0})\bigr]^2.
\label{app:binary:eq:4-6}
\end{equation}
Here $\widehat X$ denotes the appended length-$(n+1)$ path constructed with
that common pair $(C,Z)$.
Thus every nonterminal term in $\mathcal D^{\mathrm{blk}}_n(H)$ is absorbed by its counterpart in
$\mathcal D^{\mathrm{blk}}_{n+1}(F)$.

It remains to combine the old terminal-block term with the conditional
variance in \eqref{app:binary:eq:4-5}.  Fix all colors outside $B_*$ and write
\[
x=F(0,0),\quad y=F(1,0),\quad
w=F(0,1),\quad t=F(1,1),
\]
\begin{equation}
\Delta=t-y-w+x.
\label{app:binary:eq:4-7}
\end{equation}
The displayed arguments are the old terminal-block color and the appended
site color.  Let $Y\sim\operatorname{Bernoulli}(\pi_1)$ be the old terminal
color.  When $C=0$, retain $Z$ as an independent dummy variable.  If
$H_0,H_1$ are the two values of \eqref{app:binary:eq:4-4}, then
\begin{equation}
\E_Y\Var_{C,Z}(F\mid Y)+u(H_1-H_0)^2
=
\Var_{C,Y,Z}(F).
\label{app:binary:eq:4-8}
\end{equation}

The contribution to the new block energy from the terminal region is
\begin{equation}
\begin{aligned}
\mathcal D^{\mathrm{blk}}_{\mathrm{int}}
=u\bigl[&\rho(t-x)^2\\
&+\gamma\{\pi_0(y-x)^2+\pi_1(t-w)^2
+\pi_0(w-x)^2+\pi_1(t-y)^2\}\bigr].
\end{aligned}
\label{app:binary:eq:4-9}
\end{equation}
The four values $x,y,w,t$ occur with respective probabilities
$\pi_0(\rho+\gamma\pi_0)$, $\gamma u$,
$\gamma u$, and $\pi_1(\rho+\gamma\pi_1)$.
Expanding the variance with these four weights gives the exact identity
\begin{equation}
\mathcal D^{\mathrm{blk}}_{\mathrm{int}}-\Var_{C,Y,Z}(F)
=
\gamma^2u^2\Delta^2\ge0.
\label{app:binary:eq:4-10}
\end{equation}
Equations \eqref{app:binary:eq:4-8}--\eqref{app:binary:eq:4-10} absorb the last-block term and the conditional
variance in \eqref{app:binary:eq:4-5}.  Equation \eqref{app:binary:eq:4-6} handles all other blocks.  Averaging the
old skeleton and colors proves
\[
\E\Var_{C,Z}(F\mid X_1,\ldots,X_n)+\mathcal D^{\mathrm{blk}}_n(H)
\le\mathcal D^{\mathrm{blk}}_{n+1}(F),
\]
which closes the induction.
\end{proof}

\begin{remark}
Applying product Efron--Stein only after conditioning on $\cB_n$ controls
$\E\Var(F\mid\cB_n)$ but misses
$\Var\{\E(F\mid\cB_n)\}$.  The latter is carried by
$\Var\{H(X_1,\ldots,X_n)\}$ in the induction above.  This is the reason for
the append-one-site proof.
\end{remark}

\subsection{Palm convolution and the diagonal block estimate}
\label{app:binary:sec:palm-convolution}

For polynomials with real coefficients, write $A\preceq_{\rm coef}B$ when
every coefficient of $B-A$ is nonnegative.  The next lemma is the only place
where we use monotonicity in the persistence parameter.  Throughout its
proof, $\pi$ is held fixed while $\rho$ varies.

\begin{lemma}[Coefficientwise Palm convolution]
\label{app:binary:lem:palm-convolution}
For every $m\ge1$,
\begin{equation}
\sum_{r=0}^{m-1}V_r(z)H_{m-r}(z)
\preceq_{\rm coef}
mH_m(z).
\label{app:binary:eq:5-1}
\end{equation}
\end{lemma}

\begin{proof}
Differentiate the first identity in Lemma~\ref{app:binary:lem:generating} with respect
to $t$.  Since
\[
\partial_\rho\{1-(\alpha+\delta z)t+\rho zt^2\}
=-(\pi_1+\pi_0z)t+zt^2,
\]
identities \eqref{app:binary:eq:3-5}--\eqref{app:binary:eq:3-6} give
\begin{equation}
\partial_t\{t\mathsf H(t,z)\}
-\mathsf V(t,z)\mathsf H(t,z)
=\rho\,\partial_\rho\mathsf H(t,z).
\label{app:binary:eq:5-2}
\end{equation}
The coefficient of $t^{m-1}$ on the left is the right side of \eqref{app:binary:eq:5-1}
minus its left side.  It is therefore enough to prove
\begin{equation}
\partial_\rho q_{m,j}\ge0,
\qquad 0\le j<m.
\label{app:binary:eq:5-3}
\end{equation}

We prove \eqref{app:binary:eq:5-3} from the positive representation \eqref{app:binary:eq:3-8}.  The case $m=1$ is
immediate.  For $m\ge2$, put
\[
N=m-1,\qquad r=N-j,\qquad s=j,\qquad
y=\frac{\beta\chi}{\alpha\delta}.
\]
By \eqref{app:binary:eq:3-4}, $0\le y\le1$.  Normalize the summands of
$R_{r,s}(y)$ to a probability distribution:
\begin{equation}
\Pp(Z=h)
=\frac{\binom rh\binom sh y^h}{R_{r,s}(y)},
\qquad M:=\E Z.
\label{app:binary:eq:5-4}
\end{equation}
The cases $r=0$ or $s=0$ follow directly, so suppose that $r,s>0$.
Using
\[
(1-\rho)\alpha'=\beta,\quad
(1-\rho)\delta'=\chi,\quad
\alpha+\beta=\delta+\chi=1,
\]
and $yR'_{r,s}(y)/R_{r,s}(y)=M$, logarithmic differentiation of
\eqref{app:binary:eq:3-8} gives
\begin{equation}
(1-\rho)\alpha\delta\,
\partial_\rho\log q_{m,j}
=\mathcal A-\mathcal B M,
\label{app:binary:eq:5-5}
\end{equation}
where
\begin{equation}
\mathcal A=r\beta\delta+s\chi\alpha,
\qquad
\mathcal B=\alpha+\delta.
\label{app:binary:eq:5-6}
\end{equation}
Indeed, the derivative of $y$ contributes
\[
\alpha\delta(1-\rho)\partial_\rho\log y
=-\{2\alpha\delta+\beta\delta+\chi\alpha\}
=-(\alpha+\delta).
\]

The successive-weight identity in \eqref{app:binary:eq:5-4} is
\begin{equation}
\E Z^2=y\,\E\{(r-Z)(s-Z)\}.
\label{app:binary:eq:5-7}
\end{equation}
For
\[
h(x)=x^2-y(r-x)(s-x),
\]
equation \eqref{app:binary:eq:5-7} implies
\begin{equation}
h(M)=(y-1)\Var(Z)\le0.
\label{app:binary:eq:5-8}
\end{equation}
Moreover,
\begin{equation}
h'(x)=2(1-y)x+y(r+s)\ge0,
\qquad x\ge0.
\label{app:binary:eq:5-9}
\end{equation}
Let $x_0=\mathcal A/\mathcal B$.  If
$x_0\ge\min(r,s)$, then $M\le\min(r,s)\le x_0$.  Otherwise a direct
expansion, using $\alpha+\beta=\delta+\chi=1$, yields
\begin{equation}
\mathcal B^2h(x_0)=(\beta r-\chi s)^2\ge0.
\label{app:binary:eq:5-10}
\end{equation}
Equations \eqref{app:binary:eq:5-8}--\eqref{app:binary:eq:5-10} again imply $M\le x_0$.  Hence the right side of
\eqref{app:binary:eq:5-5} is nonnegative.  This proves \eqref{app:binary:eq:5-3}, and coefficient extraction in
\eqref{app:binary:eq:5-2} proves the lemma.
\end{proof}

We now specialize Theorem~\ref{app:binary:thm:annealed} to count functions.  For a
refresh-skeleton block $B$, let $\ell_B$ be its length and let $S_B$ be the
number of ones in all blocks other than $B$.  Thus recoloring $B$ changes
the total count from $S_B$ to $S_B+\ell_B$.  Define the $n$-by-$n$ matrix
\begin{equation}
K^{(n)}_{ij}
=
\E\sum_{B\in\cB_n}
\boldsymbol 1_{\{S_B\le i<S_B+\ell_B\}}
\boldsymbol 1_{\{S_B\le j<S_B+\ell_B\}},
\quad 0\le i,j<n.
\label{app:binary:eq:5-11}
\end{equation}
The expectation in \eqref{app:binary:eq:5-11} includes both the random skeleton and all block
colors.  If $v_B$ is the edge-incidence vector of
$[S_B,S_B+\ell_B)$, then
\begin{equation}
K^{(n)}=\E\sum_{B\in\cB_n}v_Bv_B^{\mathsf T},
\label{app:binary:eq:5-12}
\end{equation}
so $K^{(n)}$ is symmetric, entrywise nonnegative, and positive
semidefinite.

\begin{corollary}[Block form and diagonal estimate]
\label{app:binary:cor:block-diagonal}
Let $g_i=f(i+1)-f(i)$.  Then
\begin{equation}
\Var_{\nu_n}(f)
\le
\pi_0\pi_1\,g^{\mathsf T}K^{(n)}g.
\label{app:binary:eq:5-13}
\end{equation}
Moreover, for $0\le i<m$,
\begin{equation}
K^{(m)}_{ii}
\le
\{1+\gamma(m-1)\}q_{m,i}.
\label{app:binary:eq:5-14}
\end{equation}
\end{corollary}

\begin{proof}
For a fixed block,
\[
f(S_B+\ell_B)-f(S_B)
=\sum_{i=S_B}^{S_B+\ell_B-1}g_i.
\]
Substitution in \eqref{app:binary:eq:4-1} gives \eqref{app:binary:eq:5-13}.

For \eqref{app:binary:eq:5-14}, distinguish the first block from blocks preceded by a cut.
The first block has diagonal edge profile $q_m$.  If a block begins after a
prefix of length $r$, the preceding edge is a cut with probability
$\gamma$; conditionally, the prefix count has generating polynomial $V_r$
and the marked suffix block has profile $H_{m-r}$.  Consequently, exactly
coefficient by coefficient,
\begin{equation}
K^{(m)}_{ii}
=
[z^i]\left\{H_m(z)
+\gamma\sum_{r=1}^{m-1}V_r(z)H_{m-r}(z)\right\}.
\label{app:binary:eq:5-15}
\end{equation}
Writing the braces as
$\rho H_m+\gamma\sum_{r=0}^{m-1}V_rH_{m-r}$ and applying
Lemma~\ref{app:binary:lem:palm-convolution} proves \eqref{app:binary:eq:5-14}.
\end{proof}

\subsection{The sharp block-kernel estimate}
\label{app:binary:sec:block-kernel}

The diagonal estimate alone does not control the spectral norm of
$K^{(n)}$.  The missing information is supplied by an exact deletion
identity.  Recall that blocks are blocks of the refresh skeleton: adjacent
blocks remain distinct even when they receive the same color.

\begin{lemma}[Marked-block deletion]
\label{app:binary:lem:deletion}
For $0\le i\le i+h<n$,
\begin{equation}
K^{(n)}_{i,i+h}=\rho^hK^{(n-h)}_{ii}.
\label{app:binary:eq:6-1}
\end{equation}
\end{lemma}

\begin{proof}
Expand the expectation in \eqref{app:binary:eq:5-11} as a sum over colored skeletons and a
marked block $B$.  A term contributes to the left side precisely when
\begin{equation}
S_B\le i,\qquad S_B+\ell_B>i+h.
\label{app:binary:eq:6-2}
\end{equation}
Delete the rightmost $h$ physical sites of $B$, shifting every later block
to the left.  The marked block is still nonempty by \eqref{app:binary:eq:6-2}, its color and
all other blocks and colors are unchanged, and the number $S_B$ of ones
outside it is unchanged.  The image therefore has a marked block crossing
the diagonal edge $i$ in a length-$(n-h)$ configuration.

Conversely, insert $h$ sites at the physical right end of a marked block
crossing edge $i$, with no cuts between the inserted sites or between the
old endpoint and the first inserted site.  This uniquely recovers \eqref{app:binary:eq:6-2},
including when the marked block touches the right boundary.  Thus there is
no positional multiplicity.  The deleted part contains exactly $h$
no-cut skeleton edges, so its probability weight is $\rho^h$; the block
color weights are unchanged.  This proves \eqref{app:binary:eq:6-1}.  When $\rho=0$ and
$h>0$, both sides vanish because every block is a singleton; $h=0$ is
tautological.
\end{proof}

\begin{lemma}[Cross-length Palm comparison]
\label{app:binary:lem:cross-palm}
For $0\le i\le i+h<n$,
\begin{equation}
q_{n,i}q_{n,i+h}
\ge
\rho^h q_{n-h,i}^2.
\label{app:binary:eq:6-3}
\end{equation}
\end{lemma}

\begin{proof}
Set $v=n-h-1-i$.  Formula \eqref{app:binary:eq:3-8} gives
\begin{equation}
\frac{q_{n,i}q_{n,i+h}}{q_{n-h,i}^2}
=(\alpha\delta)^h
\frac{R_{v+h,i}(y)R_{v,i+h}(y)}{R_{v,i}(y)^2}.
\label{app:binary:eq:6-4}
\end{equation}
The polynomials $R_{r,s}(y)$ are nondecreasing in either integer index for
$y\ge0$.  The last fraction is therefore at least one.  Finally,
$\alpha\delta=\rho+\beta\chi\ge\rho$, proving \eqref{app:binary:eq:6-3}.  All denominators in
\eqref{app:binary:eq:6-4} are positive because $0<b,c<1$.
\end{proof}

\begin{theorem}[Sharp persistent block bound]
\label{app:binary:thm:persistent-block}
If $0<\gamma\le1$, then
\begin{equation}
K^{(n)}
\preceq
n(1+\sqrt\rho)^2\,\operatorname{diag}(q_{n,0},\ldots,q_{n,n-1}).
\label{app:binary:eq:6-5}
\end{equation}
Consequently,
\begin{equation}
\Gap(\tP_n)\ge\frac{\gamma}{n(1+\sqrt\rho)^2}\ge\frac{\gamma}{4n}.
\label{app:binary:eq:6-6}
\end{equation}
\end{theorem}

\begin{proof}
First suppose $0<\rho<1$.  For $i\le i+h$, Lemma~\ref{app:binary:lem:deletion},
Corollary~\ref{app:binary:cor:block-diagonal}, and Lemma~\ref{app:binary:lem:cross-palm} give
\begin{equation}
\begin{aligned}
K^{(n)}_{i,i+h}
&\le
\rho^h\{1+\gamma(n-h-1)\}q_{n-h,i}\\
&\le
\rho^{h/2}\{1+\gamma(n-h-1)\}
\sqrt{q_{n,i}q_{n,i+h}}.
\end{aligned}
\label{app:binary:eq:6-7}
\end{equation}
Every $q_{n,i}$ is positive by \eqref{app:binary:eq:3-8}.  Put
$D_q=\operatorname{diag}(q_{n,0},\ldots,q_{n,n-1})$ and
$s=\sqrt\rho$.  In a fixed row of
$D_q^{-1/2}K^{(n)}D_q^{-1/2}$ there are at most two entries at each
positive distance.  Its row sum is therefore at most
\begin{equation}
\begin{aligned}
R_n
&=1+\gamma(n-1)
+2\sum_{h=1}^{n-1}s^h\{1+\gamma(n-h-1)\}\\
&=n(1+s)^2-s^2-\frac{2s(1-s^n)}{1-s}\\
&\le n(1+s)^2\le4n.
\end{aligned}
\label{app:binary:eq:6-8}
\end{equation}
The normalized matrix is symmetric and entrywise nonnegative.  The Schur
row-sum bound gives operator norm at most $n(1+s)^2$, and \eqref{app:binary:eq:5-12} turns this into
the Loewner inequality \eqref{app:binary:eq:6-5}.

If $\rho=0$, every refresh block is a singleton.  Thus $K^{(n)}$ is
diagonal, and \eqref{app:binary:eq:5-14} with $\gamma=1$ directly implies \eqref{app:binary:eq:6-5}.  This treats $\gamma=1$
without dividing by a zero power of $\rho$.

Finally, \eqref{app:binary:eq:5-13}, \eqref{app:binary:eq:6-5}, and the conductance identity \eqref{app:binary:eq:3-10} yield
\[
\Var_{\nu_n}(f)
\le n(1+\sqrt\rho)^2\pi_0\pi_1\sum_{i=0}^{n-1}q_{n,i}g_i^2
=\frac{n(1+\sqrt\rho)^2}{\gamma}\cE_n(f,f).
\]
Taking the infimum in \eqref{app:binary:eq:2-3} proves \eqref{app:binary:eq:6-6}.
\end{proof}

\subsection{The antipersistent regime}
\label{app:binary:sec:anti}

The refresh representation is unavailable when $\gamma>1$.  In that
range a direct edgewise comparison is stronger than the estimate required
for Theorem~\ref{app:binary:main}.

Let $\nu$ be a probability measure on $\{0,\ldots,n\}$ with mean
$\bar k$, and define its Stein conductances by
\begin{equation}
A_j
=\sum_{r=0}^{j}(\bar k-r)\nu(r)
=\sum_{r=j+1}^{n}(r-\bar k)\nu(r),
\qquad 0\le j<n.
\label{app:binary:eq:7-1}
\end{equation}

\begin{lemma}[Independent-copy Poincare inequality]
\label{app:binary:lem:stein-poincare}
For every $f:\{0,\ldots,n\}\to\mathbb R$,
\begin{equation}
\Var_\nu(f)
\le
\sum_{j=0}^{n-1}A_j\{f(j+1)-f(j)\}^2.
\label{app:binary:eq:7-2}
\end{equation}
\end{lemma}

\begin{proof}
Let $X,Y$ be independent with law $\nu$ and write
$g_j=f(j+1)-f(j)$.  Symmetry gives
\[
\Var_\nu(f)
=\E\bigl[(f(Y)-f(X))^2\boldsymbol 1_{\{X<Y\}}\bigr].
\]
On $X<Y$, Cauchy--Schwarz along the integer interval gives
\[
(f(Y)-f(X))^2
\le(Y-X)\sum_{X\le j<Y}g_j^2.
\]
If $F_j=\Pp(X\le j)$ and
$L_j=\E[X\boldsymbol 1_{\{X\le j\}}]$, independence gives
\[
\E[(Y-X)\boldsymbol 1_{\{X\le j<Y\}}]
=F_j(\bar k-L_j)-(1-F_j)L_j
=A_j.
\]
Summing over $j$ proves \eqref{app:binary:eq:7-2}.
\end{proof}

\begin{theorem}[Antipersistent Stein comparison]
\label{app:binary:thm:anti-stein}
If $1\le\gamma<2$, then, with $\bar k=n\pi_1$,
\begin{equation}
C_j\ge\frac{\gamma}{n}A_j,
\qquad 0\le j<n.
\label{app:binary:eq:7-3}
\end{equation}
Consequently,
\begin{equation}
\Gap(\tP_n)\ge\frac{\gamma}{n}.
\label{app:binary:eq:7-4}
\end{equation}
\end{theorem}

\begin{proof}
Put
\begin{equation}
u=1-b,\qquad v=1-c,\qquad s=\gamma-1=1-u-v.
\label{app:binary:eq:7-5}
\end{equation}
Under the present assumptions $u,v,s\ge0$.  Write
$N=X_1+\cdots+X_n$ and define
\begin{equation}
H_j
=\E\left[\left(\frac{N}{n}-X_1\right)
\boldsymbol 1_{\{N\le j\}}\right].
\label{app:binary:eq:7-6}
\end{equation}
We first establish
\begin{equation}
C_j-\frac{\gamma}{n}A_j=\gamma H_j.
\label{app:binary:eq:7-7}
\end{equation}
Set $C_{-1}=0$.  By reversibility,
\[
C_{j-1}=\Pp(N=j,X_1=1,X_{n+1}=0).
\]
The upward and downward endpoint flows, conditioned on $X_n$, give
\begin{equation}
\begin{aligned}
C_j-C_{j-1}
&=b\Pp(N=j,X_1=0)-c\Pp(N=j,X_1=1)\\
&\quad +(1-\gamma)
\{\Pp(N=j,X_1=0,X_n=1)
-\Pp(N=j,X_1=1,X_n=0)\}.
\end{aligned}
\label{app:binary:eq:7-8}
\end{equation}
The expression in braces vanishes by reversal of the stationary path.
Since $b=\gamma\pi_1$ and $c=\gamma\pi_0$, summing \eqref{app:binary:eq:7-8} from $0$ through
$j$ yields
\begin{equation}
C_j=\gamma\E[(\pi_1-X_1)\boldsymbol 1_{\{N\le j\}}].
\label{app:binary:eq:7-9}
\end{equation}
Equation \eqref{app:binary:eq:7-7} follows from \eqref{app:binary:eq:7-1}.

It remains to show $H_j\ge0$.  Introduce
\begin{equation}
M(z)=
\begin{pmatrix}
u&(1-u)z\\
1-v&vz
\end{pmatrix},
\qquad \mathbf e=\binom11,
\label{app:binary:eq:7-10}
\end{equation}
and the polynomials
\begin{equation}
\begin{aligned}
a_n(z)&=\gamma\E z^N
=\bigl(1-v,(1-u)z\bigr)M(z)^{n-1}\mathbf e,\\
r_n(z)&=\gamma\E[X_1z^N]
=\bigl(0,(1-u)z\bigr)M(z)^{n-1}\mathbf e,\\
h_n(z)&=\frac zn a_n'(z)-r_n(z).
\end{aligned}
\label{app:binary:eq:7-11}
\end{equation}
Because $h_n(1)=0$, the cumulative polynomial is
\begin{equation}
\mathcal H_n(z):=\gamma\sum_{j=0}^{n-1}H_jz^j
=\frac{h_n(z)}{1-z}.
\label{app:binary:eq:7-12}
\end{equation}

The trace and determinant of $M(z)$ are $u+vz$ and $-sz$.  Put
\begin{equation}
T=u+vz,\qquad Q=sz,
\qquad D(t,z)=1-Tt-Qt^2.
\label{app:binary:eq:7-13}
\end{equation}
Cayley--Hamilton and the first two terms of \eqref{app:binary:eq:7-11} give the formal series
\begin{equation}
\begin{aligned}
\mathsf A(t,z)&:=\sum_{m\ge1}a_m(z)t^{m-1}
=\frac{(1-v)+(1-u)z+\gamma szt}{D(t,z)},\\
\mathsf R(t,z)&:=\sum_{m\ge1}r_m(z)t^{m-1}
=\frac{(1-u)z(1+st)}{D(t,z)}.
\end{aligned}
\label{app:binary:eq:7-14}
\end{equation}
A direct differentiation of these two rational functions yields
\begin{equation}
z\partial_z\mathsf A(t,z)-\partial_t\{t\mathsf R(t,z)\}
=\sum_{m\ge1}m h_m(z)t^{m-1}
=\frac{(1-u)(1-v)sz(1-z)t^2}{D(t,z)^2}.
\label{app:binary:eq:7-15}
\end{equation}
Dividing by $1-z$ and extracting coefficients gives
\begin{equation}
\mathcal H_2(z)=0,
\qquad
\mathcal H_m(z)
=\frac{(1-u)(1-v)sz}{m}
[t^{m-3}]D(t,z)^{-2},
\quad m\ge3.
\label{app:binary:eq:7-16}
\end{equation}
Because $T$ and $Q$ have nonnegative coefficients, so does
$D^{-2}=(1-Tt-Qt^2)^{-2}$.  Thus every coefficient of
$\mathcal H_m$ is nonnegative, and $H_j\ge0$.  Equations \eqref{app:binary:eq:7-7} and \eqref{app:binary:eq:7-3}
follow.

Finally, Lemma~\ref{app:binary:lem:stein-poincare} and \eqref{app:binary:eq:7-3} imply
\[
\cE_n(f,f)
\ge\frac\gamma n\sum_{j=0}^{n-1}A_j\{f(j+1)-f(j)\}^2
\ge\frac\gamma n\Var_{\nu_n}(f),
\]
which proves \eqref{app:binary:eq:7-4}.
\end{proof}

\subsection{Boundary kernels and completion of the proof}
\label{app:binary:completion}

\begin{proof}[Proof of Theorem~\ref{app:binary:main}]
For $0<b,c<1$, Theorem~\ref{app:binary:thm:persistent-block} proves
\eqref{app:binary:regimes} on $0<\gamma\le1$, and
Theorem~\ref{app:binary:thm:anti-stein} proves it on
$1\le\gamma<2$. Both imply the first inequality in
\eqref{app:binary:sharp}.

Now let $0<b,c\le1$, allowing a zero diagonal entry, and keep $n$
fixed. With its stationary probability $\pi$, put
\[
 P_\varepsilon=(1-\varepsilon)P+\varepsilon\Pi_\pi,
 \qquad 0<\varepsilon<1.
\]
This kernel is strictly positive and has the same stationary law.
Its right gap is
$\gamma_\varepsilon=(1-\varepsilon)\gamma+\varepsilon$.
Given any function on the actual limiting count support, extend it
arbitrarily to $\{0,\ldots,n\}$. Its stationary variance and adjacent-window
Dirichlet form are finite sums of products of transition probabilities,
so both converge as $\varepsilon\downarrow0$. Apply the interior
Poincar\'e inequality to this one fixed extension and pass to the limit.
If the limiting variance is positive this gives the claimed gap bound;
if it is zero the variance inequality is vacuous. This argument does
not assert continuity of the entire count spectrum across a change
of support.

For completeness, the only periodic irreducible two-state kernel is
$b=c=1$. For even $n$, every window has exactly $n/2$ ones and its
count support is a singleton, assigned gap one in this paper. Then
$1\ge2/n\ge2/(4n)$. For odd $n$, the two possible counts alternate
and the actual count kernel has right gap two, again at least $2/n$.
Thus the periodic case is also consistent with both displayed lower
bounds, including the singleton convention.

Finally Theorem~\ref{thm:exact-cardinality-upper}, applied at $m=2$,
provides, for every $\delta>0$, an entrywise positive two-state kernel
and a finite $n\ge2$ with normalized count gap at most
$q_0+\delta=1/4+\delta$. Hence $c_2^\star\le1/4$ over the same
class used in its definition. The lower bound just established gives
the reverse inequality and proves $c_2^\star=1/4$.
\end{proof}

\begin{remark}[Scope of the binary argument]
The first inequality of \eqref{app:binary:regimes} is a stronger
pointwise statement, but the equality $c_2^\star=1/4$ is an infimum
over kernels and windows; it does not say that any fixed interior
kernel attains that value. In particular, on the iid line
$\gamma=1$, the count kernel is the Ehrenfest kernel and its gap is
exactly $1/n$.
The annealed block calculation above uses four binary values. It
supplies no multicolor analogue with the same unit constant and,
by itself, gives no conclusion $c_m^\star=q_{m-2}$ for $m\ge3$.
\end{remark}

\end{document}